\documentclass[12pt,a4paper]{article}

\usepackage{graphicx}
\usepackage{subfig}
\usepackage{amsmath,amsthm,amssymb}
\usepackage{esint}
\usepackage{mathrsfs}

\usepackage[T1]{fontenc}

\usepackage{hyperref}

\usepackage[left=2.2cm,right=2.2cm,top=3cm,bottom=3.5cm]{geometry}

\usepackage{color}
\usepackage{array}

\theoremstyle{plain}
\newtheorem{thm}{Theorem}[section]
\newtheorem{prop}[thm]{Proposition}
\newtheorem{lemma}[thm]{Lemma}
\newtheorem{cor}[thm]{Corollary}
\theoremstyle{definition}
\newtheorem{defn}[thm]{Definition}

\theoremstyle{remark}
\newtheorem{rem}[thm]{Remark}

\numberwithin{equation}{section}

\usepackage{csquotes}
\usepackage[backend=bibtex8, style=numeric-comp, maxnames=5, sorting=nyt]{biblatex}
\usepackage{enumitem}
\renewcommand\labelenumi{(\alph{enumi})}
\renewcommand\theenumi\labelenumi

\newcommand{\R}{\mathbb{R}} 
\newcommand{\N}{\mathbb{N}}
\newcommand{\Z}{\mathbb{Z}}

\newcommand{\Grad}{\nabla}  
\newcommand{\Div}{{\rm div}\,} 
\newcommand{\Divtx}{{\rm div}_{t,\vx}\,} 
\newcommand{\parthree}[3]{\partial_{#1#2#3}^3}

\newcommand{\dx}{\,{\rm d}\vx}
\newcommand{\dt}{\,{\rm d}t}  
\newcommand{\dr}{\,{\rm d}r} 
\newcommand{\dS}{\,{\rm d}S}

\renewcommand{\vec}[1]{{\bf #1}}
\newcommand{\vx}{\vec{x}}
\newcommand{\vz}{\vec{z}} 
\newcommand{\vzero}{\vec{0}}
\newcommand{\vn}{\vec{n}}
\newcommand{\va}{\vec{a}}
\newcommand{\vb}{\vec{b}} 
\newcommand{\vf}{\vec{f}}
\newcommand{\vp}{\vec{p}}
\newcommand{\vq}{\vec{q}}
\newcommand{\vs}{\vec{s}}
\newcommand{\veta}{\boldsymbol{\eta}}
\newcommand{\valpha}{\boldsymbol{\alpha}}
\newcommand{\vsigma}{\boldsymbol{\sigma}}
\newcommand{\vF}{\vec{F}}
\newcommand{\vu}{\vec{u}} 
\newcommand{\vv}{\vec{v}}
\newcommand{\vm}{\vec{m}} 
\newcommand{\vphi}{\boldsymbol{\varphi}} 

\newcommand{\mU}{\mathbb{U}} 
\newcommand{\id}{\mathbb{I}}
\newcommand{\mA}{\mathbb{A}}
\newcommand{\mM}{\mathbb{M}}
\newcommand{\mzero}{\mathbb{O}} 

\newcommand{\sym}[1]{{\rm Sym}(#1)}
\newcommand{\symz}[1]{{\rm Sym}_0(#1)}
\newcommand{\sz}{\symz{2}}

\newcommand{\tr}{{\rm tr}\,} 
\newcommand{\trans}{\top} 

\newcommand{\closure}[1]{\ov{#1}} 
\newcommand{\interior}[1]{{#1}^\circ} 

\renewcommand{\rho}{\varrho}
\newcommand{\half}{\tfrac{1}{2}}  
\newcommand{\supp}{{\rm supp}} 
\newcommand{\loc}{{\rm loc}} 
\newcommand{\glob}{{\rm glob}} 
\newcommand{\new}{{\rm new}} 
\newcommand{\Cc}{C^\infty_{\rm c}} 
\newcommand{\Cb}{C_{\rm b}} 

\newcommand{\ov}[1]{\overline{#1}}  
\newcommand{\ep}{\varepsilon} 
\newcommand{\ext}{{\rm ext}} 
\newcommand{\co}{{\rm co}} 
\newcommand{\pert}{{\rm pert}}
\newcommand{\dist}{{\rm dist}}
\let\subsetnotused\subset 
\renewcommand{\subset}{\subseteq}
\renewcommand{\supset}{\supseteq}
\newcommand{\sU}{\mathcal{U}}
\newcommand{\sV}{\mathcal{V}}
\newcommand{\sW}{\mathcal{W}}
\newcommand{\sK}{\mathcal{K}}
\newcommand{\sB}{\mathcal{B}}
\newcommand{\phase}{\mathcal{PH}}
\renewcommand{\setminus}{\smallsetminus}
\newcommand{\opL}{\mathscr{L}}
\newcommand{\subsetcomp}{\subsetnotused\joinrel\subsetnotused}
\newcommand{\ts}{{\rm s}}
\newcommand{\tf}{{\rm f}}
\renewcommand{\emptyset}{\varnothing}
\newcommand{\name}[1]{\textsc{#1}} 
\newcommand{\I}{I}

\begin{document} 

\title{A New Convex Integration Approach for the Compressible Euler Equations and Failure of the Local Maximal Dissipation Criterion}

\author{Simon Markfelder\footnote{Email: simon.markfelder@uni-konstanz.de}} 

\date{October 25, 2024}

\maketitle

\bigskip

\centerline{University of Konstanz, Department of Mathematics and Statistics,} 

\centerline{78457 Konstanz, Germany}

\bigskip

\begin{abstract} 
	In this paper we establish a new convex integration approach for the barotropic compressible Euler equations in two space dimensions. In contrast to existing literature, our new method generates not only the momentum for given density, but also the energy and the energy flux. This allows for a simple way to construct admissible solutions, i.e. solutions which satisfy the energy inequality. 
	
	Moreover using the convex integration method developed in this paper, we show that the local maximal dissipation criterion fails in the following sense: There exist wild solutions which beat the self-similar solution of the one-dimensional Riemann problem extended to two dimensions. Hence the local maximal dissipation criterion rules out the self-similar solution.
	
	The convex integration machinery itself is carried out in a very general way. Hence this paper provides a universal framework for convex integration which not even specifies the form of the partial differential equations under consideration. Therefore this general framework is applicable in many different situations. 
\end{abstract}

\bigskip

\noindent\textbf{Keywords:} Convex Integration, Compressible Euler Equations, Barotropic Euler Equations, Maximal Dissipation Criterion, Entropy Rate Admissibility Criterion, Non-Uniqueness, Weak Solutions 

\bigskip

\noindent\textbf{MSC (2020) codes:} 35F50 (primary), 35A02, 76N10, 35Q31 (secondary)

\bigskip

\tableofcontents

\section{Introduction} \label{sec:intro} 

\subsection{The Barotropic Compressible Euler System and Maximal Dissipation} \label{subsec:intro-euler}

One aim of this paper is to consider the \emph{barotropic compressible Euler system}
\begin{align}
	\partial_t \rho + \Div \vm &= 0, \label{eq:euler-d} \\
	\partial_t \vm + \Div \left(\frac{\vm \otimes \vm}{\rho}\right) + \Grad p(\rho) &= \vzero, \label{eq:euler-m}
\end{align}
with unknown density $\rho$ and momentum $\vm$, both of which are functions of time $t$ and position in space $\vx$. We look at the two-dimensional case and search for global-in-time solutions on the whole space, i.e. $t\in [0,\infty)$, $\vx\in \R^2$. The solutions considered in this paper will not contain vacuum, i.e. $\rho(t,\vx)>0$ for any $(t,\vx)\in [0,\infty)\times \R^2$. This way the term $\frac{\vm \otimes \vm}{\rho}$ does not cause any troubles. 

\begin{rem}
	We decided to write the Euler equations in terms of momentum $\vm$ rather than velocity $\vu$. However due to the fact that we exclude vacuum, it is straightforward to switch between the momentum formulation and the velocity formulation via the identity $\vu=\frac{\vm}{\rho}$. 
\end{rem}

The pressure $p$ which appears in \eqref{eq:euler-m} is a given function of $\rho$, and we suppose $p\in C^1(\R^+_0;\R^+_0)$. We want to solve the initial value problem, whose solutions $(\rho,\vm)$ shall satisfy the initial condition 
\begin{equation} \label{eq:init}
	(\rho,\vm)(0,\cdot) = (\rho_0,\vm_0),
\end{equation}
with given functions $\rho_0\in L^\infty(\R^2)$ and $\vm_0\in L^\infty(\R^2;\R^2)$, in addition to the partial differential equations (PDEs) \eqref{eq:euler-d}, \eqref{eq:euler-m}.

\subsubsection{Admissible Weak Solutions} \label{subsubsec:intro-euler-adm}

Since it is well-known that smooth solutions to problem \eqref{eq:euler-d}, \eqref{eq:euler-m}, \eqref{eq:init} do not exist globally in time, no matter how smooth the initial data are, we shall deal with weak solutions.

\begin{defn} \label{defn:weak-sol} 
	A pair $(\rho,\vm) \in L^\infty((0,\infty) \times \R^2; \R^+ \times \R^2)$ is a \emph{weak solution} of the initial value problem \eqref{eq:euler-d}, \eqref{eq:euler-m}, \eqref{eq:init} if the following equations are satisfied for all test functions $(\phi,\vphi) \in \Cc([0,\infty) \times \R^2; \R\times \R^2)$:
	\begin{align}
		\int_0^\infty \int_{\R^2} \Big[\rho \partial_t \phi + \vm\cdot\Grad \phi\Big]\dx\dt + \int_{\R^2} \rho_0\phi(0,\cdot) \dx &= 0 ; \label{eq:euler-weak-d} \\
		\int_0^\infty \int_{\R^2} \left[\vm \cdot\partial_t \vphi + \frac{\vm\otimes\vm}{\rho}:\Grad \vphi + p(\rho)\Div \vphi\right]\dx\dt + \int_{\R^2} \vm_0\cdot\vphi(0,\cdot) \dx &= 0. \label{eq:euler-weak-m}
	\end{align}
\end{defn}

It is also well-known that weak solutions are highly non-unique and that many of them are not physically relevant in the sense that their energy increases. We exclude those unphysical solutions by imposing the energy inequality as an admissibility criterion. Here the energy inequality reads 
\begin{equation} \label{eq:euler-E}
	\partial_t \left( \frac{|\vm|^2}{2\rho} + P(\rho) \right) + \Div\left[\left(\frac{|\vm|^2}{2\rho} + P(\rho) + p(\rho) \right) \frac{\vm}{\rho} \right] \leq 0 ,
\end{equation}
where the \emph{pressure potential} $P$ is given by 
\begin{equation} \label{eq:pressure-potential}
	P(\rho) = \rho \int_{\rho^\ast}^{\rho} \frac{p(r)}{r^2}\dr,
\end{equation}
and the number $\rho^\ast>0$ is arbitrary. 

Admissible weak solutions are weak solutions which satisfy the energy inequality \eqref{eq:euler-E} in the weak form: 

\begin{defn} \label{defn:adm-sol} 
	A weak solution $(\rho,\vm)$ is called \emph{admissible} if
	\begin{align} 
		\int_0^\infty \int_{\R^2} \left[\bigg(\frac{|\vm|^2}{2\rho} + P(\rho)\bigg) \partial_t \psi + \bigg(\frac{|\vm|^2}{2\rho} + P(\rho) + p(\rho)\bigg)\frac{\vm}{\rho}\cdot\Grad \psi \right]\dx\dt \qquad & \notag \\ 
		+ \int_{\R^2} \bigg(\frac{|\vm_0|^2}{2\rho_0} + P(\rho_0)\bigg)\psi(0,\cdot) \dx &\geq 0 \label{eq:euler-weak-E}
	\end{align}
	for all $\psi \in \Cc([0,\infty) \times \R^2;\R^+_0)$.
\end{defn}

\subsubsection{Global Maximal Dissipation Criterion} \label{subsubsec:intro-euler-global-max-diss}

Using the technique of convex integration as adopted from \name{Gromov} (see e.g. \cite{Gromov}) by \name{De~Lellis}-\name{Sz{\'e}kelyhidi}~\cite{DelSze09,DelSze10} one is able to show that there exist infinitely many admissible weak solutions\footnote{As it is a common practice in the literature, we will sometimes denote the solutions generated by convex integration as \emph{wild solutions}.} to \eqref{eq:euler-d}, \eqref{eq:euler-m}, \eqref{eq:init} for some initial data, see \cite{DelSze10,Chiodaroli14,Feireisl14,AkrWie21,ChiDelKre15,ChiKre14,KliMar18_1,ChiKre18,KliMar18_2,BreChiKre18,CKMS21} among others. Due to this lack of uniqueness and because some of the solutions constructed by convex integration seem to by unphysical, one studied additional criteria. In particular the \emph{global maximal dissipation criterion}, which was proposed by \name{Dafermos}~\cite{Dafermos73} under the name \emph{entropy rate admissibility criterion}, was considered by \name{Feireisl}~\cite{Feireisl14} and also by \name{Chiodaroli}-\name{Kreml}~\cite{ChiKre14}.

\begin{defn} \label{defn:global-max-diss}
	\begin{itemize}
		\item For a weak solution $(\rho,\vm) \in L^\infty((0,\infty) \times \Omega; \R^+ \times \R^2)$ of the Euler system \eqref{eq:euler-d}, \eqref{eq:euler-m} we define the \emph{energy dissipation rate} as the right derivative of the total energy with respect to time\footnote{We suppose that the solution $(\rho,\vm)$ has enough regularity to ensure that this right derivative exists.} 
		, i.e.
		\begin{equation*} 
			D_\glob[\rho,\vm](t) := \frac{{\rm d}_+}{\dt} \int_\Omega \left( \frac{|\vm|^2}{2\rho} + P(\rho) \right) \dx .
		\end{equation*}
		Here the domain $\Omega\subset\R^2$ is supposed to be bounded\footnote{As mentioned above, in this paper we are going to consider the domain to be the whole space, i.e. $\Omega=\R^2$. In this case the total energy does not need to be finite and consequently the energy dissipation rate is no longer well-defined. However this problem can be overcome, see Remark~\ref{rem:finite-energy}.} 
		which makes the total energy 
		$$
			\int_\Omega \left( \frac{|\vm|^2}{2\rho} + P(\rho) \right) \dx
		$$ 
		finite. 
		
		\item A weak solution $(\rho,\vm)$ of \eqref{eq:euler-d}, \eqref{eq:euler-m}, \eqref{eq:init} satisfies the \emph{global maximal dissipation criterion} if there is no other weak solution $(\overline{\rho},\overline{\vm})$ with the property that for some time $\tau\in [0,\infty)$, $(\rho,\vm)=(\overline{\rho},\overline{\vm})$ on $[0,\tau]\times \Omega$ and $D_\glob[\overline{\rho},\overline{\vm}](\tau)<D_\glob[\rho,\vm](\tau)$. 
	\end{itemize}
\end{defn}

So simply speaking, for a solution which satisfies the global maximal dissipation criterion, the total energy decreases with maximal rate.

\name{Feireisl}~\cite{Feireisl14} showed that the admissible weak solutions of problem \eqref{eq:euler-d}, \eqref{eq:euler-m}, \eqref{eq:init} which are generated by convex integration do not comply with the global maximal dissipation criterion. In particular for any such solution, one can construct another admissible weak solution, which beats the first one in the sense described in Defn.~\ref{defn:global-max-diss}, i.e. which has a smaller dissipation rate. \name{Feireisl}'s result hints that the global maximal dissipation criterion might be able to restore uniqueness of solutions. 

On the other hand \name{Chiodaroli}-\name{Kreml}~\cite{ChiKre14} gave a particular example where the global maximal dissipation criterion discards the solution which intuitively seems to be the physical one. In particular in \cite{ChiKre14} initial data of the form\footnote{We call such data \emph{Riemann data}. Moreover we use the notation $\vx=(x,y)^\trans$.} 
\begin{equation} \label{eq:riemann-init}
	(\rho_0,\vm_0)(\vx) := \left\{ \begin{array}{cc} (\rho_-,\vm_-) & \text{ if } y<0, \\ (\rho_+,\vm_+) & \text{ if } y>0, \end{array}\right.
\end{equation}
with constant $\rho_\pm\in \R^+$, $\vm_\pm\in \R^2$, are considered, together with the pressure law 
\begin{equation} \label{eq:pressure}
	p(\rho) = \rho^\gamma, \qquad \text{ where } \gamma \geq 1 .
\end{equation}
Using classical methods, one is able to solve the initial value problem for the Euler equations \eqref{eq:euler-d}, \eqref{eq:euler-m}, \eqref{eq:init} with initial data \eqref{eq:riemann-init} which leads to a one-dimensional self-similar solution. We refer to \cite[Sect.~2]{ChiKre14}, \cite[Prop.~1.3]{KliMar18_1} or \cite[Prop.~7.1.1]{Markfelder} for more details. One of the main results in \cite{ChiKre14} is the following theorem.

\begin{thm}[{\name{Chiodaroli}-\name{Kreml}~\cite[Thm.~2]{ChiKre14}}] \label{thm:chikre14}
	Let $p$ as in \eqref{eq:pressure} with $1\leq \gamma < 3$. There exist Riemann data (i.e. data of the form \eqref{eq:riemann-init}) for which the one-dimensional self-similar solution to \eqref{eq:euler-d}, \eqref{eq:euler-m}, \eqref{eq:init} does not satisfy the global maximal dissipation criterion.
\end{thm} 

\begin{rem} \label{rem:finite-energy}
	Note that in the setting described above (i.e. $\Omega=\R^2$), the total energy does not need to be finite and hence the definition of the global maximal dissipation criterion (Defn.~\ref{defn:global-max-diss}) does not make sense any more. This can be overcome by restricting to a square $(-L,L)^2\subset \R^2$ when measuring the total energy and the energy dissipation rate. We refer to \cite[Sect.~1.1]{ChiKre14} for more details. See also \cite[Sect.~6]{ChiKre14} for an alternative approach. 
\end{rem}

The proof of Thm.~\ref{thm:chikre14} is built upon ideas that were orginally provided in \cite{ChiDelKre15}. One applies convex integration in a ``wedge'' in space-time which leads to the notion of a \emph{fan subsolution}, see \cite[Sect.~3]{ChiDelKre15}. As shown in \cite[Prop.~3.6]{ChiDelKre15} the existence of a fan subsolution for given data of the form \eqref{eq:riemann-init} implies existence of infinitely many admissible weak solutions. In \cite{ChiKre14} one observed that the energy dissipation rate of these solutions only depends on the fan subsolution. Then the proof of Thm.~\ref{thm:chikre14} results from proving that there exists a fan subsolution which guarantees that the energy dissipation rate of the wild solutions beats the energy dissipation rate of the one-dimensional self-similar solution. 

As a consequence of Thm.~\ref{thm:chikre14}, if we expect the one-dimensional self-similar solution to problem \eqref{eq:euler-d}, \eqref{eq:euler-m}, \eqref{eq:init} with initial data \eqref{eq:riemann-init} to be the physical relevant one, we have to give up the global maximal dissipation criterion. 

\subsubsection{Local Maximal Dissipation Criterion} \label{subsubsec:intro-euler-local-max-diss}

In this paper we will prove a result that is analoguous to Thm.~\ref{thm:chikre14}, with the difference that we study the \emph{local} rather than the global maximal dissipation criterion. 

\begin{defn} \label{defn:local-max-diss}
	\begin{itemize}
		\item For an admissible weak solution $(\rho,\vm) \in L^\infty((0,\infty) \times \R^2; \R^+ \times \R^2)$ of the Euler equations \eqref{eq:euler-d}, \eqref{eq:euler-m} we define the \emph{energy dissipation measure} as the left-hand side of \eqref{eq:euler-E}, i.e.
		\begin{equation*}
			D_\loc[\rho,\vm]:=\partial_t \left( \frac{|\vm|^2}{2\rho} + P(\rho) \right) + \Div\left[\left(\frac{|\vm|^2}{2\rho} + P(\rho) + p(\rho) \right) \frac{\vm}{\rho} \right]. 
		\end{equation*}
		
		\item An admissible weak solution $(\rho,\vm)$ of \eqref{eq:euler-d}, \eqref{eq:euler-m}, \eqref{eq:init} complies with the \emph{local maximal dissipation criterion} if there is no other admissible weak solution $(\overline{\rho},\overline{\vm})$ emanating from the same initial data with the property
		\begin{equation} \label{eq:D-ineq}
			D_\loc[\overline{\rho},\overline{\vm}]< D_\loc[\rho,\vm].
		\end{equation}
	\end{itemize} 
\end{defn}

\begin{rem} \label{rem:understanding-of-measure-inequality} 
	Inequality \eqref{eq:D-ineq} has to be understood in the weak sense. In particular \eqref{eq:D-ineq} means that 
	\begin{align} 
		&\int_0^\infty \int_{\R^2} \left[\bigg(\frac{|\ov{\vm}|^2}{2\ov{\rho}} + P(\ov{\rho})\bigg) \partial_t \psi + \bigg(\frac{|\ov{\vm}|^2}{2\ov{\rho}} + P(\ov{\rho}) + p(\ov{\rho})\bigg)\frac{\ov{\vm}}{\ov{\rho}}\cdot\Grad \psi \right]\dx\dt \notag \\
		&\geq \int_0^\infty \int_{\R^2} \left[\bigg(\frac{|\vm|^2}{2\rho} + P(\rho)\bigg) \partial_t \psi + \bigg(\frac{|\vm|^2}{2\rho} + P(\rho) + p(\rho)\bigg)\frac{\vm}{\rho}\cdot\Grad \psi \right]\dx\dt . \label{eq:D-ineq-weak}
	\end{align}
	for all test functions $\psi\in \Cc((0,\infty)\times \R^2;\R^+_0)$, and there exists a test function for which the inequality \eqref{eq:D-ineq-weak} holds strictly.
\end{rem}

\begin{rem} \label{rem:total-order}
	We would like to remark that the order defined in \eqref{eq:D-ineq} is not a total order. In particular there might be two admissible weak solutions $(\rho_1,\vm_1)$, $(\rho_2,\vm_2)$ which are not comparable, i.e. with 
	$$
		D_\loc[\rho_1,\vm_1]\not\leq D_\loc[\rho_2,\vm_2] \qquad \text{ and } \qquad D_\loc[\rho_1,\vm_1]\not\geq D_\loc[\rho_2,\vm_2].
	$$
\end{rem}

We will prove the following, which is one of our main theorems.

\begin{thm} \label{thm:local-max-diss}
	Let $p$ as in \eqref{eq:pressure} with $\gamma=2$. There exist Riemann data (i.e. data of the form \eqref{eq:riemann-init}) for which the one-dimensional self-similar solution to \eqref{eq:euler-d}, \eqref{eq:euler-m}, \eqref{eq:init} does not satisfy the local maximal dissipation criterion.
\end{thm}

Similarly to \name{Chiodaroli}-\name{Kreml}'s proof of Thm.~\ref{thm:chikre14}, the strategy to show Thm.~\ref{thm:local-max-diss} is to construct a wild solution $(\ov{\rho},\ov{\vm})$ which beats the one-dimensional self-similar solution $(\rho,\vm)$ in the sense of Defn.~\ref{defn:local-max-diss}, i.e. \eqref{eq:D-ineq} is satisfied. However with the convex integration approach as in \cite{ChiDelKre15,ChiKre14} we did not manage to succeed. The problem is that the energy inquality which has to be satisfied by fan subsolutions as in \cite[Defn.~3.5]{ChiDelKre15} is an obstacle in constructing wild solutions which are comparable with the one-dimensional self-similar solution in the sense of the order defined in \eqref{eq:D-ineq}, cf. Remark~\ref{rem:total-order}. We overcome this with a new convex integration approach which also constructs the energy and the energy flux, see Sect.~\ref{subsubsec:intro-ci-with-energy} below for more details. 

Analogously to the global maximal dissipation criterion above, Thm.~\ref{thm:local-max-diss} has the following consequence. If we expect the one-dimensional self-similar solution to problem \eqref{eq:euler-d}, \eqref{eq:euler-m}, \eqref{eq:init} with initial data \eqref{eq:riemann-init} to be the physically relevant one, we have to discard the local maximal dissipation criterion.

\subsection{The Convex Integration Approach Developed in This Paper} \label{subsec:intro-ci} 

\subsubsection{Convex Integration for Compressible Euler with Energy Inequality} \label{subsubsec:intro-ci-with-energy} 

Most convex integration approaches used in the literature in the context of the compressible Euler equations \eqref{eq:euler-d}, \eqref{eq:euler-m} start with a reduction of the equations to some kind of incompressible system, e.g. in \cite{DelSze10,Chiodaroli14,Feireisl14,Feireisl16_1,AkrWie21,ChiDelKre15,ChiKre14,KliMar18_1,ChiKre18,KliMar18_2,BreChiKre18,CKMS21}. This way, in order to construct solutions to the latter system, one only has to modify the incompressible convex integration method by \name{De~Lellis}-\name{Szek{\'e}lyhidi}~\cite{DelSze09,DelSze10}. In other words the convex integration approach generates a momentum $\vm$ for a given density $\rho$ which remains untouched. Consequently we will call such convex integration approaches ``incompressible'' even if they are applied to a compressible model. \name{Feireisl}~\cite{Feireisl16_1} provided a general framework for such an incompressible convex integration approach, which works for many systems appearing in mathematical fluid mechanics, like the Euler-Fourier system, quantum fluids or binary mixtures of compressible fluids.

Note that a genuinely compressible convex integration approach does not lead to a larger hull, see e.g. \cite{Markfelder} or \cite{DebSkiWie23}. This means that the use of an incompressible method in the context of the compressible Euler equations \eqref{eq:euler-d}, \eqref{eq:euler-m} is not a severe restriction.

The common feature of most convex integration methods that are used in the literature so far in order to show non-uniqueness, is that one introduces the notion of a subsolution and proves that existence of a subsolution implies existence of infinitely many solutions. The latter is achieved by ``deforming'' the subsolution in order to obtain solutions. Admissibility, i.e. the validity of the energy inequality (see Defn.~\ref{defn:adm-sol}), is then achieved by the observation that energy and energy flux of the solutions can be estimated by an expression which only depends on the corresponding subsolution. This leads to an inequality on the subsolution which we also call \emph{energy inequality} even if it is not exactly the same as inequality \eqref{eq:euler-E}. Consequently the proof of non-uniquness of admissible weak solutions boils down to finding a suitable subsolution which fulfills the energy inequality. 

It seems however that the energy inequality is quite restrictive when trying to find a subsolution. This motivated us to design a convex integration approach which not only generates the momentum $\vm$ for a given density $\rho$, but also the energy $\frac{|\vm|^2}{2\rho} + P(\rho)$ and the energy flux $\left(\frac{|\vm|^2}{2\rho} + P(\rho) + p(\rho)\right)\frac{\vm}{\rho}$. For this reason we call our approach ``convex integration for the compressible Euler equations \emph{with} energy inequality''. Note again that also the methods in the existing literature yield admissible solutions, but the validity of the energy inequality is not achieved by convex integration but by a suitable choice of the subsolution.

In our method we also get a restriction on the subsolution which emanates from the energy inequality. Note however that this constraint is less restrictive than the aforementioned energy inequality for subsolutions. This is the key point which allows to prove Thm.~\ref{thm:local-max-diss}. Moreover we would like to point out that our approach is incompressible in the sense explained above, i.e. we do not add oscillations in $\rho$. Like in the case of convex integration approaches which do not produce energy and energy flux, the use of an incompressible method is not a restriction, see Remark~\ref{rem:incompressible-approach} below.

Finally we would like to mention that a similar approach (i.e. also energy and energy flux are constructed by convex integration) has already been developed in the context of the incompressible Euler equations by \name{Gebhard}-\name{Kolumb{\'a}n}~\cite{GebKol22}.

\subsubsection{General Framework}\label{subsubsec:intro-ci-gf} 

The convex integration machinery follows more or less the same lines as soon as the $\Lambda$-convex hull (see Defn.~\ref{defn:app-Lconvex}) of the so-called constitutive set $\sK$, see Sect.~\ref{subsubsec:gen-prelim-pde+K} below, coincides with its convex hull. For this reason we have decided to write up a general convex integration framework. This framework is not specified to the form of the PDEs, which was the case in \name{Feireisl}'s framework \cite{Feireisl16_1}, see Sect.~\ref{subsubsec:intro-ci-with-energy} above. Moreover our general framework includes what is done in many papers on convex integration in fluid mechanics, e.g. \cite{DelSze09,Chiodaroli14,ChiDelKre15,ChiKre14,KliMar18_1,ChiKre18,KliMar18_2,BreChiKre18,Markfelder}. We are convinced that it is also applicable in other contexts.


\subsection{Outline of This Paper} \label{subsec:intro-outline}

The paper is organized as follows. In Sect.~\ref{sec:general} we present our general convex integration framework. Subsequently in Sect.~\ref{sec:euler-in-framework} we show that the Euler equations \eqref{eq:euler-d}, \eqref{eq:euler-m} with energy inequality \eqref{eq:euler-E} fit into the framework established in Sect.~\ref{sec:general}. Moreover we identify a subset $\sW_{\rho,Q}$ of $\sU_{\rho,Q}$ which will be needed in the proof of Thm.~\ref{thm:local-max-diss}. Finally in Sect.~\ref{sec:riemann-ci} we deal with Riemann initial data and prove Thm.~\ref{thm:local-max-diss}.

\section{A General Framework for Convex Integration} \label{sec:general} 

When we carry out convex integration, we follow \cite{Markfelder}. As mentioned in Sect.~\ref{subsubsec:intro-ci-gf} our general framework contains many works on convex integration in fluid mechanics as special cases, e.g. \cite{DelSze09,Chiodaroli14,ChiDelKre15,ChiKre14,KliMar18_1,ChiKre18,KliMar18_2,BreChiKre18,Markfelder}. For this reason, analogue versions of most statements in the current section can be found in these papers. To ease reading, especially for the unexperienced reader, we refer to \cite[Sect.~4.1]{Markfelder} for a sketch of the essential ideas behind convex integration. 

\subsection{Preliminaries} \label{subsec:general-prelim}

\subsubsection{The Linear PDE-System and the Family of Constitutive Sets $(\sK_\vb)_{\vb\in \sB}$} \label{subsubsec:gen-prelim-pde+K}

We begin by introducing some notation. Let $\Gamma\subset\R^n$ be a Lipschitz domain (not necessarily bounded), where $n\in \N$ is the space dimension. The unknown is a vector-valued function $\vz: \Gamma \to \R^M$, where $M\in \N$ is the ``number of scalar unknowns''. Let furthermore $\mA: \R^M \to \R^{m\times n}$ a linear map, where $m\in \N$ is the ``number of scalar inequalities''. Finally let\footnote{As commonly used in the literature, $\Cb$ denotes the set of all continuous and bounded functions.} $\vb\in \Cb(\closure{\Gamma};\R^m)$ and set $\sB:= \closure{\vb(\closure{\Gamma})}$. Note that $\sB$ is a compact subset of $\R^m$. 

We study a general linear, first order and not necessarily homogeneous system of partial differential inequalities
\begin{equation} \label{eq:lin-eq}
	\Div \mA (\vz(\vx)) \leq \vb(\vx).
\end{equation}
Here the divergence is meant row-wise, i.e.\footnote{We use the following notation. The $k$-th component of a vector $\vv$ is denoted by $v_k$ or $[\vv]_k$. The $k\ell$-entry of a matrix $\mM$ is denoted by $[\mM]_{k\ell}$.} 
$$
	[\Div \mA (\vz)]_i = \sum_{j=1}^n \partial_j [\mA(\vz)]_{ij},
$$
where $\partial_j:= \frac{\partial}{\partial x_j}$, and moreover the order $\leq$ for vectors is meant component-wise, i.e. for any $\va,\vb\in \R^M$ we have
\begin{equation} \label{eq:vector-order}
	\va \leq \vb \quad \Leftrightarrow \quad a_i \leq b_i \text{ for all }i=1,...,M . 
\end{equation}

Note that the linearity of $\vz\mapsto \mA(\vz)$ implies 
$$
	[\Div \mA (\vz)]_i = \sum_{j=1}^n \partial_j [\mA(\vz)]_{ij} = \sum_{j=1}^n [\mA(\partial_j\vz)]_{ij} ,
$$
for all $i=1,...,m$.

\begin{rem} \label{rem:equation-vs-inequality}
	The reader should notice that any equation can be written equivalently as two inequalities. Hence the system of partial differential equations 
	$$
		\Div \mA (\vz) = \vb
	$$
	can be equivalently written in the form \eqref{eq:lin-eq} by setting $m^\new=2m$, 
	$$
		[\vb^\new]_i = \left\{ \begin{array}{ll} \hphantom{-} [\vb]_i & \text{ if } i\in\{1,...,m\}, \\ -[\vb]_{i-m} & \text{ if } i\in\{m+1,...,2m\}, \end{array} \right.
	$$
	and 
	$$
		[\mA^\new(\vz)]_{ij} = \left\{ \begin{array}{ll} \hphantom{-} [\mA(\vz)]_{ij} & \text{ if } i\in\{1,...,m\}, \\ -[\mA(\vz)]_{i-m,j} & \text{ if } i\in\{m+1,...,2m\}. \end{array} \right. 
	$$
\end{rem}

In addition to system \eqref{eq:lin-eq} we consider for each\footnote{Here we overload notation and use the letter $\vb$ both for the given function $\vb\in \Cb(\closure{\Gamma};\R^m)$ and for an element in $\sB$.} $\vb\in \sB$ the constitutive set $\sK_{\vb}\subset \R^M$. Our goal is to find weak solutions $\vz\in L^\infty(\Gamma;\R^M)$ of the linear system \eqref{eq:lin-eq} which satisfy 
\begin{equation} \label{eq:UinK}
	\vz(\vx)\in \sK_{\vb(\vx)}\qquad \text{ for a.e. }\vx\in \Gamma. 
\end{equation}

\begin{rem} \label{rem:tartars-framework}
	The usual application considers a non-linear system of PDEs (and possibly inequalities). Then one replaces all non-linearities by new unknowns in order to obtain a linear system of the form \eqref{eq:lin-eq}. In addition, this replacement defines the constitutive sets $\sK_\vb$. This procedure goes back to \name{Tartar}~\cite{Tartar79} and is therefore called \emph{Tartar's framework}. This way a solution of the linear system \eqref{eq:lin-eq} which satisfies \eqref{eq:UinK} is also a solution of the original non-linear system of PDEs. Examples of this procedure can be found in \cite[Sect.~2]{DelSze09}, \cite[Sect.~4.1.2]{Markfelder} and also in Sect.~\ref{subsec:eif-prelim} below. 
\end{rem}

We assume the family of sets $(\sK_\vb)_{\vb\in\sB}$ to be suitable in the following sense.

\begin{defn} \label{defn:suitable-K}
	We call the family of constitutive sets $(\sK_\vb)_{\vb\in\sB}$ \emph{suitable} if it satisfies the following three properties:
	\begin{enumerate}
		\item \label{item:suitable-K-compact} For all $\vb\in \sB$, $\sK_\vb$ is compact in $\R^m$.
		
		\item \label{item:suitable-K-continuity} The dependence of $\sK_{\vb}$ on $\vb$ is uniformly continuous in the following sense: For all $\ep>0$, there exists $\delta>0$ such that for all $\vb_1,\vb_2\in \sB$ with $|\vb_1-\vb_2|<\delta$ and all $\vz_1\in \sK_{\vb_1}$ there exists $\vz_2\in \sK_{\vb_2}$ with $|\vz_1-\vz_2|<\ep$. 
		
		\item \label{item:suitable-K-boundary} We have $\interior{\big((\sK_{\vb})^\co\big)} \cap \sK_{\vb} = \emptyset$ for all $\vb\in \sB$, where $(\sK_{\vb})^\co$ denotes the convex hull of $\sK_{\vb}$, see Appendix~\ref{app:convex}.
	\end{enumerate}
\end{defn}

We observe that suitability implies uniform boundedness:

\begin{lemma} \label{lemma:uniform-boundedness-K}
	Let $(\sK_\vb)_{\vb\in\sB}$ a suitable family of constitutive sets. Then there exists\footnote{In particular $c$ is does not depend on $\vb$.} $c>0$ such that 
	$$
		|\vz|\leq c \qquad \text{ for all }\vb\in \sB \text{ and all } \vz\in \sK_{\vb} .
	$$
\end{lemma}

\begin{proof}
We know from item \ref{item:suitable-K-compact} in Defn.~\ref{defn:suitable-K} that for each $\vb\in \sB$, $\sK_{\vb}$ is bounded by 
$$
	c(\vb) := \max_{\vz\in \sK_\vb} |\vz|,
$$
which defines a function $c:\sB \to \R^+_0$. Let us next show that this function is continuous. For $\ep>0$ given, there exists $\delta>0$ such that for all $\vb_1,\vb_2\in \sB$ with $|\vb_1-\vb_2|<\delta$ we obtain the following. Let $\vz_1\in \sK_{\vb_1}$ such that $c(\vb_1) = |\vz_1|$. Then item \ref{item:suitable-K-continuity} in Defn.~\ref{defn:suitable-K} guarantees existence of $\vz_2\in \sK_{\vb_2}$ with $|\vz_1-\vz_2|<\ep$. This yields
$$
	c(\vb_1) - c(\vb_2) = |\vz_1| - \max_{\vz\in \sK_{\vb_2}} |\vz| \leq |\vz_1| - |\vz_2| \leq |\vz_1 - \vz_2| < \ep. 
$$
Interchanging $\vb_1$ and $\vb_2$ we obtain analogously $c(\vb_2) - c(\vb_1) < \ep$, i.e. the map $c:\sB \to \R^+_0$ is indeed continuous. As $\sB$ is compact, the function $c:\sB \to \R^+_0$ is bounded which finishes the proof. 
\end{proof}

Moreover, a statement that is analogous to item \ref{item:suitable-K-continuity} of Defn.~\ref{defn:suitable-K} also holds for the convex hulls $(\sK_{\vb})^\co$ and their boundaries $\partial(\sK_{\vb})^\co$. More precisely we have the following:

\begin{lemma} \label{lemma:continuity-Kco}
	Let $(\sK_\vb)_{\vb\in\sB}$ a suitable family of constitutive sets. Then the following statements hold:
	\begin{itemize}
		\item Let $\ep>0$, and $\delta>0$ the $\delta$ from item \ref{item:suitable-K-continuity} of Defn.~\ref{defn:suitable-K} which corresponds to $\ep$. Let furthermore $\vb_1,\vb_2\in\sB$ with $|\vb_1-\vb_2|<\delta$ and $\vz_1\in(\sK_{\vb_1})^\co$. Then there exists $\vz_2\in(\sK_{\vb_2})^\co$ with $|\vz_1 - \vz_2| < \ep$. 
		
		\item Let $\ep>0$, and $\delta>0$ the $\delta$ from item \ref{item:suitable-K-continuity} of Defn.~\ref{defn:suitable-K} which corresponds to $\frac{\ep}{3}$. Let furthermore $\vb_1,\vb_2\in\sB$ with $|\vb_1-\vb_2|<\delta$ and $\vz_1\in\partial(\sK_{\vb_1})^\co$. Then there exists $\vz_2\in\partial (\sK_{\vb_2})^\co$ with $|\vz_1 - \vz_2| < \ep$. 
	\end{itemize}
\end{lemma}

\begin{proof}
\begin{itemize}
	\item As $\vz_1\in(\sK_{\vb_1})^\co$, there exist $N\in \N$ and $(\tau_i,\vz_{1,i})\in \R^+\times \sK_{\vb_1}$ for all $i=1,...,N$ with $\sum_{i=1}^N \tau_i = 1$ and $\sum_{i=1}^N \tau_i \vz_{1,i} = \vz_1$, see Prop.~\ref{prop:app-caratheodory}. Applying item \ref{item:suitable-K-continuity} of Defn.~\ref{defn:suitable-K}, for each $i=1,...,N$ there exists $\vz_{2,i}\in \sK_{\vb_2}$ such that $|\vz_{1,i} - \vz_{2,i}|<\ep$ for all $i=1,...,N$. Then $\vz_2 := \sum_{i=1}^N \tau_i \vz_{2,i} \in (\sK_{\vb_2})^\co$ satisfies 
	$$
		|\vz_1 - \vz_2| \leq \sum_{i=1}^N \tau_i |\vz_{1,i} - \vz_{2,i}| < \ep \sum_{i=1}^N \tau_i = \ep.
	$$
	
	\item From item \ref{item:suitable-K-compact} of Defn.~\ref{defn:suitable-K} we deduce that $(\sK_{\vb_1})^\co$ is compact as well. Thus $\vz_1\in\partial(\sK_{\vb_1})^\co \subset (\sK_{\vb_1})^\co$ and the first bullet point of Lemma~\ref{lemma:continuity-Kco} yields $\widetilde{\vz}_2 \in(\sK_{\vb_2})^\co$ with $|\vz_1 - \widetilde{\vz}_2|<\frac{\ep}{3}$. Set $\vz_2\in\partial (\sK_{\vb_2})^\co$ such that $\dist(\widetilde{\vz}_2, \partial (\sK_{\vb_2})^\co) = |\widetilde{\vz}_2 - \vz_2|$.
	
	Next choose $\widetilde{\vz}_1\in \R^M \setminus (\sK_{\vb_1})^\co$ such that $\dist(\widetilde{\vz}_1,(\sK_{\vb_1})^\co)= |\widetilde{\vz}_1 - \vz_1 |=\frac{\ep}{3}$.
	
	Assume that $\widetilde{\vz}_1\in (\sK_{\vb_2})^\co$, then according to the first bullet point, there exists $\widehat{\vz}_1\in (\sK_{\vb_1})^\co$ with $|\widetilde{\vz}_1 - \widehat{\vz}_1| < \frac{\ep}{3}$. Thus $\frac{\ep}{3} =\dist(\widetilde{\vz}_1,(\sK_{\vb_1})^\co) \leq |\widetilde{\vz}_1 - \widehat{\vz}_1| < \frac{\ep}{3}$, a contradiction, and hence $\widetilde{\vz}_1\in \R^M \setminus (\sK_{\vb_2})^\co$.
	
	Therefore we may estimate as follows 
	$$
		|\widetilde{\vz}_2 - \vz_2| = \dist(\widetilde{\vz}_2, \partial (\sK_{\vb_2})^\co) \leq |\widetilde{\vz}_2 - \widetilde{\vz}_1| \leq |\widetilde{\vz}_2 - \vz_1| + |\vz_1 - \widetilde{\vz}_1| < \frac{2}{3} \ep,
	$$
	and thus $|\vz_1 - \vz_2| \leq |\vz_1 -\widetilde{\vz}_2 | + |\widetilde{\vz}_2 - \vz_2| < \ep$. 
\end{itemize}
\end{proof}

\subsubsection{Plane Waves and the Wave Cone $\Lambda$} \label{subsubsec:gen-prelim-pl.waves+Lambda}

We will relax the constitutive sets $\sK_{\vb}$ to larger sets $\sU_{\vb}$, see Sect.~\ref{subsec:general-geometric} below for details. A solution $\vz$ of the linear system \eqref{eq:lin-eq}, which satisfies 
\begin{equation*}
	\vz(\vx)\in \sU_{\vb(\vx)}
\end{equation*}
for all $\vx$ is then called \emph{subsolution}. Our goal will be to add an oscillation to a given subsolution in order to obtain a new subsolution which is closer to being a solution. This oscillation will be built upon plane wave solutions of the linear system 
\begin{equation} \label{eq:lin-eq-homo}
	\Div \mA(\vz) = \vzero .
\end{equation}
Note in contrast to \eqref{eq:lin-eq}, that \eqref{eq:lin-eq-homo} is a system of equations rather than inequalities, and it is homogeneous. 

A plane wave solution $\widetilde{\vz}$ of \eqref{eq:lin-eq-homo} is a solution of the form 
\begin{equation} \label{eq:plane-wave}
	\widetilde{\vz}(\vx) = \vz \, h(\vx\cdot \veta),
\end{equation} 
where $\vz\in \R^M$ is a constant vector, $h:\R\to\R$ is a profile and $\veta\in \R^n\setminus\{\vzero\}$ is a direction in space. 

The sets $\sU_{\vb}$ have to be chosen in such a way that they are compatible with plane waves. In order to define them, we have to study the wave cone $\Lambda'$, i.e. the cone of all constants $\vz\in \R^M$ with the property that there exists a direction $\veta\in \R^n\setminus\{\vzero\}$ such that every profile $h:\R\to\R$ yields a plane wave. 

\begin{defn} \label{defn:wave-cone} 
	The wave cone $\Lambda'$ is defined by 
	\begin{equation*}
		\Lambda' := \left\{ \vz\in \R^M\,\Big|\,\exists\veta\in \R^n\setminus\{\vzero\}\text{ such that } \mA(\vz) \cdot \veta = \vzero \right\}.
	\end{equation*} 
\end{defn}

It is simple to see that if $\vz\in \Lambda'$ and $h\in C^1(\R)$, then the plane wave defined in \eqref{eq:plane-wave} is a solution of \eqref{eq:lin-eq-homo}. Indeed we compute for any $i=1,...,m$
\begin{align*}
	[\Div \mA (\widetilde{\vz})]_i &= \sum_{j=1}^n \partial_j [\mA(\widetilde{\vz})]_{ij} = \sum_{j=1}^n [\mA(\partial_j\widetilde{\vz})]_{ij} = \sum_{j=1}^n [\mA(\vz h'(\vx\cdot \veta) \eta_j)]_{ij} \\
	&= h'(\vx\cdot \veta) \sum_{j=1}^n [\mA(\vz)]_{ij} \eta_j = h'(\vx\cdot \veta) [\mA(\vz) \cdot \veta]_{i} = 0 .
\end{align*}

For technical reasons, we work with a cone $\Lambda$ which is a subset of the wave cone, i.e. $\Lambda\subset \Lambda'$. This procedure is caused by the fact that suitable differential operators (see Defn.~\ref{defn:suitable-operator} below) may not exist for all $\vz\in\Lambda'$, see also Remark~\ref{rem:subset-of-wavecone} below.

\subsubsection{Suitable Differential Operator $\opL_\vz$} \label{subsubsec:gen-prelim-operator}

In order to implement convex integration, we will need a suitable differential operator $\opL_\vz$. It will be used to cut the plane waves off.

\begin{defn} \label{defn:suitable-operator} 
	Let $\vz\in\Lambda$ and $\ell\in \N$. An $\ell$-th order homogeneous differential operator 
	$$
		\opL_\vz : C^\infty(\R^{n}) \to C^\infty(\R^{n};\R^M) 
	$$
	is called \emph{suitable} if it satisfies the following two properties:
	\begin{enumerate}
		\item \label{item:suitable-operator-a} For any function $g\in C^\infty(\R^{n})$, $\opL_\vz[g]$ solves the linear system \eqref{eq:lin-eq-homo}, i.e.
		$$
			\Div \mA(\opL_\vz[g]) = \vzero.
		$$

		\item \label{item:suitable-operator-b} If we set $g(\vx):= h(\vx\cdot \veta)$ with an arbitrary function $h\in C^\infty(\R)$ and where $\veta$ corresponds to $\vz\in \Lambda\subset \Lambda'$, we obtain
		$$
			\opL_\vz [g](\vx) = \vz \,h^{(\ell)}(\vx\cdot \veta) ,
		$$
		where $h^{(\ell)}$ denotes the $\ell$-th derivative of $h$. 
	\end{enumerate}
\end{defn}

\subsection{Geometric Setup: The Family of Relaxed Sets $(\sU_\vb)_{\vb\in \sB}$ and Their Properties} \label{subsec:general-geometric}

Next we're going to define the relaxed sets $\sU_{\vb}$ for $\vb\in \sB$. As mentioned above, the relaxed sets $\sU_{\vb}$ have to be compatible with plane waves. To this end we set $\sU_{\vb}:= \interior{\big((\sK_{\vb})^\Lambda\big)}$ for all $\vb\in \sB$, i.e. the interior of the $\Lambda$-convex hull of the set $\sK_{\vb}$, see Appendix~\ref{app:Lconvex} for the definition of the $\Lambda$-convex hull and its properties. 

For our study we will assume that 
\begin{equation} \label{eq:KbLambda=KbCo}
	(\sK_{\vb})^\Lambda = (\sK_{\vb})^\co \qquad \text{ for all } \vb\in \sB,
\end{equation}
i.e. the $\Lambda$-convex hull of $\sK_{\vb}$ coincides with its convex hull.

\begin{rem} 
	Note that the fact \eqref{eq:KbLambda=KbCo} is true in most applications, see \cite[Proof of Lemma~4.3]{DelSze09} in the case of incompressible Euler, or \cite[Prop.~4.3.2]{Markfelder} in the case of compressible Euler, or Sect.~\ref{subsubsec:eif-sa-KLambda=Kco} below in the case of compressible Euler with energy inequality. Note further that \eqref{eq:KbLambda=KbCo} can be shown by proving that the wave cone $\Lambda$ is complete with respect to $\sK_\vb$, see Defn.~\ref{defn:app-complete-wc} and Prop.~\ref{prop:app-complete-wc}.
\end{rem}

Before we move on, let us note the following simple fact.

\begin{lemma} \label{lemma:uniform-boundedness-U}
	Let $(\sK_\vb)_{\vb\in\sB}$ a family of constitutive sets which is suitable in the sense of Defn.~\ref{defn:suitable-K}. Then
	$$
	|\vz|\leq c \qquad \text{ for all }\vb\in \sB \text{ and all } \vz\in (\sK_{\vb})^\co ,
	$$
	with $c>0$ from Lemma~\ref{lemma:uniform-boundedness-K}.
\end{lemma}

\begin{proof} 
The claim follows immediately from Prop.~\ref{prop:app-caratheodory}, Lemma~\ref{lemma:uniform-boundedness-K} and triangle inequality.
\end{proof}

The assumption \eqref{eq:KbLambda=KbCo} implies the following fact. Its proof can be also found in a less general form in \cite{Markfelder}, cf. Lemmas~5.2.2 and 5.2.3 therein. For the sake of completeness we recall this proof below.

\begin{prop} \label{prop:geom-property-U} 
	Consider a family of constitutive sets $(\sK_\vb)_{\vb\in \sB}$ which is suitable in the sense of Defn.~\ref{defn:suitable-K}, and assume that \eqref{eq:KbLambda=KbCo} holds. Let $\vb\in \sB$ fixed, $\vz\in \sU_{\vb}$ and $\ep>0$. Then there exist $N\in \N$ with $N\geq 2$, and $(\tau_i,\vz_i)\in \R^+\times \R^M$ for $i=1,...,N$ with the following properties\footnote{We refer to Appendix~\ref{app:Lconvex} for the definition and basic properties of the $H_N$-condition and the barycenter.}.
	\begin{enumerate}
		\item \label{item:geom-propU-a} The family $\{(\tau_i,\vz_i)\}_{i=1,...,N}$ satisfies the $H_N$-condition.
		
		\item \label{item:geom-propU-b} All endpoints lie in $\sU_{\vb}$ and their distance to $\partial \sU_{\vb}$ can be estimated by the distance of $\vz$ to $\partial \sU_{\vb}$, more precisely
		\begin{align*} 
			\vz_i &\in \sU_{\vb} & &\text{ for all }i=1,...,N, \\
			\dist ( \vz_i , \partial \sU_{\vb}) &\geq \frac{\ep}{2c} \dist(\vz,\partial \sU_{\vb}) & &\text{ for all }i=1,...,N,
		\end{align*}
		where $c$ is the bound from Lemma~\ref{lemma:uniform-boundedness-K}.
		
		\item \label{item:geom-propU-c} All endpoints are close to $\sK_{\vb}$, i.e.
		$$
			\dist(\vz_i,\sK_{\vb})\leq \ep \qquad \text{ for all }i=1,...,N.
		$$ 
		
		\item \label{item:geom-propU-d} The barycenter of the family $\{(\tau_i,\vz_i)\}_{i=1,...,N}$ is given by $\vz$, i.e. 
		$$
			\sum_{i=1}^N \tau_i \vz_i = \vz .
		$$
	\end{enumerate} 
\end{prop} 

\begin{proof} 
	We proceed in a similar way as in \cite{Markfelder}, cf. Lemmas~5.2.2 and 5.2.3 therein. Since $\vz\in \sU_{\vb} = \interior{\big((\sK_{\vb})^\Lambda\big)} \subset (\sK_{\vb})^\Lambda$, we can apply Prop.~\ref{prop:app-laminates} to find $N\in \N$ and $(\tau_i,\widehat{\vz}_i)\in \R^+\times \R^M$ for $i=1,...,N$ with the following properties: 
	\begin{itemize}
		\item The family $\{ ( \tau_i,\widehat{\vz}_i) \}_{i=1,...,N}$ satisfies the $H_N$-condition, 
		
		\item $\widehat{\vz}_i\in \sK_{\vb}$ for all $i=1,...,N$ and
		
		\item the barycenter of the family $\{ ( \tau_i,\widehat{\vz}_i) \}_{i=1,...,N}$ is given by $\vz$. 
	\end{itemize}
	Assume $N=1$, then $\vz=\widehat{\vz}_1 \in \sK_{\vb}$. Using \eqref{eq:KbLambda=KbCo} this means that $\vz\in \interior{\big((\sK_{\vb})^\Lambda\big)} \cap \sK_{\vb}$ which contradicts item \ref{item:suitable-K-boundary} of Defn.~\ref{defn:suitable-K}. Hence $N\geq 2$. 
	
	Now choose $\tau\in (0,1)$ such that 
	\begin{equation} \label{eq:aux05}
		\frac{\ep}{2c}\leq 1-\tau \leq \frac{\ep}{\max\limits_{i=1,...,N} | \vz - \widehat{\vz}_i | }.
	\end{equation} 
	Note that this is possible because\footnote{Note that actually we also need $\frac{\ep}{2c}<1$ to make the choice in \eqref{eq:aux05} possible. However when we apply Prop.~\ref{prop:geom-property-U} in the proof of the Perturbation Property (Prop.~\ref{prop:pert-prop}) below, we may assume that $\ep>0$ is suitably small. Alternatively, we may increase the universal constant $c$ in order to achieve $\frac{\ep}{2c}<1$.}
	$$
		| \vz - \widehat{\vz}_i | \leq |\vz| + |\widehat{\vz}_i | \leq 2c \qquad \text{ for all } i=1,...,N,
	$$
	due to Lemmas~\ref{lemma:uniform-boundedness-K} and \ref{lemma:uniform-boundedness-U}.
	
	Next we set 
	$$
		\vz_i:= \tau \widehat{\vz}_i + (1-\tau) \vz \qquad \text{ for } i=1,...,N.
	$$
	
	Now we prove that the desired conditions \ref{item:geom-propU-a}-\ref{item:geom-propU-d} hold.
	
	\begin{enumerate}
		\item The fact that the family $\{(\tau_i,\vz_i)\}_{i=1,...,N}$ satisfies the $H_N$-condition follows from the circumstance that 
		$$
			\widehat{\vz}_2 - \widehat{\vz}_1 \in \Lambda \quad \Rightarrow \quad \vz_2 - \vz_1 \in \Lambda .
		$$
		For more details we refer to \cite[Lemma~5.2.3]{Markfelder}. 
	
		\item We infer $\vz_i \in \sU_{\vb}$ for $i=1,...,N$ from Lemma~\ref{lemma:app-interior-convex} with $S:= (\sK_{\vb})^\co$ and assumption \eqref{eq:KbLambda=KbCo}. Furthermore note that according to Cor.~\ref{cor:app-boundary-vs-boundaryofinterior} we have 
		$$
			\dist(\cdot, \partial\sU_{\vb}) = \dist(\cdot, \partial(\sK_{\vb})^\co),
		$$
		which is concave on $(\sK_{\vb})^\co$ due to Lemma~\ref{lemma:app-distance-concave}. Hence we obtain
		$$
			\dist(\vz_i, \partial \sU_\vb) \geq \tau \dist(\widehat{\vz}_i, \partial \sU_\vb) + (1-\tau) \dist(\vz, \partial \sU_\vb).
		$$
		Item \ref{item:suitable-K-boundary} of Defn.~\ref{defn:suitable-K} implies $\sK_\vb\subset \partial (\sK_{\vb})^\co$ and therefore $\dist(\widehat{\vz}_i, \partial \sU_\vb)=0$. Together with \eqref{eq:aux05} this yields 
		$$
			\dist(\vz_i, \partial \sU_\vb) \geq \frac{\ep}{2c} \dist(\vz, \partial \sU_\vb)
		$$
		as desired.

		\item Our choice of $\tau$ (i.e. \eqref{eq:aux05}) ensures that for all $i\in\{1,...,N\}$ we have
		\begin{equation*}
			\dist(\vz_i,\sK_{\vb}) \leq | \vz_i - \widehat{\vz}_i | = (1-\tau) | \vz - \widehat{\vz}_i | \leq \ep.
		\end{equation*}
	
		\item Again analogously to \cite[Lemma~5.2.3]{Markfelder}, we obtain for the barycenter
		$$
			\sum_{i=1}^N \tau_i \vz_i = \tau \sum_{i=1}^N \tau_i \widehat{\vz}_i + (1-\tau) \vz = \vz.
		$$
	\end{enumerate}
\end{proof}

\subsection{Functional Setup} \label{subsec:general-functional}

\subsubsection{Convex Integration Theorem}

Next we state our convex integration theorem. 

\begin{thm} \label{thm:conv-int}
	Let $\Gamma\subset\R^n$ be a Lipschitz domain (open but not necessarily bounded), $\mA:\R^M\to \R^{m\times n}$ linear, $\vb\in \Cb(\closure{\Gamma};\R^m)$, and $(\sK_{\vb})_{\vb\in\sB}$ a family of constitutive sets, where $\sB:=\closure{\vb(\closure{\Gamma})}$ and $\sK_\vb\subset \R^M$ for all $\vb\in \sB$. Moreover let $\Lambda$ a cone with\footnote{See Defn.~\ref{defn:wave-cone} for the definition of the wave cone $\Lambda'$.} $\Lambda\subset \Lambda'$. \\ 
	Suppose that the following structural assumptions hold:
	\begin{itemize} 
		\item The family of constitutive sets $(\sK_{\vb})_{\vb\in\sB}$ is suitable in the sense of Defn.~\ref{defn:suitable-K}.
		
		\item For any $\vz\in \Lambda$ there exists $\ell\in \N$ and an $\ell$-th order homogeneous differential operator 
		$$
			\opL_\vz : C^\infty(\R^{n}) \to C^\infty(\R^{n};\R^M) 
		$$ 
		which is suitable in the sense of Defn.~\ref{defn:suitable-operator}.
		
		\item For any $\vb\in \sB$ the $\Lambda$-convex hull of $\sK_{\vb}$ coincides with its convex hull, i.e. 
		\begin{equation} \label{eq:ass-KbLambda=KbCo}
			(\sK_{\vb})^\Lambda = (\sK_{\vb})^\co \qquad \text{ for all } \vb\in \sB.
		\end{equation}
	\end{itemize}
	Finally assume there exists $\ov{\vz}\in C^1(\closure{\Gamma};\R^M)$ with the following properties:
	\begin{itemize}
		\item It satisfies the linear system of inequalities \eqref{eq:lin-eq}, i.e. 
		$$
			\Div \mA (\ov{\vz}) \leq \vb \qquad \text{ holds pointwise for all }\vx\in \Gamma;
		$$
		
		\item It takes values in $\sU_{\vb(\vx)}$ where $\sU_\vb:= \interior{\big((\sK_{\vb})^\Lambda\big)}$ for all $\vb\in \sB$, i.e. 
		$$ 
			\ov{\vz}(\vx) \in \sU_{\vb(\vx)}\qquad \text{ for all }\vx\in \Gamma . 
		$$ 
	\end{itemize}
	Then there exist infinitely many solutions $\vz\in L^\infty(\Gamma; \R^M)$ in the following sense:
	\begin{enumerate}
		\item \label{item:main-thm-sol1} They satisfy \eqref{eq:lin-eq} in the sense of distributions with boundary data given by $\ov{\vz}$, more precisely 
		\begin{equation*}
			-\int_\Gamma \mA(\vz) : \Grad \vphi \dx + \int_{\partial \Gamma} (\mA(\ov{\vz})\cdot \vn)\cdot \vphi \dS \leq \int_\Gamma \vb \cdot \vphi \dx
		\end{equation*}
		for all test functions $\vphi\in \Cc(\closure{\Gamma};\R^m)$ with $\vphi\geq \vzero$ in the sense of \eqref{eq:vector-order}. Here $\vn$ denotes the outward pointing normal vector on $\partial \Gamma$. 
		
		\item \label{item:main-thm-sol2} They take values in $\sK_{\vb(\vx)}$, i.e.
		\begin{equation*}
			\vz(\vx)\in \sK_{\vb(\vx)}\qquad \text{ for a.e. }\vx\in \Gamma. 
		\end{equation*}
	\end{enumerate}
\end{thm}

\begin{rem} \label{rem:subset-of-wavecone}
	Note that the wave cone $\Lambda'$ could be too large to ensure that there is a suitable differential operator $\opL_\vz$ for all $\vz\in \Lambda'$, see e.g. Remark~\ref{rem:wave-cone-euler} below. In this case we have to shrink it to a strict subset $\Lambda\subsetneq \Lambda'$. Note that on the other hand $\Lambda$ must be sufficiently large in order to guarantee \eqref{eq:ass-KbLambda=KbCo}. 
\end{rem}

The remainder of this subsection is devoted to the proof of Thm.~\ref{thm:conv-int}. It works similar to existing literature on convex integration. We will follow \cite{Markfelder}.

\subsubsection{The Functionals $\I_{\Gamma_0}$ and Their Properties} 

Let $\Gamma\subset\R^n$, $\mA:\R^M\to \R^{m\times n}$ linear, $\vb\in \Cb(\closure{\Gamma};\R^m)$, $(\sK_{\vb})_{\vb\in \sB}$ and $\ov{\vz}\in C^1(\closure{\Gamma};\R^M)$ be given such that the assumptions of Thm.~\ref{thm:conv-int} hold. 

\begin{defn} \label{defn:X0andX}
	We define the set $X_0$ by
	$$
		X_0 := \left\{ \vz\in C^1(\closure{\Gamma};\R^M) \, \Big|\, \text{Properties \eqref{eq:defn-X0-eq}-\eqref{eq:defn-X0-boundary} hold} \right\} ,
	$$
	where the properties \eqref{eq:defn-X0-eq}-\eqref{eq:defn-X0-boundary} read as follows:
	\begin{align}
		\Div \mA (\vz) &\leq \vb & & \text{ holds pointwise for all }\vx\in \Gamma; \label{eq:defn-X0-eq} \\
		\vz(\vx) &\in \sU_{\vb(\vx)} & & \text{ for all }\vx\in \Gamma ; \label{eq:defn-X0-subs} \\ 
		\vz(\vx) &= \ov{\vz}(\vx) & & \text{ for all }\vx\in \partial\Gamma. \label{eq:defn-X0-boundary}
	\end{align}
	Moreover we denote the closure of $X_0$ with respect to the $L^\infty$ weak-$\ast$ topology by $X$. 
\end{defn}

In other words $\vz\in X_0$ satisfies the linear system \eqref{eq:lin-eq}, takes values in $\sU_{\vb(\vx)}$ and coincides on the boundary with $\ov{\vz}$.

\begin{prop} \label{prop:X-metric-space}
	There exists a metric $d$ on $X$ which induces the $L^\infty$ weak-$\ast$ topology, and furthermore the metric space $(X,d)$ is compact and complete. 
\end{prop}

\begin{proof} 
	First of all, we observe that $X_0$ is bounded with respect to the $L^\infty$ norm. Indeed according to Lemma~\ref{lemma:uniform-boundedness-U} each $(\sK_{\vb})^\co$ is bounded and its bound is independent on $\vb\in \sB$. Hence assumption \eqref{eq:ass-KbLambda=KbCo} and condition \eqref{eq:defn-X0-subs} yield that $X_0$ is bounded with respect to the $L^\infty$ norm as desired.
	
	As a consequence, well-known facts yield that the weak-$\ast$ topology on $X$ is metrizable, and that the resulting metric space $(X,d)$ is compact and complete, see e.g. \cite[Proof of Prop.~5.1.6]{Markfelder} and references therein for more details. 
\end{proof}

Now we define the functionals $\I_{\Gamma_0}$.

\begin{defn} \label{defn:I}
	For any open and bounded subset $\Gamma_0 \subset \Gamma$ we define the functional $\I_{\Gamma_0}:L^\infty(\Gamma_0;\R^M) \to \R$ by 
	\begin{equation*} 
		\vz \mapsto \I_{\Gamma_0} (\vz) := \int_{\Gamma_0} \dist\big(\vz(\vx),\sK_{\vb(\vx)}\big) \dx .
	\end{equation*}
\end{defn}

Next we summarize some important properties of $\dist(\vz,\sK_{\vb})$ and the functionals $\I_{\Gamma_0}$ respectively. We begin with the properties of $\dist(\vz,\sK_{\vb})$.

\begin{lemma} \label{lemma:properties-dist}
	The following statements hold.
	\begin{enumerate}
		\item We have 
		\begin{equation} \label{eq:dist-bounded} 
			0\leq \dist (\vz, \sK_{\vb}) \leq 2c \qquad \text{ for all } \vb\in \sB \text{ and all } \vz\in (\sK_\vb)^\co,
		\end{equation}
		where $c>0$ is the bound from Lemma~\ref{lemma:uniform-boundedness-K}. 
		
		\item For all $\vz_1,\vz_2\in \R^M$ and $\vb\in \sB$ it holds that 
		\begin{equation} \label{eq:dist-cont-z}
			\Big| \dist(\vz_1, \sK_{\vb}) - \dist(\vz_2, \sK_{\vb}) \Big| \leq |\vz_1 - \vz_2|.
		\end{equation}
		Estimate \eqref{eq:dist-cont-z} also holds if one replaces $\sK_\vb$ by\footnote{In fact one can replace $\sK_\vb$ by any other compact set $S\subset \R^M$ and \eqref{eq:dist-cont-z} still holds.} $(\sK_\vb)^\co$ or $\partial(\sK_\vb)^\co$.
		
		\item Let $\ep>0$ and $\delta>0$ the $\delta$ from item \ref{item:suitable-K-continuity} of Defn.~\ref{defn:suitable-K} which corresponds to $\ep$. Then for all $\vz\in \R^M$ and $\vb_1,\vb_2\in \sB$ with $|\vb_1 - \vb_2 |<\delta$ it holds that 
		\begin{equation} \label{eq:dist-cont-b}
			\Big| \dist(\vz, \sK_{\vb_1}) - \dist(\vz, \sK_{\vb_2}) \Big| \leq \ep .
		\end{equation}
		Estimate \eqref{eq:dist-cont-b} also holds if one replaces $\sK_\vb$ by $(\sK_\vb)^\co$ or $\partial(\sK_\vb)^\co$, where one has to use the corresponding $\delta$ from Lemma~\ref{lemma:continuity-Kco}.
	\end{enumerate}
\end{lemma}

\begin{proof}
\begin{enumerate} 
	\item Lemmas~\ref{lemma:uniform-boundedness-K} and \ref{lemma:uniform-boundedness-U} lead to
	$$
		0\leq \dist (\vz, \sK_{\vb}) \leq |\vz - \widehat{\vz}| \leq |\vz| + |\widehat{\vz}| \leq 2c ,
	$$
	where $\widehat{\vz}\in \sK_{\vb}$ arbitrary. Hence \eqref{eq:dist-bounded} is valid.
	
	\item Let $S\subset \R^M$ a compact set and we set $\widehat{\vz}_1,\widehat{\vz}_2\in S$ such that 
	$$
		\dist (\vz_i , S) = |\vz_i - \widehat{\vz}_i| \qquad \text{ for } i=1,2.
	$$ 
	The triangle inequality yields 
	\begin{align*} 
		\dist(\vz_1, S) &\leq |\vz_1 - \widehat{\vz}_2| \leq |\vz_1 - \vz_2| + |\vz_2 - \widehat{\vz}_2| \leq |\vz_1 - \vz_2| + \dist(\vz_2, S) , \\
		\dist(\vz_2, S) &\leq |\vz_2 - \widehat{\vz}_1| \leq |\vz_2 - \vz_1| + |\vz_1 - \widehat{\vz}_1| \leq |\vz_2 - \vz_1| + \dist(\vz_1, S) ,
	\end{align*}
	and thus \eqref{eq:dist-cont-z}.
	
	\item Similar to above, we set $\widehat{\vz}_1\in \sK_{\vb_1}$ and $\widehat{\vz}_2\in \sK_{\vb_2}$ such that 
	$$
		\dist (\vz , \sK_{\vb_i}) = |\vz - \widehat{\vz}_i| \qquad \text{ for } i=1,2.
	$$
	Then item \ref{item:suitable-K-continuity} of Defn.~\ref{defn:suitable-K} yields $\widetilde{\vz}_1\in \sK_{\vb_1}$ and $\widetilde{\vz}_2\in \sK_{\vb_2}$ with 
	$$
		|\widetilde{\vz}_1 - \widehat{\vz}_2 | < \ep \qquad \text{ and } \qquad |\widetilde{\vz}_2 - \widehat{\vz}_1 | < \ep.
	$$
	Using triangle inequality we find 
	\begin{align*} 
		\dist(\vz, \sK_{\vb_1}) &\leq |\vz - \widetilde{\vz}_1| \leq |\vz - \widehat{\vz}_2| + |\widehat{\vz}_2 - \widetilde{\vz}_1| < \dist(\vz, \sK_{\vb_2}) + \ep , \\
		\dist(\vz, \sK_{\vb_2}) &\leq |\vz - \widetilde{\vz}_2| \leq |\vz - \widehat{\vz}_1| + |\widehat{\vz}_1 - \widetilde{\vz}_2| < \dist(\vz, \sK_{\vb_1}) + \ep ,
	\end{align*}
	and hence \eqref{eq:dist-cont-b}. The fact that \eqref{eq:dist-cont-b} still holds if one replaces $\sK_\vb$ by $(\sK_\vb)^\co$ or $\partial(\sK_\vb)^\co$, follows in the same fashion using Lemma~\ref{lemma:continuity-Kco}.
\end{enumerate}
\end{proof}

Now we summarize some properties of the functionals $\I_{\Gamma_0}$.

\begin{lemma} \label{lemma:properties-I}
	The following claims hold.
	\begin{enumerate} 
		\item \label{item:prop-I-Baire1} For all open and bounded sets $\Gamma_0\subset \Gamma$ the map $\I_{\Gamma_0} : X\to \R$ is a Baire-1 function with respect to the metric $d$, i.e. $\I_{\Gamma_0}$ is the pointwise limit of a sequence of continuous functionals. 
		
		\item \label{item:prop-I-X0} For all non-empty, open and bounded sets $\Gamma_0\subset \Gamma$ and all $\vz\in X_0$ we have $\I_{\Gamma_0}(\vz) > 0$. 
		
		\item \label{item:prop-I-sol} If $\vz\in X$ with $\I_{\Gamma_0}(\vz)=0$ for all open and bounded sets $\Gamma_0 \subset \Gamma$, then $\vz$ is a solution in the sense specified in Thm.~\ref{thm:conv-int}, i.e. items \ref{item:main-thm-sol1} and \ref{item:main-thm-sol2} hold. 
	\end{enumerate}
\end{lemma}

\begin{proof} 
\begin{enumerate}
	\item Let $\Gamma_0\subset \Gamma$ open and bounded. We want to show that $\I:=\I_{\Gamma_0}$ is a Baire-1 function. To this end we first show the following auxiliary claim: 
	\begin{equation} \label{eq:aux01}
		\text{If }\vz_k\to \vz \text{ in } L^1(\Gamma_0;\R^M), \text{ then } \lim\limits_{k\to \infty} \I(\vz_k) = \I(\vz).
	\end{equation} 
	Lemma~\ref{lemma:properties-dist} yields for any fixed $\vx\in \Gamma_0$, that
	$$
		\Big| \dist\big(\vz_k(\vx), \sK_{\vb(\vx)}\big) - \dist\big(\vz(\vx), \sK_{\vb(\vx)}\big) \Big| \leq |\vz_k(\vx) - \vz(\vx)|,
	$$
	see \eqref{eq:dist-cont-z}. Hence 
	\begin{align*}
		\Big| \I(\vz_k) - \I(\vz) \Big| &\leq \int_{\Gamma_0} \Big| \dist\big(\vz_k(\vx),\sK_{\vb(\vx)}\big) - \dist\big(\vz(\vx),\sK_{\vb(\vx)}\big) \Big| \dx \\
		&\leq \int_{\Gamma_0} |\vz_k(\vx) - \vz(\vx)| \dx = \| \vz_k - \vz \|_{L^1(\Gamma_0;\R^M)}
	\end{align*}
	which gives \eqref{eq:aux01}.
	
	Let furthermore $\phi_\ep$ a standard mollifier and define $\I_\ep : L^\infty(\Gamma_0;\R^M) \to \R $ for $\ep>0$ by 
	$$
		\vz \mapsto \I_{\ep} (\vz) := \I(\vz\ast \phi_\ep) = \int_{\Gamma_0} \dist\big((\vz\ast \phi_\ep)(\vx),\sK_{\vb(\vx)}\big) \dx .
	$$
	
	Now for fixed $\vz\in L^\infty(\Gamma_0;\R^M)$, we know that $\vz\ast \phi_\ep \to \vz$ in $L^1$ as $\ep \to 0$, see e.g. \cite[Appendix~C.5]{Evans}. Hence we deduce from \eqref{eq:aux01} that the $\I_\ep$ converge pointwise to $\I$ as $\ep\to 0$. 
	
	It remains to show that for any fixed $\ep>0$, the map $\I_\ep$ is continuous with respect to $d$ or equivalently with respect to the $L^\infty$ weak-$\ast$ topology. Since $\vz_k \mathop{\rightharpoonup}\limits^{\ast} \vz$ in $L^\infty$ implies that $\vz_k\ast \phi_\ep \to \vz\ast \phi_\ep$ in $L^1$ as $k\to \infty$, the desired continuity again follows from \eqref{eq:aux01}. 
	
	\item From property \eqref{eq:defn-X0-subs} and Defn.~\ref{defn:suitable-K} \ref{item:suitable-K-compact} and \ref{item:suitable-K-boundary}, we infer that 
	$$
		\dist(\vz(\vx),\sK_{\vb(\vx)} ) > 0 \qquad \text{ for all } \vx\in \Gamma,
	$$
	which immediately leads to $\I_{\Gamma_0}(\vz)> 0$.
	
	\item It is easy to see from the definition of the set $X_0$ (Defn.~\ref{defn:X0andX}), that any $\vz\in X_0$ satisfies item \ref{item:main-thm-sol1} in Thm.~\ref{thm:conv-int}. Now let $\vz\in X$. Then by defintion of the set $X$, there is a sequence $(\vz_k)_{k\in \N}\subset X_0$ with $\vz_k \mathop{\rightharpoonup}\limits^\ast \vz$. So item \ref{item:main-thm-sol1} in Thm.~\ref{thm:conv-int} holds for any $\vz_k$, $k\in \N$, and by taking the limit we obtain item \ref{item:main-thm-sol1} even for $\vz\in X$. 
	
	It remains to prove that item \ref{item:main-thm-sol2} in Thm.~\ref{thm:conv-int} holds. The fact that $\I_{\Gamma_0}(\vz)=0$ for all open and bounded sets $\Gamma_0\subset \Gamma$ implies that
	$$
		\dist(\vz(\vx),\sK_{\vb(\vx)}) = 0 \qquad \text{ for a.e. } \vx\in \Gamma,
	$$
	a detailed proof of which is given in \cite[Lemma~A.4.1]{Markfelder}\footnote{Here we have to make use that the map $\vx\mapsto \dist(\vz(\vx),\sK_{\vb(\vx)})$ is essentially bounded according to Lemma~\ref{lemma:properties-dist}, see \eqref{eq:dist-bounded}.}. Due to Defn.~\ref{defn:suitable-K} \ref{item:suitable-K-compact}, this in turn means that item \ref{item:main-thm-sol2} in Thm.~\ref{thm:conv-int} is true.
\end{enumerate}
\end{proof}

\subsubsection{Perturbation Property and Proof of the Convex Integration Theorem (Thm.~\ref{thm:conv-int})}

Let us next state the Perturbation Property, which we prove in Sect.~\ref{subsec:general-pert-prop} below. It is the core in the proof of Thm.~\ref{thm:conv-int}.

\begin{prop}[The Perturbation Property] \label{prop:pert-prop}
	Let $\ep,\ov{\ep}>0$ and $\Gamma_0\subset \Gamma$ open and bounded. For all $\vz \in X_0$ there exists a perturbation $\vz_\pert\in X_0$ with the following properties
	\begin{align}
		d(\vz_\pert, \vz) &\leq \ov{\ep} ; \label{eq:pert-prop1} \\
		\I_{\Gamma_0}(\vz_\pert) &\leq \ep \label{eq:pert-prop2} .
	\end{align}
\end{prop}

Now in order to prove Thm.~\ref{thm:conv-int} we can proceed as in \cite{Markfelder}, see Sect.~5.1.4 therein. For the sake of completeness we repeat the main ideas.

We consider an exhausting sequence of open and bounded subsets of $\Gamma$ denoted by $(\Gamma_j)_{j\in \N}$. Moreover we define
\begin{align*}
	\Xi_j &:= \left\{\vz \in X \, \Big| \, \I_{\Gamma_j} \text{ is continuous at } \vz \text{ with respect to } d\right\} , \\
	\Xi \ &:= \bigcap_{j\in \N} \Xi_j .
\end{align*}

Thm.~\ref{thm:conv-int} is a consequence of the following two lemmas, both of which are in turn shown using the Perturbation Property (Prop.~\ref{prop:pert-prop}).

\begin{lemma} \label{lemma:Xi-inf-elements}
	The sets $X_0$ and $\Xi$ contain infinitely many elements.
\end{lemma}

\begin{lemma} \label{lemma:Xi-I=0}
	If $\vz\in \Xi$, then $I_{\Gamma_0}(\vz)=0$ for all open and bounded $\Gamma_0\subset \Gamma$. 
\end{lemma}

\begin{proof}[Proof of Thm.~\ref{thm:conv-int}] 
Now Thm.~\ref{thm:conv-int} follows from Lemmas~\ref{lemma:properties-I}~\ref{item:prop-I-sol}, \ref{lemma:Xi-inf-elements} and \ref{lemma:Xi-I=0}, see \cite[Proof of Thm.~5.1.2]{Markfelder} for more details.
\end{proof}

It remains to prove Lemmas~\ref{lemma:Xi-inf-elements} and \ref{lemma:Xi-I=0}.

\begin{proof}[Proof of Lemma~\ref{lemma:Xi-inf-elements}] 
Let us first note that $\ov{\vz}$ satisfies \eqref{eq:defn-X0-eq}-\eqref{eq:defn-X0-boundary} and hence $\ov{\vz}\in X_0$, in particular $X_0\neq \emptyset$. Similarly to \cite[Proof of Lemma~5.1.12]{Markfelder} one can use the Perturbation Property (Prop.~\ref{prop:pert-prop}) to construct a sequence of pairwise different $\vz_k\in X_0$, which shows that $X_0$ contains infinitely many elements.

Due to Lemma~\ref{lemma:properties-I}~\ref{item:prop-I-Baire1} and \cite[Prop.~7.3.2]{Waldmann} the sets $\Xi_j$ are residual for all $j\in \N$. Following \cite[Proof of Lemma~5.1.14]{Markfelder} we can deduce using the Baire Category Theorem that $\Xi$ is dense in $X$, and that $X$ contains infinitely many elements too.
\end{proof} 

\begin{proof}[Proof of Lemma~\ref{lemma:Xi-I=0}] 
Analogously to \cite[Proof of Lemma~5.1.13]{Markfelder} one can prove using the Perturbation Property (Prop.~\ref{prop:pert-prop}) that if $\vz\in \Xi_j$ for some $j\in \N$, then $\I_{\Gamma_j}(\vz)=0$. This implies the claim, see \cite[Proof of Lemma~5.1.13]{Markfelder} for more details.
\end{proof}

\subsection{Proof of the Perturbation Property (Prop.~\ref{prop:pert-prop})} \label{subsec:general-pert-prop}

The only thing that remains is to prove the Perturbation Property (Prop.~\ref{prop:pert-prop}). 

\subsubsection{Step 1: Inductive Lemma} 

We begin with the following lemma, cf. \cite[Lemma~5.2.4]{Markfelder} for a less general version.

\begin{lemma} \label{lemma:pert-prop-step1} 
	Let $\Gamma^\ast\subset \R^n$ an open and bounded set (not necessarily connected), $\ep>0$, $\mu\in (0,1)$, $\vb^\ast\in \sB$ and $\vz^\ast\in \sU_{\vb^\ast}$. Let furthermore $N\in\N$ with $N\geq 2$, and $(\tau_i,\vz_i)\in \R^+ \times \sU_{\vb^\ast}$ for all $i=1,...,N$ such that 
	\begin{itemize} 
		\item the family $\{(\tau_i,\vz_i)\}_{i=1,...,N}$ satisfies the $H_N$-condition and 
		\item its barycenter is given by $\vz^\ast$, i.e. $\sum_{i=1}^N \tau_i \vz_i = \vz^\ast$. 
	\end{itemize}
	Then there exists a sequence of oscillations $(\widetilde{\vz}_k)_{k\in \N} \subset \Cc(\Gamma^\ast;\R^M)$ with the following properties:
	\begin{enumerate}
		\item \label{item:pert-prop-step1-a} The sequence $(\widetilde{\vz}_k)_{k\in \N}$ converges to $\vzero$ with respect to $d$, i.e. $\widetilde{\vz}_k \mathop{\to}\limits^d \vzero$ as $k\to \infty$. 
	\end{enumerate}	
	For each fixed $k\in \N$ the following statements hold:
	\begin{enumerate} \setcounter{enumi}{1}
		\item \label{item:pert-prop-step1-b} The linear system of PDEs $\ \Div \mA(\widetilde{\vz}_k) = \vzero\ $ holds pointwise for all $\vx\in \Gamma^\ast$.
		
		\item \label{item:pert-prop-step1-c} There exist open sets $\Gamma_i \subsetcomp \Gamma^\ast$ for all $i=1,...,N$ such that
		\begin{itemize}
			\item the $\closure{\Gamma_i}$ are pairwise disjoint and 
			\item $\vz^\ast+\widetilde{\vz}_k(\vx) = \vz_i$ for all $\vx\in \Gamma_i$.
		\end{itemize} 
		
		\item \label{item:pert-prop-step1-d} The sum $\vz^\ast + \widetilde{\vz}_k$ takes values in $\sU_{\vb^\ast}$ and its distance to $\partial\sU_{\vb^\ast}$ can be estimated by the distance of the $\vz_i$ to $\partial\sU_{\vb^\ast}$, more precisely 
		\begin{align}
			\vz^\ast + \widetilde{\vz}_k(\vx) &\in \sU_{\vb^\ast} & &\text{ for all }\vx\in \Gamma^\ast, \label{eq:pert-prop-step1-d1} \\
			\dist (\vz^\ast + \widetilde{\vz}_k(\vx), \partial\sU_{\vb^\ast}) &\geq \mu \min_{i=1,...,N} \dist(\vz_i,\partial\sU_{\vb^\ast}) & & \text{ for all }\vx\in \Gamma^\ast. \label{eq:pert-prop-step1-d2} 
		\end{align}
		
		\item \label{item:pert-prop-step1-e} The following estimate holds
		$$ 
			\int_{\Gamma^\ast} \dist\big(\vz^\ast + \widetilde{\vz}_k(\vx),\sK_{\vb^\ast}\big) \dx < \ep + \sum_{i=1}^N \int_{\Gamma_i} \dist(\vz_i,\sK_{\vb^\ast}) \dx.
		$$
	\end{enumerate}
\end{lemma}

\begin{proof} 
The proof works in the same fashion as \cite[Proof of Lemma~5.2.4]{Markfelder}. We proceed by induction over $N\geq 2$. 

\textbf{Induction basis:} For $N=2$ we have by assumption $\tau_1+\tau_2=1$, $\vz^\ast=\tau_1 \vz_1 + \tau_2 \vz_2$ and $\vz_2 - \vz_1 \in \Lambda$. Defn.~\ref{defn:wave-cone} yields $\veta\in \R^n\setminus \{\vzero\}$ with $\mA(\vz_2-\vz_1) \cdot \veta = \vzero$. Fix an open set $\widetilde{\Gamma}\subsetcomp \Gamma^\ast$ such that the volume of the ``frame'' $\Gamma_\tf:= \Gamma^\ast \setminus \widetilde{\Gamma}$ is sufficiently small, more precisely 
\begin{equation} \label{eq:choice-Gamma_f}
	|\Gamma_\tf| < \frac{\ep}{4c}, 
\end{equation}
where $c$ is the constant from Lemma~\ref{lemma:uniform-boundedness-K}. 

Let $Q\subset \R^n$ be an open cube such that $\widetilde{\Gamma}\subsetcomp Q$, one edge of $Q$ is parallel to $\veta$, and the edge length of $Q$ is a natural multiple of $\frac{1}{|\veta|}$. Next we fix 
\begin{equation} \label{eq:choice-delta}
	0< \delta < \min\left\{ \frac{\ep}{16c |Q|} , \frac{\tau_1}{2} , \frac{\tau_2}{2} \right\}, 
\end{equation} 
and let $\Phi\in \Cc(\Gamma^\ast;[0,1])$ be a smooth cutoff function with $\Phi \equiv 1$ on $\widetilde{\Gamma}$. 

Let $f:\R\to \R$, 
$$
	f(t) := \left\{ \begin{array}{rl} -\tau_2 & \text{ if } t\in [0,\tau_1) + \Z, \\ \tau_1 & \text{ if } t\in [\tau_1,1) + \Z. \end{array} \right.
$$
Mollify $f$ to obtain $f_\delta\in C^\infty (\R)$, which is periodic with zero mean, takes values in $[-\tau_2,\tau_1]$ and satisfies 
\begin{align}
	f_\delta (t) &= - \tau_2 & &\text{ for all } t\in [\delta,\tau_1-\delta] + \Z , \label{eq:f_delta-case1} \\
	f_\delta (t) &= \tau_1 & &\text{ for all } t\in [\tau_1+\delta, 1-\delta] + \Z , \label{eq:f_delta-case2} 
\end{align}
see \cite[Lemma~A.4.7]{Markfelder}. By assumption there exists $\ell\in \N$ and an $\ell$-th order homogeneous differential operator $\opL_{\vz_2-\vz_1}: C^\infty(\R^{n}) \to C^\infty(\R^{n};\R^M)$ which is suitable. 

Moreover there exists $h\in C^\infty(\R)$ such that $h^{(\ell)}= f_\delta$, and $h$ and all derivatives of $h$ up to $\ell$-th order are bounded, see \cite[Lemma~A.4.8]{Markfelder}. We define $g_k\in C^\infty(\R^n)$ by 
$$
	g_k (\vx) := \frac{1}{k^\ell} h( k \, \vx\cdot \veta) \qquad \text{ for } k\in \N.
$$
and $\widetilde{\vz}_k := \opL_{\vz_2 - \vz_1} [g_k \Phi]$. Note that $\widetilde{\vz}_k \in \Cc(\Gamma^\ast;\R^M)$ because $\Phi\in \Cc(\Gamma^\ast)$.

The reader should notice that according to Defn.~\ref{defn:suitable-operator} we have 
\begin{equation} \label{eq:aux02}
	\opL_{\vz_2 - \vz_1} [g_k](\vx) = (\vz_2 - \vz_1) h^{(\ell)} (k\,\vx\cdot \veta ) = (\vz_2 - \vz_1) f_\delta (k\,\vx\cdot \veta ).
\end{equation}
Moreover we may estimate that 
\begin{equation} \label{eq:commutator-opL-phi}
	\Big\| \opL_{\vz_2 - \vz_1} [g_k \Phi] - \opL_{\vz_2 - \vz_1} [g_k] \Phi \Big\|_\infty \lesssim \frac{1}{k}.
\end{equation}

Now we look at properties \ref{item:pert-prop-step1-a}-\ref{item:pert-prop-step1-e} in Lemma~\ref{lemma:pert-prop-step1}.

\begin{enumerate}
	\item It is clear\footnote{It is a well-known fact that if $f$ is periodic, then $f_k:= f(k\cdot)$ converges weakly-$\ast$ to the mean of $f$ as $k\to \infty$. The reader might want to consult \cite[Lemma~A.7.2]{Markfelder} for a proof of this fact.} from \eqref{eq:aux02} that $\opL_{\vz_2 - \vz_1} [g_k] \Phi$ converges weakly-$\ast$ to $\vzero$ as $k\to \infty$. Due to \eqref{eq:commutator-opL-phi}, even $\widetilde{\vz}_k = \opL_{\vz_2 - \vz_1} [g_k \Phi] \mathop{\to}\limits^d \vzero$. 

	\item The fact that $\ \Div \mA(\widetilde{\vz}_k) = \vzero\ $ holds pointwise for all $\vx\in \Gamma^\ast$ follows immediately from Defn.~\ref{defn:suitable-operator}.
	
	\item On $\widetilde{\Gamma}$ it holds that $\Phi\equiv 1$ and due to \eqref{eq:f_delta-case1}, \eqref{eq:f_delta-case2} and \eqref{eq:aux02} we obtain $\vz^\ast + \widetilde{\vz}_k(\vx) = \vz_i$ for $\vx\in \Gamma_i$, where 
	\begin{align*}
		\Gamma_1 &:= \widetilde{\Gamma} \cap \left\{ \vx\in \R^n \, \Big|\, k\,\vx\cdot \veta \in (\delta, \tau_1 - \delta) +\Z \right\}, \\ 
		\Gamma_2 &:= \widetilde{\Gamma} \cap \left\{ \vx\in \R^n \, \Big|\, k\,\vx\cdot \veta \in (\tau_1 + \delta , 1-\delta) +\Z \right\}.
	\end{align*}
	Note that $\closure{\Gamma_1} \cap \closure{\Gamma_2} = \emptyset$ due to the choice of $\delta$, see \eqref{eq:choice-delta}.
	
	\item Because of \eqref{eq:aux02}, for each $\vx\in \Gamma^\ast$ there exists $\tau\in [-\tau_2,\tau_1]$ such that 
	\begin{equation} \label{eq:aux04}
		\vz^\ast + \opL_{\vz_2 - \vz_1} [g_k] \Phi = (\tau_1-\tau) \vz_1 + (\tau_2 + \tau) \vz_2 \ \in \ \{\vz_1,\vz_2\}^\co.
	\end{equation}
	Hence Lemma~\ref{lemma:app-interior-convex} yields that 
	$$
		\vz^\ast + \opL_{\vz_2 - \vz_1} [g_k](\vx) \Phi(\vx) \in \sU_{\vb^\ast} \qquad \text{ for all }\vx\in \Gamma^\ast. 
	$$
	Together with \eqref{eq:commutator-opL-phi} this implies \eqref{eq:pert-prop-step1-d1} as soon as $k$ is sufficiently large. 
	
	It remains to show \eqref{eq:pert-prop-step1-d2}. Note that due to Cor.~\ref{cor:app-boundary-vs-boundaryofinterior} we have 
	$$
		\dist(\cdot, \partial\sU_{\vb^\ast}) = \dist(\cdot, \partial(\sK_{\vb^\ast})^\co),
	$$
	which is concave on $(\sK_{\vb^\ast})^\co$ according to Lemma~\ref{lemma:app-distance-concave}. Hence \eqref{eq:aux04} yields 
	\begin{align*}
		\dist (\vz^\ast + \opL_{\vz_2 - \vz_1} [g_k](\vx) \Phi(\vx), \partial\sU_{\vb^\ast}) &\geq (\tau_1-\tau) \dist (\vz_1, \partial\sU_{\vb^\ast}) + (\tau_2+\tau) \dist (\vz_2, \partial\sU_{\vb^\ast}) \\
		&\geq \min_{i=1,2} \dist(\vz_i,\partial\sU_{\vb^\ast}) \qquad \text{ for all }\vx\in \Gamma^\ast .
	\end{align*} 
	Together with 
	\begin{align*}
		&\Big| \dist (\vz^\ast + \opL_{\vz_2 - \vz_1} [g_k \Phi](\vx), \partial\sU_{\vb^\ast}) - \dist (\vz^\ast + \opL_{\vz_2 - \vz_1} [g_k](\vx) \Phi(\vx), \partial\sU_{\vb^\ast}) \Big| \\
		&\leq \Big| \opL_{\vz_2 - \vz_1} [g_k \Phi](\vx) - \opL_{\vz_2 - \vz_1} [g_k](\vx) \Phi(\vx) \Big|
	\end{align*} 
	which holds according to Lemma~\ref{lemma:properties-dist} (see \eqref{eq:dist-cont-z}), we end up with \eqref{eq:pert-prop-step1-d2} for $k$ sufficiently large, see \eqref{eq:commutator-opL-phi}.
	
	\item Define the ``slices'' 
	$$
		\Gamma_\ts := \widetilde{\Gamma} \cap \left\{ \vx\in \R^n \, \Big|\, k\,\vx\cdot \veta \in \Big( [0,\delta] \cup [\tau_1 - \delta \tau_1+\delta] \cup [1-\delta, 1) \Big) +\Z \right\}.
	$$
	Then we have 
	\begin{align*}
		\int_{\Gamma^\ast} \dist\big(\vz^\ast + \widetilde{\vz}_k(\vx),\sK_{\vb^\ast}\big) \dx &= \int_{\Gamma_\tf} \dist \big(\vz^\ast + \widetilde{\vz}_k(\vx) , \sK_{\vb^\ast} \big) \dx \\
		&\qquad + \int_{\Gamma_\ts} \dist \big(\vz^\ast + \widetilde{\vz}_k(\vx) , \sK_{\vb^\ast} \big) \dx \\
		&\qquad + \sum_{i=1}^2 \int_{\Gamma_i} \dist\big(\vz^\ast + \widetilde{\vz}_k(\vx),\sK_{\vb^\ast}\big) \dx,
	\end{align*}
	which leads to 
	$$ 
		\int_{\Gamma^\ast} \dist\big(\vz^\ast + \widetilde{\vz}_k(\vx),\sK_{\vb^\ast}\big) \dx \leq 2c \Big( |\Gamma_\tf| + |\Gamma_\ts| \Big) + \sum_{i=1}^2 \int_{\Gamma_i} \dist(\vz_i,\sK_{\vb^\ast}) \dx
	$$
	by using items \ref{item:pert-prop-step1-c} and \ref{item:pert-prop-step1-d} and estimate \eqref{eq:dist-bounded}. The volume of the the slices $\Gamma_\ts$ can be estimated as
	$$
		|\Gamma_\ts| \leq 4\delta |Q|,
	$$
	see \cite[Proof of Lemma~5.2.4]{Markfelder} for more details. Thus we obtain from \eqref{eq:choice-Gamma_f} and \eqref{eq:choice-delta} that $2c ( |\Gamma_\tf| + |\Gamma_\ts| ) < \ep$ which finishes the proof for $N=2$. 
\end{enumerate}

\textbf{Induction step:} Let $N\geq 3$. We know according to Defn.~\ref{defn:app-hn} (after relabelling if necessary) that $\vz_2-\vz_1 \in \Lambda$ and 
\begin{equation} \label{eq:hn-iteration}
	\left\{ \left( \tau_1 + \tau_2 , \frac{\tau_1}{\tau_1 + \tau_2} \vz_1 + \frac{\tau_2}{\tau_1 + \tau_2} \vz_2\right) \right\} \cup \big\{(\tau_i,\vz_i)\big\}_{i=3,...,N}
\end{equation}
satisfies the $H_{N-1}$-condition. An obvious calculation shows that the family in \eqref{eq:hn-iteration} and the original family $\{(\tau_i,\vz_i)\}_{i=1,...,N}$ have the same barycenter $\vz^\ast$. Define $\vz_a := \frac{\tau_1}{\tau_1 + \tau_2} \vz_1 + \frac{\tau_2}{\tau_1 + \tau_2} \vz_2$. By means of Lemma~\ref{lemma:app-interior-convex} we obtain $\vz_a\in \sU_{\vb^\ast}$. Hence we may apply the induction hypothesis (where $\frac{\ep}{2}$ plays the role of $\ep$) to obtain a sequence of oscillations $(\widetilde{\vz}_{A,k})_{k\in \N} \subset \Cc(\Gamma^\ast;\R^M)$ with the following properties:
\begin{enumerate}
	\item \label{item:pert-prop-step1-IH-a} The sequence $(\widetilde{\vz}_{A,k})_{k\in \N}$ converges to $\vzero$ with respect to $d$, i.e. $\widetilde{\vz}_{A,k} \mathop{\to}\limits^d \vzero$ as $k\to \infty$. 
\end{enumerate}	
For each fixed $k\in \N$ the following statements hold:
\begin{enumerate} \setcounter{enumi}{1}
	\item \label{item:pert-prop-step1-IH-b} The linear system of PDEs $\ \Div \mA(\widetilde{\vz}_{A,k}) = \vzero\ $ holds pointwise for all $\vx\in \Gamma^\ast$.
	
	\item \label{item:pert-prop-step1-IH-c} There exist open sets $\Gamma_i \subsetcomp \Gamma^\ast$ for all $i=a,3,...,N$ such that
	\begin{itemize}
		\item the $\closure{\Gamma_i}$ are pairwise disjoint and 
		\item $\vz^\ast+\widetilde{\vz}_{A,k}(\vx) = \vz_i$ for all $\vx\in \Gamma_i$.
	\end{itemize} 
	
	\item \label{item:pert-prop-step1-IH-d} The sum $\vz^\ast + \widetilde{\vz}_{A,k}$ takes values in $\sU_{\vb^\ast}$ and its distance to $\partial\sU_{\vb^\ast}$ can be estimated by the distance of the $\vz_i$ to $\partial\sU_{\vb^\ast}$, more precisely 
	\begin{align}
		\vz^\ast + \widetilde{\vz}_{A,k}(\vx) &\in \sU_{\vb^\ast} & &\text{ for all }\vx\in \Gamma^\ast, \label{eq:pert-prop-step1-IH-d1} \\
		\dist (\vz^\ast + \widetilde{\vz}_{A,k}(\vx), \partial\sU_{\vb^\ast}) &\geq \mu \min_{i=a,3,...,N} \dist(\vz_i,\partial\sU_{\vb^\ast}) & & \text{ for all }\vx\in \Gamma^\ast. \label{eq:pert-prop-step1-IH-d2} 
	\end{align}
	
	\item \label{item:pert-prop-step1-IH-e} The following estimate holds
	$$ 
		\int_{\Gamma^\ast} \dist\big(\vz^\ast + \widetilde{\vz}_{A,k}(\vx),\sK_{\vb^\ast}\big) \dx < \frac{\ep}{2} + \int_{\Gamma_a} \dist(\vz_a,\sK_{\vb^\ast}) \dx + \sum_{i=3}^N \int_{\Gamma_i} \dist(\vz_i,\sK_{\vb^\ast}) \dx.
	$$
\end{enumerate}

Next we apply\footnote{Note that we already proved Lemma~\ref{lemma:pert-prop-step1} for $N=2$, see the induction basis above.} Lemma~\ref{lemma:pert-prop-step1} with $N=2$ to the situation where $\Gamma_a$, $\frac{\ep}{2}$, $\vz_a$ and $\left\{\left(\frac{\tau_1}{\tau_1+\tau_2},\vz_1\right) \cup \left(\frac{\tau_2}{\tau_1+\tau_2},\vz_2\right)\right\}$ play the role of $\Gamma^\ast$, $\ep$ and $\vz^\ast$, $\{(\tau_i,\vz_i)\}_{i=1,2}$, respectively. This way we find a sequence of oscillations $(\widetilde{\vz}_{B,\ell})_{\ell\in \N} \subset \Cc(\Gamma_a;\R^M)$ for each fixed $k\in \N$ with the following properties:
\begin{enumerate}
	\item \label{item:pert-prop-step1-2-a} The sequence $(\widetilde{\vz}_{B,\ell})_{\ell\in \N}$ converges to $\vzero$ with respect to $d$, i.e. $\widetilde{\vz}_{B,\ell} \mathop{\to}\limits^d \vzero$ as $\ell\to \infty$. 
\end{enumerate}	
For each fixed $\ell\in \N$ the following statements hold:
\begin{enumerate} \setcounter{enumi}{1}
	\item \label{item:pert-prop-step1-2-b} The linear system of PDEs $\ \Div \mA(\widetilde{\vz}_{B,\ell}) = \vzero\ $ holds pointwise for all $\vx\in \Gamma_a$.
	
	\item \label{item:pert-prop-step1-2-c} There exist open sets $\Gamma_1,\Gamma_2 \subsetcomp \Gamma_a$ such that
	\begin{itemize}
		\item $\closure{\Gamma_1} \cap \closure{\Gamma_2} = \emptyset$ and 
		\item $\vz_a+\widetilde{\vz}_{B,\ell}(\vx) = \vz_i$ for all $\vx\in \Gamma_i$, $i=1,2$.
	\end{itemize} 
	
	\item \label{item:pert-prop-step1-2-d} The sum $\vz_a + \widetilde{\vz}_{B,\ell}$ takes values in $\sU_{\vb^\ast}$ and its distance to $\partial\sU_{\vb^\ast}$ can be estimated by the distance of the $\vz_i$ to $\partial\sU_{\vb^\ast}$, more precisely 
	\begin{align}
		\vz_a + \widetilde{\vz}_{B,\ell}(\vx) &\in \sU_{\vb^\ast} & &\text{ for all }\vx\in \Gamma_a, \label{eq:pert-prop-step1-2-d1} \\
		\dist (\vz_a + \widetilde{\vz}_{B,\ell}(\vx), \partial\sU_{\vb^\ast}) &\geq \mu \min_{i=1,2} \dist(\vz_i,\partial\sU_{\vb^\ast}) & & \text{ for all }\vx\in \Gamma_a. \label{eq:pert-prop-step1-2-d2} 
	\end{align}
	
	\item \label{item:pert-prop-step1-2-e} The following estimate holds
	$$ 
		\int_{\Gamma_a} \dist\big(\vz_a + \widetilde{\vz}_{B,\ell}(\vx),\sK_{\vb^\ast}\big) \dx < \frac{\ep}{2} + \int_{\Gamma_1} \dist(\vz_1,\sK_{\vb^\ast}) \dx + \int_{\Gamma_2} \dist(\vz_2,\sK_{\vb^\ast}) \dx.
	$$
\end{enumerate}

Due to property~\ref{item:pert-prop-step1-2-a}, we find $\ell(k)\in \N$ for each $k\in \N$ such that $d(\widetilde{\vz}_{B,\ell(k)},\vzero) \leq \frac{1}{k}$. Now we can finally define 
$$
	(\widetilde{\vz}_k)_{k\in \N} := (\widetilde{\vz}_{A,k})_{k\in \N} + (\widetilde{\vz}_{B,\ell(k)})_{k\in \N}.
$$

The desired properties \ref{item:pert-prop-step1-a}-\ref{item:pert-prop-step1-e} of $(\widetilde{\vz}_k)_{k\in \N}$ can be readily checked using the properties \ref{item:pert-prop-step1-IH-a}-\ref{item:pert-prop-step1-IH-e} of $(\widetilde{\vz}_{A,k})_{k\in \N}$ and \ref{item:pert-prop-step1-2-a}-\ref{item:pert-prop-step1-2-e} of $(\widetilde{\vz}_{B,\ell(k)})_{k\in \N}$. Indeed properties \ref{item:pert-prop-step1-a} and \ref{item:pert-prop-step1-b} follow immediately. 

In order to verify \ref{item:pert-prop-step1-c}, we note that the $\closure{\Gamma_i}$, $i=1,...,N$, are pairwise disjoint. Moreover we observe that 
$$
	\vz^\ast + \widetilde{\vz}_k(\vx) = \vz^\ast + \widetilde{\vz}_{A,k}(\vx) + \widetilde{\vz}_{B,\ell(k)}(\vx) = \vz^\ast + \widetilde{\vz}_{A,k}(\vx) = \vz_i
$$
if $\vx\in \Gamma_i$ and $i=3,...,N$ because in this case $\vx\notin \Gamma_a$ and thus $\widetilde{\vz}_{B,\ell(k)}(\vx)= \vzero$. If $\vx\in \Gamma_i$ and $i=1,2$, then $\vx\in \Gamma_a$, and hence 
$$
	\vz^\ast + \widetilde{\vz}_k(\vx) = \vz^\ast + \widetilde{\vz}_{A,k}(\vx) + \widetilde{\vz}_{B,\ell(k)}(\vx) = \vz_a + \widetilde{\vz}_{B,\ell(k)}(\vx) =\vz_i.
$$ 

Property \ref{item:pert-prop-step1-d} can be shown in a similar fashion. Let us first consider $\vx\in \Gamma_a$. Then 
$$
	\vz^\ast + \widetilde{\vz}_k(\vx) = \vz_a + \widetilde{\vz}_{B,\ell(k)}(\vx) \in \sU_{\vb^\ast},
$$
see \eqref{eq:pert-prop-step1-2-d1}. Moreover we obtain 
\begin{align*}
	\dist (\vz^\ast + \widetilde{\vz}_k(\vx), \partial\sU_{\vb^\ast}) &= \dist (\vz_a + \widetilde{\vz}_{B,\ell(k)}(\vx), \partial\sU_{\vb^\ast}) \\
	&\geq \mu \min_{i=1,2} \dist(\vz_i,\partial\sU_{\vb^\ast}) \geq \mu \min_{i=1,...,N} \dist(\vz_i,\partial\sU_{\vb^\ast}),
\end{align*}
where we made use of \eqref{eq:pert-prop-step1-2-d2}. Next we look at the case $\vx\notin\Gamma_a$. Then 
$$
	\vz^\ast + \widetilde{\vz}_k(\vx) = \vz^\ast + \widetilde{\vz}_{A,k}(\vx) \in \sU_{\vb^\ast},
$$
according to \eqref{eq:pert-prop-step1-IH-d1}, and 
\begin{equation} \label{eq:pf-step1-temp1}
	\dist (\vz^\ast + \widetilde{\vz}_k(\vx), \partial\sU_{\vb^\ast}) = \dist (\vz^\ast + \widetilde{\vz}_{A,k}(\vx), \partial\sU_{\vb^\ast}) \geq \mu \min_{i=a,...,N} \dist(\vz_i,\partial\sU_{\vb^\ast}),
\end{equation}
due to \eqref{eq:pert-prop-step1-IH-d2}. From Lemma~\ref{lemma:app-distance-concave} (and Cor.~\ref{cor:app-boundary-vs-boundaryofinterior}) we infer 
\begin{equation} \label{eq:pf-step1-temp2}
	\dist(\vz_a,\partial\sU_{\vb^\ast}) \geq \frac{\tau_1}{\tau_1+\tau_2} \dist(\vz_1,\partial\sU_{\vb^\ast}) + \frac{\tau_2}{\tau_1+\tau_2} \dist(\vz_2,\partial\sU_{\vb^\ast}) \geq \min_{i=1,2} \dist(\vz_i,\partial\sU_{\vb^\ast}).
\end{equation}
The desired property \eqref{eq:pert-prop-step1-d2} follows from plugging \eqref{eq:pf-step1-temp2} into \eqref{eq:pf-step1-temp1}. 

Finally we check \ref{item:pert-prop-step1-e}. To this end, we compute 
\begin{align*}
	&\int_{\Gamma^\ast} \dist\big(\vz^\ast + \widetilde{\vz}_k(\vx),\sK_{\vb^\ast}\big) \dx \\
	&= \int_{\Gamma^\ast \setminus \Gamma_a} \dist\big(\vz^\ast + \widetilde{\vz}_{A,k}(\vx),\sK_{\vb^\ast}\big) \dx + \int_{\Gamma_a} \dist\big(\vz_a + \widetilde{\vz}_{B,\ell(k)}(\vx),\sK_{\vb^\ast}\big) \dx \\
	&= \int_{\Gamma^\ast} \dist\big(\vz^\ast + \widetilde{\vz}_{A,k}(\vx),\sK_{\vb^\ast}\big) \dx - \int_{\Gamma_a} \dist\big(\vz_a,\sK_{\vb^\ast}\big) \dx + \int_{\Gamma_a} \dist\big(\vz_a + \widetilde{\vz}_{B,\ell(k)}(\vx),\sK_{\vb^\ast}\big) \dx\\
	&< \ep + \sum_{i=1}^N \int_{\Gamma_i} \dist(\vz_i,\sK_{\vb^\ast}) \dx,
\end{align*}
where we used properties \ref{item:pert-prop-step1-IH-c} and \ref{item:pert-prop-step1-IH-e} of $\widetilde{\vz}_{A,k}$ and property \ref{item:pert-prop-step1-2-e} of $\widetilde{\vz}_{B,\ell(k)}$.
\end{proof}

\subsubsection{Step 2: Perturbations around Fixed Vectors $\vz^\ast$ and $\vb^\ast$} 

We continue with the following lemma which corresponds to \cite[Lemma~5.2.1]{Markfelder}. 

\begin{lemma} \label{lemma:pert-prop-step2} 
	Let $\Gamma^\ast\subset \R^n$ be an open and bounded domain and $\ep>0$. Let furthermore $\vb^\ast\in \sB$ and $\vz^\ast\in \sU_{\vb^\ast}$. Then there exists a sequence of oscillations $(\widetilde{\vz}_k)_{k\in \N} \subset \Cc(\Gamma^\ast;\R^M)$ with the following properties:
	\begin{enumerate}
		\item \label{item:pert-prop-step2-a} The sequence $(\widetilde{\vz}_k)_{k\in \N}$ converges to $\vzero$ with respect to $d$, i.e. $\widetilde{\vz}_k \mathop{\to}\limits^d \vzero$ as $k\to \infty$. 
	\end{enumerate}	
	For each fixed $k\in \N$ the following statements hold:
	\begin{enumerate} \setcounter{enumi}{1}
		\item \label{item:pert-prop-step2-b} The linear system of PDEs $\ \Div \mA(\widetilde{\vz}_k) = \vzero\ $ holds pointwise for all $\vx\in \Gamma^\ast$;
		
		\item \label{item:pert-prop-step2-c} The sum $\vz^\ast + \widetilde{\vz}_k$ takes values in $\sU_{\vb^\ast}$ and its distance to $\partial\sU_{\vb^\ast}$ can be estimated by the distance of the $\vz_i$ to $\partial\sU_{\vb^\ast}$, more precisely 
		\begin{align}
			\vz^\ast + \widetilde{\vz}_k(\vx) &\in \sU_{\vb^\ast} & &\text{ for all }\vx\in \Gamma^\ast, \label{eq:pert-prop-step2-d1} \\
			\dist (\vz^\ast + \widetilde{\vz}_k(\vx), \partial\sU_{\vb^\ast}) &\geq \frac{\ep}{8c |\Gamma^\ast|} \dist(\vz^\ast, \partial\sU_{\vb^\ast}) & & \text{ for all }\vx\in \Gamma^\ast, \label{eq:pert-prop-step2-d2} 
		\end{align}
		where $c$ is the bound from Lemma~\ref{lemma:uniform-boundedness-K}; 
		
		\item \label{item:pert-prop-step2-d} The following estimate holds
		$$ 
			\int_{\Gamma^\ast} \dist\big(\vz^\ast + \widetilde{\vz}_k(\vx),\sK_{\vb^\ast}\big) \dx < \ep .
		$$
	\end{enumerate}
\end{lemma}

\begin{proof}
Again the proof is analogous to \cite[Proof of Lemma~5.2.1]{Markfelder} and thus we only sketch the proof here.

We begin by applying Prop.~\ref{prop:geom-property-U} where $\vb^\ast$, $\vz^\ast$ and $\frac{\ep}{2|\Gamma^\ast|}$ play the role of $\vb$, $\vz$ and $\ep$ in Prop.~\ref{prop:geom-property-U}, respectively. This way we obtain $N\in \N$ with $N\geq 2$, and $(\tau_i,\vz_i)\in \R^+\times \sU_{\vb^\ast}$ for each $i=1,...,N$ which satisfy conditions \ref{item:geom-propU-a}-\ref{item:geom-propU-d} stated in Prop.~\ref{prop:geom-property-U}. 

Thus we can apply Lemma~\ref{lemma:pert-prop-step1}, where $\half\ep$ and $\half$ play the role of $\ep$ and $\mu$ therein, respectively. As a result we obtain a sequence of oscillations $(\widetilde{\vz}_k)_{k\in \N} \subset \Cc(\Gamma^\ast;\R^M)$ with properties \ref{item:pert-prop-step1-a}-\ref{item:pert-prop-step1-e} stated in Lemma~\ref{lemma:pert-prop-step1}. 

As items \ref{item:pert-prop-step2-a} and \ref{item:pert-prop-step2-b} of Lemma~\ref{lemma:pert-prop-step2} follow immediately from items \ref{item:pert-prop-step1-a} and \ref{item:pert-prop-step1-b} of Lemma~\ref{lemma:pert-prop-step1}, it remains to prove items \ref{item:pert-prop-step2-c} and \ref{item:pert-prop-step2-d}.

\begin{enumerate} \setcounter{enumi}{2}
	\item The fact that \eqref{eq:pert-prop-step2-d1}, follows immediately from \eqref{eq:pert-prop-step1-d1}. From \eqref{eq:pert-prop-step1-d2} and item \ref{item:geom-propU-b} in Prop.~\ref{prop:geom-property-U} we deduce 
	$$
		\dist (\vz^\ast + \widetilde{\vz}_k(\vx), \partial\sU_{\vb^\ast}) \geq \frac{1}{2} \min_{i=1,...,N} \dist(\vz_i,\partial\sU_{\vb^\ast}) \geq \frac{\ep}{8c |\Gamma^\ast|} \dist(\vz^\ast,\partial\sU_{\vb^\ast})
	$$ 
	for all $\vx\in \Gamma^\ast$ and thus \eqref{eq:pert-prop-step2-d2} holds. 
	
	\item We obtain from Lemma~\ref{lemma:pert-prop-step1}~\ref{item:pert-prop-step1-c} and \ref{item:pert-prop-step1-e}, and Prop.~\ref{prop:geom-property-U}~\ref{item:geom-propU-c} that
	$$
		\int_{\Gamma^\ast} \dist\big(\vz^\ast + \widetilde{\vz}_k(\vx),\sK_{\vb^\ast}\big) \dx < \frac{\ep}{2} + \sum_{i=1}^N \int_{\Gamma_i} \dist(\vz_i,\sK_{\vb^\ast}) \dx \leq \frac{\ep}{2} + \frac{\ep}{2|\Gamma^\ast|} \sum_{i=1}^N |\Gamma_i| \leq \ep.
	$$ 
\end{enumerate}
\end{proof}

\subsubsection{Step 3: Conclusion} 

Now we are ready to prove the Perturbation Property (Prop.~\ref{prop:pert-prop}) using Lemma~\ref{lemma:pert-prop-step2}.

\begin{proof}[Proof of the Perturbation Property (Prop.~\ref{prop:pert-prop})]
The proof is very similar to \cite[Proof of Prop.~5.1.11]{Markfelder} which is again the reason for skipping some details here. 

We observe that the case $\Gamma_0=\emptyset$ is trivial, so we may assume $\Gamma_0\neq \emptyset$. 

Let us first introduce a grid in $\R^n$ with size $h>0$: We consider open cubes $Q_{\valpha,h} \subset \R^n$ for $\valpha\in \Z^n$ with mid points $h\valpha$ and edge length $h$, i.e. 
$$
	Q_{\valpha,h} := h\valpha + \left(-\half h , \half h\right)^n .
$$

Fix $\ov{h}>0$ such that the volume of the ``frame'' $\Gamma_\tf:= \Gamma_0\setminus \widetilde{\Gamma}$ is sufficiently small, see below, where 
$$
	\widetilde{\Gamma} := \interior{\left(\bigcup_{\valpha\in \Z^n \text{ with } Q_{\valpha,\ov{h}}\subsetcomp\Gamma_0} \closure{Q_{\valpha,\ov{h}}}\right)}.
$$

Lemma~\ref{lemma:properties-dist} implies that the map $(\vz,\vb)\mapsto \dist(\vz,\partial(\sK_\vb)^\co)$ is uniformly continuous on $\R^M\times \sB$. As $\vz\in X_0$ and $\vb\in \Cb(\closure{\Gamma};\R^m)$ are continuous, we deduce that 
$$
	\vx\mapsto \dist(\vz(\vx),\partial(\sK_{\vb(\vx)})^\co)
$$ 
is continuous as well. According to \eqref{eq:defn-X0-subs}, $\sU_{\vb(\vx)}\neq \emptyset$ for all $\vx\in \Gamma$. Thus Cor.~\ref{cor:app-boundary-vs-boundaryofinterior} allows us to write $\dist(\vz(\vx),\partial \sU_{\vb(\vx)})$ instead of $\dist(\vz(\vx),\partial(\sK_{\vb(\vx)})^\co)$. Moreover \eqref{eq:defn-X0-subs} implies that 
$$
	\dist(\vz(\vx),\partial \sU_{\vb(\vx)}) > 0 \qquad \text{ for all } \vx\in \Gamma.
$$
The fact that $\widetilde{\Gamma}\subsetcomp \Gamma_0 \subset \Gamma$ and the above mentioned continuity yield existence of $\beta>0$ such that 
\begin{equation} \label{eq:aux07}
	\dist (\vz(\vx), \partial \sU_{\vb(\vx)}) > \beta \qquad \text{ for all } \vx\in \widetilde{\Gamma}.
\end{equation}

Moreover note that since $\Gamma_0$ is bounded, the functions $\vx\mapsto \vz(\vx)$ and $\vx\mapsto \vb(\vx)$ are uniformly continuous on $\closure{\Gamma_0}$. Together with the uniform continuity of the maps 
$$
	(\vz,\vb)\mapsto \dist(\vz,\sK_\vb) \qquad \text{ and } \qquad (\vz,\vb)\mapsto \dist(\vz,\partial\sU_\vb),
$$
see Lemma~\ref{lemma:properties-dist}, we deduce that there exists $q\in \N$ such that for $h:= \frac{\ov{h}}{q}$ the following holds: If $\vx_1$ and $\vx_2$ lie in the same cube $Q_{\valpha,h}$, then 
\begin{align}
	\Big| \dist\big(\vz(\vx_1) + \widehat{\vz} , \sK_{\vb(\vx_1)} \big) - \dist\big(\vz(\vx_2) + \widehat{\vz} , \sK_{\vb(\vx_2)} \big) \Big| &< \frac{\ep}{3|\widetilde{\Gamma}|}, \label{eq:unif-cont-K} \\
	\Big| \dist\big(\vz(\vx_1) + \widehat{\vz} , \partial\sU_{\vb(\vx_1)} \big) - \dist\big(\vz(\vx_2) + \widehat{\vz} , \partial\sU_{\vb(\vx_2)} \big) \Big| &< \frac{\ep}{48 c |\widetilde{\Gamma}|} \beta , \label{eq:unif-cont-boundaryU} 
\end{align}
for any $\widehat{\vz}\in \R^M$.

Let 
$$
	\mathcal{J} := \left\{ \valpha\in \Z^n \,\Big|\, Q_{\valpha,h} \subset \widetilde{\Gamma} \right\}.
$$
Note that $\mathcal{J}$ contains only finitely many points $\valpha$, and that 
$$
	\widetilde{\Gamma} = \interior{\left( \bigcup_{\valpha\in \mathcal{J}} \closure{Q_{\valpha,h}} \right)}
$$
because $\frac{\ov{h}}{h}\in \N$, and as a consequence 
\begin{equation} \label{eq:aux06}
	|\widetilde{\Gamma}| = |\mathcal{J}| \cdot |Q_{\valpha,h}| = |\mathcal{J}| h^n .
\end{equation}

Now for each $\valpha\in \mathcal{J}$ we apply Lemma~\ref{lemma:pert-prop-step2}, where $Q_{\valpha,h}$, $\frac{\ep}{3 |\mathcal{J}|}$, $\vb(h\valpha)$, $\vz(h\valpha)$ play the role of $\Gamma^\ast$, $\ep$, $\vb^\ast$ and $\vz^\ast$ respectively. For each $\valpha\in \mathcal{J}$ this yields a sequence of oscillations $(\widetilde{\vz}_{\valpha,k})_{k\in \N} \subset \Cc(Q_{\valpha,h};\R^M)$ with properties \ref{item:pert-prop-step2-a}-\ref{item:pert-prop-step2-d} stated in Lemma~\ref{lemma:pert-prop-step2}.

We define 
$$
	\vz_{\pert} := \vz + \sum_{\valpha\in \mathcal{J}} \widetilde{\vz}_{\valpha,k},
$$
where we fix a $k\in\N$ suitable large, such that \eqref{eq:pert-prop1} holds. Note that such a choice of $k\in \N$ is possible due to Lemma~\ref{lemma:pert-prop-step2}~\ref{item:pert-prop-step2-a}.

Let us next prove \eqref{eq:pert-prop2}. Indeed using \eqref{eq:dist-bounded}\footnote{Here we make use of the fact that $\vz_\pert$ satisfies \eqref{eq:defn-X0-subs} since $\vz_\pert\in X_0$. The latter is proven below.}, Lemma~\ref{lemma:pert-prop-step2}~\ref{item:pert-prop-step2-d} and \eqref{eq:unif-cont-K} we obtain 
\begin{align*}
	\I_{\Gamma_0} (\vz_\pert) &= \int_{\Gamma_\tf} \dist\big(\vz_\pert (\vx), \sK_{\vb(\vx)}\big) \dx + \sum_{\valpha\in \mathcal{J}} \int_{Q_{\valpha,h}} \dist\big(\vz_\pert (\vx), \sK_{\vb(\vx)}\big) \dx \\
	&\leq \int_{\Gamma_\tf} \dist\big(\vz_\pert (\vx), \sK_{\vb(\vx)}\big) \dx + \sum_{\valpha\in \mathcal{J}} \int_{Q_{\valpha,h}} \dist\big(\vz(h\valpha) + \widetilde{\vz}_{\valpha,k}(\vx), \sK_{\vb(h\valpha)}\big) \dx \\
	&\qquad + \sum_{\valpha\in \mathcal{J}} \int_{Q_{\valpha,h}} \Big| \dist\big(\vz(\vx) + \widetilde{\vz}_{\valpha,k}(\vx), \sK_{\vb(\vx)}\big) - \dist\big(\vz(h\valpha) + \widetilde{\vz}_{\valpha,k}(\vx), \sK_{\vb(h\valpha)}\big) \Big| \dx \\
	&< 2c |\Gamma_\tf| + \frac{\ep}{3 |\mathcal{J}|} |\mathcal{J}| + \frac{\ep}{3|\widetilde{\Gamma}|} h^n |\mathcal{J}| .
\end{align*} 
Having chosen $|\Gamma_\tf|$ sufficiently small, we conclude using \eqref{eq:aux06}.

Finally we show that $\vz_{\pert}\in X_0$ for each fixed $k\in \N$. To this end it suffices to check \eqref{eq:defn-X0-subs} since \eqref{eq:defn-X0-eq} and \eqref{eq:defn-X0-boundary} trivially follow from Lemma~\ref{lemma:pert-prop-step2}~\ref{item:pert-prop-step2-b} and the fact that all $\widetilde{\vz}_{\valpha,k}$ are compactly supported in $Q_{\valpha,h}$. 

Let $\vx\in \Gamma$. Then $\vz_\pert(\vx)= \vz(\vx)$ or there exists a unique $\valpha\in \mathcal{J}$ such that $\vx\in Q_{\valpha,h}$. In the former case, \eqref{eq:defn-X0-subs} holds trivially. In the latter case we have $\vz_{\pert}(\vx) = \vz(\vx) + \widetilde{\vz}_{\valpha,k}(\vx)$. Lemma~\ref{lemma:pert-prop-step2}~\ref{item:pert-prop-step2-c} ensures that 
\begin{align*}
	\vz(h\valpha) + \widetilde{\vz}_{\valpha,k}(\vx)&\in \sU_{\vb(h\valpha)}, \\
	\dist\big( \vz(h\valpha) + \widetilde{\vz}_{\valpha,k}(\vx) , \partial \sU_{\vb(h\valpha)}\big) &\geq \frac{\frac{\ep}{3|\mathcal{J}|}}{8c h^n} \dist(\vz(h\valpha), \partial \sU_{\vb(h\valpha)}) > \frac{\ep}{24 c |\widetilde{\Gamma}|} \beta, 
\end{align*} 
where in the last step we made use of \eqref{eq:aux07} and \eqref{eq:aux06}. Together with \eqref{eq:unif-cont-boundaryU} this implies 
\begin{align*}
	&\dist\big( \vz(\vx) + \widetilde{\vz}_{\valpha,k}(\vx) , \partial \sU_{\vb(\vx)}\big) \\
	&\geq \dist\big( \vz(h\valpha) + \widetilde{\vz}_{\valpha,k}(\vx) , \partial \sU_{\vb(h\valpha)}\big) \\
	& \qquad - \Big| \dist\big( \vz(\vx) + \widetilde{\vz}_{\valpha,k}(\vx) , \partial \sU_{\vb(\vx)}\big) - \dist\big( \vz(h\valpha) + \widetilde{\vz}_{\valpha,k}(\vx) , \partial \sU_{\vb(h\valpha)}\big) \Big| \\
	& > \frac{\ep}{24 c |\widetilde{\Gamma}|} \beta - \frac{\ep}{48 c |\widetilde{\Gamma}|} \beta = \frac{\ep}{48 c |\widetilde{\Gamma}|} \beta > 0,
\end{align*}
from which we infer that $\vz_\pert (\vx)\in \sU_{\vb(\vx)}$. 
\end{proof}

\section{The Compressible Euler Equations with Energy Inequality within the Framework Established in Sect.~\ref{sec:general}} \label{sec:euler-in-framework}

The goal of this section is to show how the barotropic compressible Euler equations \eqref{eq:euler-d}, \eqref{eq:euler-m} with energy inequality \eqref{eq:euler-E} fit into the framework established in Sect.~\ref{sec:general}.

\subsection{Tartar's Framework and Other Preliminaries} \label{subsec:eif-prelim}

We start by rewriting the PDEs \eqref{eq:euler-d}, \eqref{eq:euler-m} and inequality \eqref{eq:euler-E} as a linear system by introducting new unknowns $\mU,q,E$ and $\vF$, see also Remark~\ref{rem:tartars-framework}. This way we obtain
\begin{align}
	\partial_t \rho + \Div \vm &= 0, \label{eq:eulerlin-d-prime} \\
	\partial_t \vm + \Div (\mU + q\id) &= \vzero, \label{eq:eulerlin-m-prime} \\
	\partial_t E + \Div \vF &\leq 0 , \label{eq:eulerlin-E-prime}
\end{align}
which coincide with \eqref{eq:euler-d}, \eqref{eq:euler-m}, \eqref{eq:euler-E} via the replacements 
\begin{align} 
	\mU + q\id &= \frac{\vm\otimes \vm}{\rho} + p(\rho)\id , \label{eq:id-Uq}\\
	E&= \frac{|\vm|^2}{2\rho} + P(\rho) , \label{eq:id-E-prime}\\
	\vF&= \left(\frac{|\vm|^2}{2\rho} + P(\rho) + p(\rho)\right) \frac{\vm}{\rho} . \label{eq:id-F-prime}
\end{align}
This procedure goes back to \name{Tartar}~\cite{Tartar79}.

For technical reason we require the function $\mU$ to take values in $\sz$, where the latter denotes the space of symmetric, trace-less $2\times 2$-matrices. In other words $\mU$ equals the trace-less part of the matrix on the right-hand side of \eqref{eq:id-Uq} while $q$ represents its trace. 

In our convex integration approach we will not add oscillations in $\rho$. Hence with the notation of Sect.~\ref{sec:general}, $\rho$ (more precisely its temporal derivative $\partial_t\rho$, see below) will be contained in the right-hand side $\vb$ in \eqref{eq:lin-eq}. 

By taking the trace of \eqref{eq:id-Uq} we find
\begin{equation} \label{eq:id-q}
	q = \frac{|\vm|^2}{2\rho} + p(\rho) . 
\end{equation}
Hence as long as \eqref{eq:id-Uq} holds, \eqref{eq:id-E-prime} and \eqref{eq:id-F-prime} are equivalent to
\begin{align} 
	E&= q + P(\rho) - p(\rho) , \label{eq:id-E}\\
	\vF&= \frac{q + P(\rho)}{\rho} \vm , \label{eq:id-F}
\end{align}
respectively.

Keeping in mind that $\rho$ will be treated as a given function and due to \eqref{eq:id-E}, we rewrite the linear system \eqref{eq:eulerlin-d-prime}-\eqref{eq:eulerlin-E-prime} as 
\begin{align}
	\Div \vm &= -\partial_t \rho , \label{eq:eulerlin-d} \\
	\partial_t \vm + \Div (\mU + q\id) &= \vzero, \label{eq:eulerlin-m} \\
	\partial_t q + \Div \vF &\leq -\partial_t P(\rho) + \partial_t p(\rho) . \label{eq:eulerlin-E}
\end{align}

In order to see that \eqref{eq:eulerlin-d}-\eqref{eq:eulerlin-E} may be written in the form \eqref{eq:lin-eq}, we have to replace the independent variable $\vx$ in Sect.~\ref{sec:general} as a pair of time $t$ and position in space $\vx$. Thus we set $n=3$ ($t\in \R$, $\vx\in\R^2$) and $\Gamma\subset \R^3$ is now a space-time domain. The unknown is 
$$
	\vz=(\vm,\mU,q,\vF):\Gamma \to \phase ,
$$
where we have used the abbreviation 
$$
	\phase := \R^2\times \sz\times \R\times \R^2 .
$$
Consequently $M= \dim(\phase) = 7$ and one should keep in mind that $\phase\simeq\R^7$. Set $m=9$ and\footnote{In \eqref{eq:defn-A} and \eqref{eq:defn-b} $\vzero$ means $\vzero\in \R^2$.} $\mA:\phase \to \R^{9\times 3}$, 
\begin{equation} \label{eq:defn-A} 
	\mA(\vm,\mU,q,\vF) := \left( \begin{array}{cc} 0 & \vm^\trans \\ \vm & \mU + q\id \\ 0 & -\vm^\trans \\ -\vm & -\mU - q\id \\ q & \vF^\trans \\ 0 & \vzero^\trans \\ 0 & \vzero^\trans \end{array} \right) .
\end{equation}
Finally, for a given function $\rho\in C^1(\closure{\Gamma};\R^+)$ with $\rho,\partial_t\rho\in \Cb(\closure{\Gamma};\R^+)$ and $\inf\limits_{(t,\vx)\in \closure{\Gamma}} \rho(t,\vx)>0$, and a given constant $Q>0$ which satisfies $Q>\sup\limits_{(t,\vx)\in \closure{\Gamma}} p(\rho(t,\vx))$, we define 
\begin{equation} \label{eq:defn-b}
	\vb := \left( \begin{array}{c} -\partial_t\rho \\ \vzero \\ \partial_t\rho \\ \vzero \\ -\partial_t P(\rho) + \partial_t p(\rho) \\ \rho \\ Q \end{array} \right) ,
\end{equation}
such that $\vb\in \Cb(\closure{\Gamma};\R^9)$. 

Thus \eqref{eq:eulerlin-d}-\eqref{eq:eulerlin-E} together with 
\begin{equation} \label{eq:additional-inequalities}
	0 \leq \rho \qquad \text{ and } \qquad 0 \leq Q 
\end{equation} 
are equivalent to \eqref{eq:lin-eq}, where $\Div$ has to be replaced by the divergence in space-time $\Divtx$. Notice that both equations \eqref{eq:eulerlin-d}, \eqref{eq:eulerlin-m} are written as two inequalities, see also Remark~\ref{rem:equation-vs-inequality}. Moreover we have added the inequalities in \eqref{eq:additional-inequalities} both of which are trivially satisfied because of the properties of the given function $\rho$ and the given constant $Q$. The reason for adding the inequalities in \eqref{eq:additional-inequalities} is that the constitutive sets, see below, will depend on $\rho$ and $Q$, so in order to write them in the form $\sK_{\vb}$ as in Sect.~\ref{sec:general}, $\vb$ must depend on $\rho$ and $Q$. 

For each $\vb\in \sB= \closure{\vb(\closure{\Gamma})}$ we define the constitutive set $\sK_\vb$ by 
\begin{equation} \label{eq:defn-euler-K}
	\sK_{\vb}:= \sK_{\rho,Q}:=\left\{ (\vm,\mU,q,\vF)\in \phase \,\Big|\, \text{\eqref{eq:id-Uq} and \eqref{eq:id-F} hold, and } q\leq Q \right\}, 
\end{equation}
and we prefer to write $\sK_{\rho,Q}$ rather than $\sK_{\vb}$ in the sequel since the set only depends on the last two components of $\vb$ which are $\rho$ and $Q$. Because $Q$ is constant, we write $(\sK_{\rho,Q})_{\rho\in \sB}$ for the corresponding family of constitutive sets $(\sK_\vb)_{\vb\in \sB}$, which is another abuse of notation as we actually mean that $\rho$ is in the projection of $\sB$ onto the corresponding component. 

\begin{rem} \label{rem:rhoQ} 
	The only reason for introducing the bound $q\leq Q$ is to make each $\sK_{\rho,Q}$ compact, see Sect.~\ref{subsubsec:eif-sa-suitableK} below. The reason for requiring that $Q>\sup\limits_{(t,\vx)\in \closure{\Gamma}} p(\rho(t,\vx))$ is to ensure that $\sK_{\rho,Q}\neq \emptyset$ for all $\rho\in \sB$. Furthermore note that 
	$$
		\sup\limits_{(t,\vx)\in \closure{\Gamma}} p(\rho(t,\vx)) = \max_{\rho\in \sB} p(\rho).
	$$
	Finally we observe that for all $\rho\in \sB$ we have $\rho>0$ due to the requirement that the function $\rho$ is bounded from below. Hence there are no troubles in the denominator in \eqref{eq:id-Uq} and \eqref{eq:id-F}. 
\end{rem} 

The reader should also notice that if $\vz=(\vm,\mU,q,\vF)\in L^\infty(\Gamma;\phase)$ solves the linear system \eqref{eq:lin-eq} and satisfies \eqref{eq:UinK}, then $(\rho,\vm)$ is a solution\footnote{What is meant by a ``solution'' here is stated precisely in Thm.~\ref{thm:conv-int-euler+energy} below.} to \eqref{eq:euler-d}, \eqref{eq:euler-m}, \eqref{eq:euler-E}. 

Finally we consider the wave cone $\Lambda'$ for the Euler equations with energy inequality, see Defn.~\ref{defn:wave-cone}, and define the subset $\Lambda\subset \Lambda'$. We observe from \eqref{eq:defn-A} that 
$$
	\mA(\vm,\mU,q,\vF)\cdot \veta = \vzero \qquad \Leftrightarrow \qquad \left(\begin{array}{cc} 0 & \vm^\trans \\ \vm & \mU + q\id \\ q & \vF^\trans\end{array}\right)\cdot\veta = \vzero ,
$$ 
and thus 
\begin{equation*} 
	\Lambda' = \Bigg\{ (\vm,\mU,q,\vF)\in \phase\,\Big|\,\exists\veta\in \R^3\setminus \{\vzero\}\text{ with }\left(\begin{array}{cc} 0 & \vm^\trans \\ \vm & \mU + q\id \\ q & \vF^\trans\end{array}\right)\cdot\veta = \vzero \Bigg\} .
\end{equation*} 
In order to guarantee existence of suitable differential operators $\opL_{(\vm,\mU,q,\vF)}$ for all $(\vm,\mU,q,\vF)\in \Lambda$, we set
\begin{equation} \label{eq:defn-euler-wave-cone}
	\Lambda = \Bigg\{ (\vm,\mU,q,\vF)\in \phase\,\Big|\,\exists\veta\in \R^3\text{ with }\veta_\vx\neq\vzero \text{ and } \left(\begin{array}{cc} 0 & \vm^\trans \\ \vm & \mU + q\id \\ q & \vF^\trans\end{array}\right)\cdot\veta = \vzero \Bigg\} ,
\end{equation} 
where we write $\veta=(\eta_t,\veta_\vx)$ with $\eta_t\in \R$ and $\veta_\vx\in \R^2$. In other words $\eta_t$ and $\veta_\vx$ denote the temporal and spatial components of the space-time vector $\veta$, respectively. Notice that $\Lambda$ is a cone and that $\Lambda\subset \Lambda'$. The fact that we have shrinked the wave cone $\Lambda'$ to $\Lambda$ will ensure that suitable differential operators exist, see Remarks~\ref{rem:subset-of-wavecone} and \ref{rem:wave-cone-euler} below.

\subsection{The Structural Assumptions of Thm.~\ref{thm:conv-int}} \label{subsec:eif-struc-ass}

Next we proof that the structural assumptions of Thm.~\ref{thm:conv-int} hold. 

\subsubsection{Suitability of the Family of Constitutive Sets $(\sK_{\rho,Q})_{\rho\in\sB}$} \label{subsubsec:eif-sa-suitableK}

\begin{lemma} \label{lemma:euler-K-suitable}
	The family of constitutive sets $(\sK_{\rho,Q})_{\rho\in\sB}$ defined in \eqref{eq:defn-euler-K} is suitable in the sense of Defn.~\ref{defn:suitable-K}.
\end{lemma}

\begin{proof}
We show that the three properties in Defn.~\ref{defn:suitable-K} hold.

\begin{enumerate}
	\item Let $\rho\in \sB$ arbitrary. First we observe that 
	$$
		\sK_{\rho,Q} = K^{-1}_\rho (\mzero,\vzero) \cap \left\{ (\vm,\mU,q,\vF) \,\Big|\, q\leq Q\right\} ,
	$$
	where $K_\rho:\phase \to \R^{2\times 2} \times \R^2$, 
	$$
		K_\rho(\vm,\mU,q,\vF) = \left(\frac{\vm\otimes \vm}{\rho} - \mU + (p(\rho)-q)\id \ ,\ \vF - \frac{q+P(\rho)}{\rho}\vm \right) .
	$$ 
	Since $K_\rho$ is continuous, we deduce that $\sK_{\rho,Q}$ is closed. 
	
	Hence it remains to prove that $\sK_{\rho,Q}$ is bounded. Recall that taking the trace of \eqref{eq:id-Uq} leads to \eqref{eq:id-q}. Together with the fact that $q\leq Q$ this immediately yields bounds on $q$ and $\vm$. This in turn leads to bounds on $\mU$ and $\vF$ via \eqref{eq:id-Uq} and \eqref{eq:id-F}, respectively.
	
	\item Our proof of property~\ref{item:suitable-K-continuity} in Defn.~\ref{defn:suitable-K} is explicit but technical. First of all we fix $\ov{Q}$ such that $\max_{\rho\in \sB} p(\rho)< \ov{Q}<Q$, which is possible by assumption on $Q$, see also Remark~\ref{rem:rhoQ}. Next we observe that the conditions on the function $\rho$ guarantee existence of $R>0$ such that 
	\begin{equation} \label{eq:suit-K-b-bound-R-rho}
		\tfrac{1}{R}< \rho < R \qquad \text{ for all }\rho\in \sB.
	\end{equation}
	We may assume that 
	\begin{equation} \label{eq:suit-K-b-bound-R-rest}
		R > \max\left\{ 1, Q , \frac{1}{2\big(\ov{Q} - \max_{\rho\in \sB} p(\rho)\big)} , \max_{\rho\in \sB} \frac{P(\rho)}{\rho} \right\}.
	\end{equation}
	
	Now let $\ep>0$. Fix $\delta>0$ such that 
	\begin{align}
		|\rho_1 - \rho_2| &< \frac{\ep}{32 R^6} , \label{eq:suit-K-b-cont-1} \\
		\left|\frac{1}{\rho_1} - \frac{1}{\rho_2}\right| &< \min\left\{ \frac{Q-\ov{Q}}{2RQ}, \frac{\ep}{8R^4} \right\} , \label{eq:suit-K-b-cont-2} \\
		\big|p(\rho_1) - p(\rho_2)\big| &< \min\left\{ \ov{Q} - \max_{\rho\in \sB} p(\rho) , \frac{Q - \ov{Q}}{2} , \frac{\ep}{32 R^4}\right\} , \label{eq:suit-K-b-cont-3} \\ 
		\left|\frac{P(\rho_1)}{\rho_1} - \frac{P(\rho_2)}{\rho_2}\right| &< \frac{\ep}{8R} \label{eq:suit-K-b-cont-4}
	\end{align}
	for all $\rho_1,\rho_2\in \sB$ with $|\rho_1-\rho_2|< \delta$, which is possible due to uniform continuity of the map
	$$
		\rho \mapsto \left( \rho , \frac{1}{\rho} , p(\rho) , \frac{P(\rho)}{\rho} \right)
	$$
	on the compact set $\sB$.
	
	Let $\rho_1,\rho_2\in \sB$ with $|\rho_1-\rho_2|< \delta$, and $(\vm_1,\mU_1,q_1,\vF_1) \in \sK_{\rho_1,Q}$. In order to show that there exists $(\vm_2,\mU_2,q_2,\vF_2) \in \sK_{\rho_2,Q}$ with 
	\begin{equation} \label{eq:suit-K-b-inK}
		\Big| (\vm_1,\mU_1,q_1,\vF_1) - (\vm_2,\mU_2,q_2,\vF_2) \Big| < \ep ,
	\end{equation}
	we distinguish between two cases. But before we continue with them, let us note that \begin{align}
		\mU_1 &= \frac{\vm_1\otimes \vm_1}{\rho_1} - \frac{|\vm_1|^2}{2\rho_1} \id , \label{eq:suit-K-b-prop-U} \\
		q_1 &= \frac{|\vm_1|^2}{2\rho_1} + p(\rho_1) , \label{eq:suit-K-b-prop-q} \\
		\vF_1 &= \frac{q_1 + P(\rho_1)}{\rho_1} \vm_1 . \label{eq:suit-K-b-prop-F}
	\end{align}
	according to \eqref{eq:defn-euler-K}, and consequently
	\begin{equation} \label{eq:suit-K-b-bound-m}
		|\vm_1|^2 \leq 2\rho_1 q_1 \leq 2 R Q,
	\end{equation}
	where we used \eqref{eq:suit-K-b-bound-R-rho}.
	
	\textbf{Case 1:} Consider $q_1 \leq \ov{Q}$. We set 
	\begin{align}
		\vm_2 &:= \vm_1 , \label{eq:suit-K-b-defn-m-c1} \\
		\mU_2 &:= \frac{\vm_2\otimes \vm_2}{\rho_2} - \frac{|\vm_2|^2}{2\rho_2} \id , \label{eq:suit-K-b-defn-U-c1} \\
		q_2 &:= \frac{|\vm_2|^2}{2\rho_2} + p(\rho_2) , \label{eq:suit-K-b-defn-q-c1} \\
		\vF_2 &:= \frac{q_2 + P(\rho_2)}{\rho_2} \vm_2 . \label{eq:suit-K-b-defn-F-c1}
	\end{align}
	It is obvious that \eqref{eq:id-Uq} and \eqref{eq:id-F} hold, so in order to show that $(\vm_2,\mU_2,q_2,\vF_2) \in \sK_{\rho_2,Q}$, it suffices to check $q_2\leq Q$. Indeed, using \eqref{eq:suit-K-b-defn-q-c1}, \eqref{eq:suit-K-b-defn-m-c1}, \eqref{eq:suit-K-b-prop-q}, \eqref{eq:suit-K-b-bound-m}, \eqref{eq:suit-K-b-cont-2} and \eqref{eq:suit-K-b-cont-3}
	we obtain
	\begin{align*}
		q_2 &\leq q_1 + \frac{|\vm_1|^2}{2} \left|\frac{1}{\rho_1} - \frac{1}{\rho_2}\right| + \big|p(\rho_1) - p(\rho_2)\big| \\
		&\leq \ov{Q} + RQ \frac{Q-\ov{Q}}{2RQ} + \frac{Q-\ov{Q}}{2}\ = \ Q.
	\end{align*}
	
	Next we compute  
	\begin{align*}
		|\vm_1 - \vm_2| &= 0 , \\
		|\mU_1 - \mU_2| &\leq |\vm_1\otimes \vm_1| \left| \frac{1}{\rho_1} - \frac{1}{\rho_2} \right| + \frac{|\id|}{2} |\vm_1|^2 \left| \frac{1}{\rho_1} - \frac{1}{\rho_2} \right| < 4 R^2 \frac{\ep}{8R^4} \ < \ \ep, \\
		|q_1 - q_2| &\leq \frac{1}{2} |\vm_1|^2 \left| \frac{1}{\rho_1} - \frac{1}{\rho_2} \right| + \big|p(\rho_1) - p(\rho_2)\big| < R^2 \frac{\ep}{8R^4} + \frac{\ep}{32 R^4} \ < \ \ep, \\ 
		|\vF_1 - \vF_2| &\leq |\vm_1| R |q_1-q_2| + |\vm_1| Q \left| \frac{1}{\rho_1} - \frac{1}{\rho_2} \right| + |\vm_1| \left|\frac{P(\rho_1)}{\rho_1} - \frac{P(\rho_2)}{\rho_2}\right| \\
		& < 2R^2 \left(\frac{\ep}{8R^2} + \frac{\ep}{32 R^4}\right) + 2R^2 \frac{\ep}{8R^4} + 2R \frac{\ep}{8R} \ < \ \ep, 
	\end{align*}
	with the help of \eqref{eq:suit-K-b-defn-m-c1}-\eqref{eq:suit-K-b-defn-F-c1}, \eqref{eq:suit-K-b-prop-U}-\eqref{eq:suit-K-b-prop-F}, \eqref{eq:suit-K-b-cont-1}-\eqref{eq:suit-K-b-cont-4} and the bounds \eqref{eq:suit-K-b-bound-R-rho}, \eqref{eq:suit-K-b-bound-R-rest} and \eqref{eq:suit-K-b-bound-m}. Hence \eqref{eq:suit-K-b-inK} is satisfied. 
	
	\textbf{Case 2:} Now consider $\ov{Q} \leq q_1 \leq Q$. Notice that in this case 
	\begin{equation} \label{eq:suit-K-b-lowerbound-m}
		|\vm_1|^2 = 2\rho_1 \big(q_1 - p(\rho_1)\big) \geq 2\rho_1 \Big(\ov{Q} - \max_{\rho\in \sB} p(\rho)\Big) > \frac{2}{R} \Big(\ov{Q} - \max_{\rho\in \sB} p(\rho)\Big) > 0, 
	\end{equation}
	due to \eqref{eq:suit-K-b-prop-q} and \eqref{eq:suit-K-b-bound-R-rho}.
	
	We define 
	\begin{equation} \label{eq:suit-K-b-defn-m-c2}
		\vm_2 := \sqrt{\frac{\rho_2}{\rho_1} |\vm_1|^2 + 2\rho_2 \big(p(\rho_1) - p(\rho_2)\big)} \frac{\vm_1}{|\vm_1|} ,  
	\end{equation}
	and $(\mU_2,q_2,\vF_2)$ as in case 1, i.e. by \eqref{eq:suit-K-b-defn-U-c1}-\eqref{eq:suit-K-b-defn-F-c1}. Note that $\vm_2$ is well-defined because of \eqref{eq:suit-K-b-lowerbound-m} and 
	$$
		\frac{\rho_2}{\rho_1} |\vm_1|^2 + 2\rho_2 \big(p(\rho_1) - p(\rho_2)\big) \geq 2\rho_2 \left( \ov{Q} - \max_{\rho\in \sB} p(\rho) - \big|p(\rho_1) - p(\rho_2)\big|\right) > 0,
	$$
	which follows from \eqref{eq:suit-K-b-lowerbound-m} and \eqref{eq:suit-K-b-cont-3}.
	
	Similar to case 1 above, we compute 
	\begin{align*} 
		|\vm_1 - \vm_2| &\leq \left| |\vm_1| - \sqrt{\frac{\rho_2}{\rho_1} |\vm_1|^2 + 2\rho_2 \big(p(\rho_1) - p(\rho_2)\big)} \right| \\
		&\leq \sqrt{\frac{R}{2\big(\ov{Q} - \max_{\rho\in \sB} p(\rho)\big)}} \Big( 2R^3 |\rho_1-\rho_2| + 2 R \big|p(\rho_1) - p(\rho_2)\big| \Big) \\
		& < 2R^4 \frac{\ep}{32 R^6} + 2 R^2 \frac{\ep}{32 R^4} \ < \ \ep, \\
		|\mU_1 - \mU_2| &\leq 2R^2 \left| \frac{1}{\rho_1} - \frac{1}{\rho_2} \right| + 4R^2 |\vm_1 - \vm_2| + 2 \big|p(\rho_1) - p(\rho_2)\big| \\
		& < 2R^2 \frac{\ep}{32R^6} + 4R^2 \left( \frac{\ep}{16 R^2} + \frac{\ep}{16 R^2} \right) + 2 \frac{\ep}{32R^4} \ < \ \ep, \\
		|q_1 - q_2| & = \left| \frac{|\vm_1|^2}{2\rho_1} - \frac{1}{2\rho_2} \left(\frac{\rho_2}{\rho_1} |\vm_1|^2 + 2\rho_2 \big(p(\rho_1) - p(\rho_2)\big)\right)  + p(\rho_1) - p(\rho_2) \right| \ = \ 0 , \\ 
		|\vF_1 - \vF_2| &\leq 2 R \left|\frac{P(\rho_1)}{\rho_1} - \frac{P(\rho_2)}{\rho_2}\right| + 2 R^2 \left| \frac{1}{\rho_1} - \frac{1}{\rho_2} \right| + (R^2 + R) |\vm_1 - \vm_2| \\
		&< 2 R \frac{\ep}{8R} + 2 R^2 \frac{\ep}{32R^6} + 2R^2 \left( \frac{\ep}{16 R^2} + \frac{\ep}{16 R^2} \right) \ < \ \ep, 
	\end{align*}
	which means that \eqref{eq:suit-K-b-inK} holds. It remains to show that $(\vm_2,\mU_2,q_2,\vF_2) \in \sK_{\rho_2,Q}$. Like in case 1, \eqref{eq:suit-K-b-defn-U-c1}-\eqref{eq:suit-K-b-defn-F-c1} immediately yield \eqref{eq:id-Uq} and \eqref{eq:id-F}. We conclude using $q_2=q_1 \leq Q$. 
	
	\item Fix an arbitrary $\rho\in \sB$. We will show the following claim:
	\begin{equation} \label{eq:suit-K-c-statement}
		\text{If }(\vm,\mU,q,\vF)\in \interior{\big((\sK_{\rho,Q})^\co\big)}, \text{ then } \lambda_{\max} \left(\frac{\vm\otimes \vm}{\rho} - \mU + (p(\rho)-q)\id \right) < 0 ,
	\end{equation} 
	where $\lambda_{\max}(\mM)$ denotes the largest eigenvalue of a symmetric matrix $\mM\in \sym{2}$. Together with \eqref{eq:id-Uq}, the claim in \eqref{eq:suit-K-c-statement} immediately yields $\interior{\big((\sK_{\rho,Q})^\co\big)} \cap \sK_{\rho,Q} = \emptyset $ as desired. So it remains to prove \eqref{eq:suit-K-c-statement}.
	
	For fixed $\ov{q}\leq Q$ and $\ov{\vF}\in \R^2$ we set 
	$$
		\widetilde{\sK}_{\rho,\ov{q},\ov{\vF}}:=\left\{ (\vm,\mU,q,\vF)\in \phase \,\Big|\, \text{\eqref{eq:id-Uq} holds, } q= \ov{q} \text{ and } \vF=\ov{\vF}\right\}.
	$$
	Obviously we have 
	\begin{equation} \label{eq:suit-K-c-subset-auxK}
		\sK_{\rho,Q} \subset \bigcup_{\ov{q}\leq Q, \ov{\vF}\in \R^2} \widetilde{\sK}_{\rho,\ov{q},\ov{\vF}} \qquad \text{ and thus } \qquad (\sK_{\rho,Q})^\co \subset \Bigg(\bigcup_{\ov{q}\leq Q, \ov{\vF}\in \R^2} \widetilde{\sK}_{\rho,\ov{q},\ov{\vF}}\Bigg)^\co .
	\end{equation}
	According to \cite[Prop.~4.2.10]{Markfelder} it holds that
	\begin{equation} \label{eq:suit-K-c-aux01}
		\bigcup_{\ov{q}\leq Q, \ov{\vF}\in \R^2} (\widetilde{\sK}_{\rho,\ov{q},\ov{\vF}})^\co \subset \Bigg(\bigcup_{\ov{q}\leq Q, \ov{\vF}\in \R^2} \widetilde{\sK}_{\rho,\ov{q},\ov{\vF}}\Bigg)^\co
	\end{equation}
	and due to Lemma~\ref{lemma:app-euler} we have
	\begin{align}
		(\widetilde{\sK}_{\rho,\ov{q},\ov{\vF}})^\co = \bigg\{ (\vm,\mU,q,\vF) \in \phase \, \Big| \, &\lambda_{\max}\left(\frac{\vm\otimes \vm}{\rho} - \mU + (p(\rho)-q) \id \right) \leq 0 , \label{eq:suit-K-c-aux02} \\
		&\qquad \text{ and } (q,\vF)= (\ov{q},\ov{\vF})\quad\bigg\} \notag
	\end{align}
	for all $\ov{q}\in \R$ and all $\ov{\vF}\in \R^2$. Combining \eqref{eq:suit-K-c-aux01} with \eqref{eq:suit-K-c-aux02} we find 
	\begin{align}
		&\left\{ (\vm,\mU,q,\vF) \in \phase \, \Big| \, \lambda_{\max}\left(\frac{\vm\otimes \vm}{\rho} - \mU + (p(\rho)-q) \id \right) \leq 0 , \text{ and }q\leq Q\right\} \notag \\
		&\subset \Bigg(\bigcup_{\ov{q}\leq Q, \ov{\vF}\in \R^2} \widetilde{\sK}_{\rho,\ov{q},\ov{\vF}}\Bigg)^\co. \label{eq:suit-K-c-aux03}
	\end{align}
	Since the map $(\vm,\mU,q)\mapsto  \lambda_{\max}\left(\frac{\vm\otimes \vm}{\rho} - \mU + (p(\rho)-q) \id \right)$ is convex (see e.g. \cite[Lemma~4.3.4]{Markfelder}) and because the left-hand side of \eqref{eq:suit-K-c-aux03} obviously contains the set $\bigcup_{\ov{q}\leq Q, \ov{\vF}\in \R^2} \widetilde{\sK}_{\rho,\ov{q},\ov{\vF}}$, we even have equality in \eqref{eq:suit-K-c-aux03}. Using this in \eqref{eq:suit-K-c-subset-auxK}, we obtain 
	\begin{equation} \label{eq:suit-K-c-aux04}
		(\sK_{\rho,Q})^\co \subset \left\{ (\vm,\mU,q,\vF) \in \phase \, \Big| \, \lambda_{\max}\left(\frac{\vm\otimes \vm}{\rho} - \mU + (p(\rho)-q) \id \right) \leq 0 \right\} .
	\end{equation}
	Next, \cite[Lemma~5.1.1]{Markfelder} shows that in order to get the interior of the right-hand side of \eqref{eq:suit-K-c-aux04} we just have to replace ``$\leq$'' by ``$<$''. Thus \eqref{eq:suit-K-c-aux04} yields 
	$$
		\interior{\big((\sK_{\rho,Q})^\co\big)} \subset \left\{ (\vm,\mU,q,\vF) \in \phase \, \Big| \, \lambda_{\max}\left(\frac{\vm\otimes \vm}{\rho} - \mU + (p(\rho)-q) \id \right) < 0 \right\} ,
	$$
	which immediately shows \eqref{eq:suit-K-c-statement}.
\end{enumerate}
\end{proof}

\subsubsection{Existence of a Suitable Differential Operator} \label{subsubsec:eif-sa-suitableOperator}

\begin{lemma} \label{lemma:euler-operators} 
	Let $(\vm,\mU,q,\vF)\in\Lambda$. Then there exists a third order homogeneous differential operator 
	\begin{align*} 
		\opL_{(\vm,\mU,q,\vF)} &: C^\infty(\R^{3}) \to C^\infty(\R^{3};\phase) 
	\end{align*}
	which is suitable in the sense of Defn.~\ref{defn:suitable-operator}.
\end{lemma}

\begin{proof} 
The proof of Lemma~\ref{lemma:euler-operators} is similar to \cite[Prop.~4.4.1]{Markfelder}. 

According to \eqref{eq:defn-euler-wave-cone} there exists $\veta\in \R^3$ with $\veta_\vx\neq \vzero$ and 
\begin{equation} \label{eq:op-lambda}
	\left(\begin{array}{cc} 0 & \vm^\trans \\ \vm & \mU + q\id \\ q & \vF^\trans\end{array}\right)\cdot\veta = \vzero .
\end{equation}

Let us write $\veta=(a,b,c)$ and define 
\begin{align*}
	\opL_{m_1}[g] &:= \alpha \Big( \parthree{1}{1}{2}g + \parthree{2}{2}{2}g \Big) ,\\
	\opL_{m_2}[g] &:= -\alpha \Big( \parthree{1}{1}{1}g + \parthree{1}{2}{2}g \Big) ,\\
	\opL_{U_{11}}[g] &:= -2\alpha \parthree{t}{1}{2}g - \beta \Big( \parthree{1}{1}{1}g - \parthree{1}{2}{2}g \Big) - \gamma \Big( \parthree{1}{1}{2}g - \parthree{2}{2}{2}g \Big) ,\\
	\opL_{U_{12}}[g] &:= \alpha \Big( \parthree{t}{1}{1}g - \parthree{t}{2}{2}g \Big) - 2\beta \parthree{1}{1}{2}g - 2\gamma \parthree{1}{2}{2}g,\\
	\opL_q[g] &:= \beta \Big( \parthree{1}{1}{1}g + \parthree{1}{2}{2}g \Big) + \gamma \Big( \parthree{1}{1}{2}g + \parthree{2}{2}{2}g \Big) ,\\
	\opL_{F_1}[g] &:= -\beta \parthree{t}{1}{1}g - \gamma \parthree{t}{1}{2}g - \delta \parthree{1}{1}{2}g - \epsilon \parthree{1}{2}{2}g + \zeta \parthree{2}{2}{2}g , \\
	\opL_{F_2}[g] &:= -\beta \parthree{t}{1}{2}g - \gamma \parthree{t}{2}{2}g + \delta \parthree{1}{1}{1}g + \epsilon \parthree{1}{1}{2}g - \zeta \parthree{1}{2}{2}g , 
\end{align*}
where\footnote{Note that $b^2+c^2=|\veta_\vx|^2\neq 0$ and accordingly if $c=0$, then $b\neq 0$.} 
\begin{equation*}
	(\alpha,\beta,\gamma,\delta,\epsilon,\zeta) := \left\{ \begin{array}{ll}
		\left(-\frac{m_2}{b^3},\frac{q}{b^3},0,\frac{F_2}{b^3},0,0\right), & \text{ if }c=0 , \\
		\left(\frac{m_1}{c(b^2+c^2)},0,\frac{q}{c(b^2+c^2)},0,\frac{F_2}{c(b^2+c^2)},\frac{F_1}{c(b^2+c^2)}\right), & \text{ if }c\neq 0 .
	\end{array}\right. 
\end{equation*}
It is then straightforward to check that the PDEs
\begin{align*}
	\Div \opL_\vm[g] &= 0 , \\
	\partial_t \opL_\vm[g] + \Div \big( \opL_\mU[g] + \opL_q[g] \id\big) &= \vzero, \\
	\partial_t \opL_q[g] + \Div \opL_\vF[g] &=0 
\end{align*}
hold, i.e. item~\ref{item:suitable-operator-a} of Defn.~\ref{defn:suitable-operator} is satisfied.

Now let $g(t,\vx):= h\big((t,\vx)\cdot \veta\big)$. We first consider the case $c=0$, so we have $g(t,\vx):= h(at + bx)$. Thus we obtain
\begin{align*}
	\opL_{m_1}[g] &= 0 ,\\
	\opL_{m_2}[g] &= -\alpha b^3 \,h'''\big((t,\vx)\cdot \veta\big) = m_2 \,h'''\big((t,\vx)\cdot \veta\big),\\
	\opL_{U_{11}}[g] &= - \beta b^3 \,h'''\big((t,\vx)\cdot \veta\big) = -q \,h'''\big((t,\vx)\cdot \veta\big) ,\\
	\opL_{U_{12}}[g] &= \alpha ab^2 \,h'''\big((t,\vx)\cdot \veta\big) = -\frac{a m_2}{b} \,h'''\big((t,\vx)\cdot \veta\big),\\
	\opL_q[g] &= \beta b^3 \,h'''\big((t,\vx)\cdot \veta\big) = q \,h'''\big((t,\vx)\cdot \veta\big)\\
	\opL_{F_1}[g] &= -\beta ab^2 \,h'''\big((t,\vx)\cdot \veta\big) = -\frac{aq}{b} \,h'''\big((t,\vx)\cdot \veta\big) , \\
	\opL_{F_2}[g] &= \delta b^3 \,h'''\big((t,\vx)\cdot \veta\big) = F_2 \,h'''\big((t,\vx)\cdot \veta\big). 
\end{align*}
As $c=0$, \eqref{eq:op-lambda} implies
\begin{align*}
	b\,m_1 &=0, \\
	a\,m_1 + b\,(U_{11}+q) &=0, \\
	a\,m_2 + b\, U_{12} &=0, \\
	a\,q + b\, F_1 &=0,
\end{align*}
which in turn shows $m_1=0$ (since $\veta_\vx\neq \vzero$ and hence $b\neq 0$). Therefore item~\ref{item:suitable-operator-b} of Defn.~\ref{defn:suitable-operator} holds (with $\ell=3$).

Next we consider $c\neq 0$. We obtain in a similar fashion 
\begin{align*} 
	\opL_{m_1}[g] &= \Big[\alpha \big( b^2c + c^3 \big)\Big]h'''\big((t,\vx)\cdot \veta\big) \\
	&= m_1 \,h'''\big((t,\vx)\cdot \veta\big),\\
	\opL_{m_2}[g] &= \Big[-\alpha \big( b^3 + bc^2 \big)\Big]h'''\big((t,\vx)\cdot \veta\big) \\
	&= -\frac{bm_1}{c} \,h'''\big((t,\vx)\cdot \veta\big),\\
	\opL_{U_{11}}[g] &= \Big[-2\alpha abc - \beta \big( b^3 - bc^2 \big) - \gamma \big( b^2c - c^3 \big)\Big]h'''\big((t,\vx)\cdot \veta\big) \\
	&= \frac{-2abm_1 - (b^2-c^2)q}{b^2+c^2} \,h'''\big((t,\vx)\cdot \veta\big),\\
	\opL_{U_{12}}[g] &= \Big[\alpha \big( ab^2 - ac^2 \big) - 2\beta b^2c - 2\gamma bc^2\Big]h'''\big((t,\vx)\cdot \veta\big) \\
	&= \frac{a(b^2-c^2)m_1 - 2bc^2 q}{c(b^2+c^2)} \,h'''\big((t,\vx)\cdot \veta\big),\\
	\opL_q[g] &= \Big[\beta \big( b^3 + bc^2 \big) + \gamma \big( b^2c + c^3 \big)\Big]h'''\big((t,\vx)\cdot \veta\big) \\
	&= q \,h'''\big((t,\vx)\cdot \veta\big),\\
	\opL_{F_1}[g] &= \Big[-\beta ab^2 - \gamma abc - \delta b^2c - \epsilon bc^2 + \zeta c^3\Big]h'''\big((t,\vx)\cdot \veta\big) \\
	&= \frac{-abq - bcF_2 + c^2F_1}{b^2+c^2} \,h'''\big((t,\vx)\cdot \veta\big), \\
	\opL_{F_2}[g] &= \Big[-\beta abc - \gamma ac^2 + \delta b^3 + \epsilon b^2c - \zeta bc^2\Big]h'''\big((t,\vx)\cdot \veta\big) \\
	&= \frac{-acq + b^2 F_2 - bcF_1}{b^2+c^2} \,h'''\big((t,\vx)\cdot \veta\big).
\end{align*}
From \eqref{eq:op-lambda} we obtain
\begin{align*}
	b\,m_1 + c\,m_2&=0, \\
	a\,m_1 + b\,(U_{11}+q) + c\,U_{12} &=0, \\
	a\,m_2 + b\, U_{12} + c\,(-U_{11}+q) &=0, \\
	a\,q + b\, F_1 + c\, F_2&=0,
\end{align*}
which in turn implies by elementary algebraic manipulations that 
\begin{align*}
	2ab\, m_1 + (b^2-c^2) q + (b^2+c^2) U_{11} &= 0 ,\\
	-a(b^2-c^2) m_1 + 2bc^2 \,q + c(b^2+c^2) U_{12} &= 0.
\end{align*}
With these identities one can simply check that again item~\ref{item:suitable-operator-b} of Defn.~\ref{defn:suitable-operator} holds. 
\end{proof}

\begin{rem} \label{rem:wave-cone-euler}
	In the definition of the subset $\Lambda$ of the wave cone $\Lambda'$, we have posed the additional requirement that $\veta_\vx\neq \vzero$, see \eqref{eq:defn-euler-wave-cone}. This additional property is essential in the proof of Lemma~\ref{lemma:euler-operators}, see also Remark~\ref{rem:subset-of-wavecone}.
\end{rem}

\subsubsection{The $\Lambda$-Convex Hull of each Constitutive Set Coincides with its Convex Hull} \label{subsubsec:eif-sa-KLambda=Kco}

\begin{lemma} \label{lemma:euler-KLambda=Kco}
	It holds that 
	$$
		(\sK_{\rho,Q})^\Lambda = (\sK_{\rho,Q})^\co \qquad \text{ for all }\rho\in \sB.
	$$
\end{lemma}

\begin{rem} \label{rem:incompressible-approach}
	The convex integration approach which is elaborated in the current section is incompressible in the sense that we consider $\rho$ as a parameter, for which no oscillations are added. Lemma~\ref{lemma:euler-KLambda=Kco} shows that this is not a restriction in the sense that the hull does not increase if we would allow for oscillations in $\rho$. Indeed the latter would mean that we work with a larger wave cone $\widetilde{\Lambda}\supset \Lambda$. Thus it holds that $(\sK_{\rho,Q})^{\widetilde{\Lambda}}\supset (\sK_{\rho,Q})^\Lambda$. If the former hull was a strict superset of the latter, then using an incompressible method would be a severe restriction. However from Prop.~\ref{prop:app-Lconvex} we obtain $(\sK_{\rho,Q})^{\widetilde{\Lambda}}\subset (\sK_{\rho,Q})^\co$ and so we infer from Lemma~\ref{lemma:euler-KLambda=Kco} that 
	$$
		(\sK_{\rho,Q})^\Lambda = (\sK_{\rho,Q})^{\widetilde{\Lambda}} = (\sK_{\rho,Q})^\co.
	$$ 
\end{rem}

\begin{proof}[Proof of Lemma~\ref{lemma:euler-KLambda=Kco}]
Let $\rho\in \sB$ arbitrary. We prove that $\Lambda$ is complete with respect to $\sK_{\rho,Q}$, see Defn.~\ref{defn:app-complete-wc}. The claim follows then from Prop.~\ref{prop:app-complete-wc}. Hence we consider 
$$
	(\vm_1,\mU_1,q_1,\vF_1),(\vm_2,\mU_2,q_2,\vF_2)\in \sK_{\rho,Q}. 
$$
Choose $\veta=(\eta_t,\veta_\vx)\in \R^3$ by $\veta_\vx\in \R^2\setminus\{\vzero\}$ such that 
$$
	\veta_\vx \perp (\vm_1 - \vm_2) \qquad \text{ and } \qquad \eta_t = -\frac{\vm_2}{\rho} \cdot \veta_\vx .
$$ 
Then we compute
\begin{align*}
	&\left( \begin{array}{cc} 0 & (\vm_1- \vm_2)^\trans \\ \vm_1 - \vm_2 & \mU_1 - \mU_2 + (q_1-q_2)\id \\ q_1-q_2 & (\vF_1 - \vF_2)^\trans \end{array} \right) \cdot \veta \\
	&= \left( \begin{array}{cc} 0 & (\vm_1- \vm_2)^\trans \\ \vm_1 - \vm_2 & \frac{1}{\rho} \big( \vm_1\otimes \vm_1 - \vm_2 \otimes \vm_2 \big) \\ q_1-q_2 & \frac{1}{\rho} \big( q_1 \vm_1^\trans - q_2 \vm_2^\trans \big) + \frac{P(\rho)}{\rho} (\vm_1 - \vm_2)^\trans \end{array} \right) \cdot \veta \\
	&= \left( \begin{array}{cc} 0 & (\vm_1- \vm_2)^\trans \\ \vm_1 - \vm_2 & \frac{1}{\rho} \vm_1 (\vm_1 - \vm_2)^\trans + \frac{1}{\rho} (\vm_1-\vm_2) \vm_2^\trans \\ q_1-q_2 & \frac{q_1 + P(\rho)}{\rho}  (\vm_1 - \vm_2)^\trans + \frac{1}{\rho} ( q_1 - q_2 ) \vm_2^\trans \end{array} \right) \cdot \veta \\
	&= \left( \begin{array}{c} 0 \\ (\vm_1 - \vm_2) \left(\eta_t + \frac{1}{\rho} \vm_2\cdot \veta_\vx\right) \\ (q_1-q_2) \left( \eta_t + \frac{1}{\rho} \vm_2\cdot \veta_\vx \right) \end{array} \right)\ = \ \vzero .
\end{align*} 
Therefore $(\vm_1,\mU_1,q_1,\vF_1)-(\vm_2,\mU_2,q_2,\vF_2)\in \Lambda$ according to \eqref{eq:defn-euler-wave-cone}.
\end{proof}

\subsection{Conclusion: Convex Integration Theorem for Compressible Euler} \label{subsec:eif-ci}

As a consequence we obtain the following convex integration theorem in the context of the barotropic compressible Euler equations \eqref{eq:euler-d}, \eqref{eq:euler-m} with energy inequality \eqref{eq:euler-E}. 

\begin{thm} \label{thm:conv-int-euler+energy} 
	Let $\Gamma\subset\R^3$ be a Lipschitz space-time domain (open but not necessarily bounded). Assume there exist $\rho\in C^1(\closure{\Gamma};\R^+)$ with $\rho,\partial_t\rho\in \Cb(\closure{\Gamma};\R^+)$ and $\inf_{(t,\vx)\in \closure{\Gamma}} \rho(t,\vx)>0$, and $(\ov{\vm},\ov{\mU},\ov{q},\ov{\vF})\in C^1(\closure{\Gamma};\phase)$ with the following properties:
	\begin{itemize}
		\item The partial differential equations and inequalities  
		\begin{align}
			\partial_t \rho + \Div \ov{\vm} &= 0 , \label{eq:ci-subs-pde1}\\
			\partial_t \ov{\vm} + \Div (\ov{\mU} + \ov{q}\id) &= \vzero , \label{eq:ci-subs-pde2} \\
			\partial_t \big(\ov{q} + P(\rho) - p(\rho)\big) + \Div \ov{\vF} &\leq 0 \label{eq:ci-subs-pde3}
		\end{align}
		hold pointwise for all $(t,\vx)\in \Gamma$;
		
		\item There exists $Q>0$ with $Q> \sup_{(t,\vx)\in \closure{\Gamma}} p(\rho(t,\vx))$ such that 
		\begin{equation}  \label{eq:ci-subs} 
			(\ov{\vm},\ov{\mU},\ov{q},\ov{\vF})(t,\vx) \ \in\ \sU_{\rho(t,\vx),Q} \qquad \text{ for all }(t,\vx)\in \Gamma , 
		\end{equation} 
		where $\sU_{\rho,Q}:= \interior{\big((\sK_{\rho,Q})^\co\big)}$. 
	\end{itemize}
	
	Then there exist infinitely many $\vm\in L^\infty(\Gamma; \R^2)$ such that $(\rho,\vm)$ solve the barotropic Euler equations in the following sense: 
	\begin{align} 
		\iint_\Gamma \big[\rho \partial_t\phi + \vm\cdot \Grad\phi \big] \dx\dt - \int_{\partial \Gamma} \big[\rho \,n_t + \ov{\vm}\cdot \vn_\vx \big] \phi \dS_{t,\vx} &= 0 , \label{eq:ci-sol-pde1}\\ 
		\iint_\Gamma \bigg[\vm\cdot\partial_t\vphi + \frac{\vm\otimes\vm}{\rho} : \Grad\vphi + p(\rho) \Div\vphi\bigg] \dx\dt \ \;\quad \qquad & \notag \\
		- \int_{\partial\Gamma} \left[ \ov{\vm}\cdot \vphi\,n_t + (\ov{\mU}\cdot \vn_\vx)\cdot \vphi + \ov{q}\, \vphi \cdot \vn_\vx \right] \dS_{t,\vx} &= 0 , \label{eq:ci-sol-pde2} \\ 
		\iint_\Gamma \bigg[\left( \frac{|\vm|^2}{2\rho} + P(\rho)\right) \partial_t\psi + \left( \frac{|\vm|^2}{2\rho} + P(\rho) + p(\rho) \right) \frac{\vm}{\rho} \cdot \Grad\psi \bigg] \dx\dt \ \;\quad \qquad & \notag \\
		- \int_{\partial\Gamma} \left[ (\ov{q} + P(\rho) - p(\rho)) n_t + \ov{\vF}\cdot \vn_\vx \right] \psi \dS_{t,\vx} &\geq 0 , \label{eq:ci-sol-pde3} 
	\end{align} 
	for all test functions $(\phi,\vphi,\psi)\in \Cc(\closure{\Gamma};\R\times\R^2\times \R_0^+)$.
\end{thm}

Thm.~\ref{thm:conv-int-euler+energy} will be the core in the proof of Thm.~\ref{thm:local-max-diss}, see Prop.~\ref{prop:fan-subs=>inf-sol} below. However we would like to emphasize that it might be useful for other applications too. 

\begin{proof}
Thm.~\ref{thm:conv-int-euler+energy} is a simple consequence of Thm.~\ref{thm:conv-int}. As shown in Lemmas~\ref{lemma:euler-K-suitable}, \ref{lemma:euler-operators} and \ref{lemma:euler-KLambda=Kco}, the structural assumptions of Thm.~\ref{thm:conv-int} hold. Thus according to Thm.~\ref{thm:conv-int} there are infinitely many $(\vm,\mU,q,\vF)\in L^\infty(\Gamma;\phase)$ satisfying items~\ref{item:main-thm-sol1} and \ref{item:main-thm-sol2} of Thm.~\ref{thm:conv-int}. It is obvious that these two properties imply \eqref{eq:ci-sol-pde1}-\eqref{eq:ci-sol-pde3}.
\end{proof}

\subsection{A Subset $\sW_{\rho,Q}$ of $\sU_{\rho,Q}$} \label{subsec:eif-subset}

When we want to apply Thm.~\ref{thm:conv-int-euler+energy}, we have to show existence of suitable $(\ov{\vm},\ov{\mU},\ov{q},\ov{\vF})\in C^1(\closure{\Gamma};\phase)$ which satisfy \eqref{eq:ci-subs-pde1}-\eqref{eq:ci-subs}. In view of \eqref{eq:ci-subs} it would be helpful to know a simple explicit form of $\sU_{\rho,Q}$ (for any $\rho\in \sB$) which makes it easy to check if functions $(\ov{\vm},\ov{\mU},\ov{q},\ov{\vF})$ take values in $\sU_{\rho,Q}$ or not. However the author did not manage to find such an expression. Instead we will work with a subset $\sW_{\rho,Q}\subset \sU_{\rho,Q}$ for which we know a desired explicit expression. 

Throughout this subsection, let $\rho\in C^1(\closure{\Gamma};\R^+)$ with $\rho,\partial_t\rho\in \Cb(\closure{\Gamma};\R^+)$ and $\inf_{(t,\vx)\in \closure{\Gamma}} \rho(t,\vx)>0$ given (like in Thm.~\ref{thm:conv-int-euler+energy}) and $Q>0$ with $Q> \sup_{(t,\vx)\in \closure{\Gamma}} p(\rho(t,\vx))$. Moreover we consider\footnote{Here we use the same abuse of notation as in other parts of this paper. In particular we use the letter $\rho$ for both the given function $\rho\in C^1(\closure{\Gamma};\R^+)$ and for an element in $\sB$.} $\rho\in \sB:= \closure{\rho(\closure{\Gamma})}$ fixed. The reader should keep Remark~\ref{rem:rhoQ} in mind, in particular $\rho\in \sB$ implies $\rho>0$.

\subsubsection{Definition of the Subset $\sW_{\rho,Q}$} \label{subsubsec:eif-sub-defn}

Let us first introduce the set $\sV_{\rho,Q}$ by 
\begin{align} 
	\sV_{\rho,Q} := \Bigg\{ (\vm,\mU,q) \in \R^2\times \sz\times \R\,\Big|\, & \bullet \ \lambda_{\max}\left(\frac{\vm\otimes\vm}{\rho} - \mU + (p(\rho)-q)\id\right)< 0, \label{eq:V}  \\
	& \bullet \ q< Q \quad \Bigg\}\notag .
\end{align} 

In order to define $\sW_{\rho,Q}$, we need four vectors $\vsigma^j\in \R^2$ and auxiliary functions $A^j_{\rho,Q}$, $r^j_{\rho,Q}$ and $\vf^j_{\rho,Q}$ for $j=1,2,3,4$. We set 
\begin{align} 
	\vsigma^1 &:= \left(\begin{array}{r} 1 \\ 1 \end{array}\right) , & 
	\vsigma^2 &:= \left(\begin{array}{r} -1 \\ -1 \end{array}\right) , & 
	\vsigma^3 &:= \left(\begin{array}{r} 1 \\ -1 \end{array}\right) , & 
	\vsigma^4 &:= \left(\begin{array}{r} -1 \\ 1 \end{array}\right) , \label{eq:defn-sigma}
\end{align}
and define $A^j_{\rho,Q}:\sV_{\rho,Q}\to \R$ by
\begin{equation} \label{eq:defn-Aj} 
	A^j_{\rho,Q}(\vm,\mU,q) = \frac{\det\left(\frac{\vm\otimes\vm}{\rho} - \mU + (p(\rho)-q)\id \right)}{-\big( [\vsigma^j]_2 , -[\vsigma^j]_1 \big) \cdot \left(\frac{\vm\otimes\vm}{\rho} - \mU + (p(\rho)-q)\id \right) \cdot \left(\begin{array}{r} [\vsigma^j]_2 \\ -[\vsigma^j]_1 \end{array} \right)} .
\end{equation}
Note that $\lambda_{\max}\left(\frac{\vm\otimes\vm}{\rho} - \mU + (p(\rho)-q)\id\right)< 0$ implies that the matrix $\frac{\vm\otimes\vm}{\rho} - \mU + (p(\rho)-q)\id$ is negative definite. Hence, with the help of Lemma~\ref{lemma:app-trdet} we observe that
\begin{equation} \label{eq:Agr0}
	A^j_{\rho,Q} (\vm,\mU,q)> 0 \qquad \text{ for all } (\vm,\mU,q) \in \sV_{\rho,Q} \text{ and all } j =1,2,3,4 . 
\end{equation}
Moreover, we notice that $A^1_{\rho,Q} (\vm,\mU,q)=A^2_{\rho,Q} (\vm,\mU,q)$ and $A^3_{\rho,Q} (\vm,\mU,q)=A^4_{\rho,Q} (\vm,\mU,q)$ for all $(\vm,\mU,q) \in \sV_{\rho,Q}$.

Next we introduce the functions $r^j_{\rho,Q}:\sV_{\rho,Q}\to \R$ through
\begin{equation} \label{eq:defn-r}
	r^j_{\rho,Q}(\vm,\mU,q) := \frac{1}{2(Q-q)}\left( -\vm\cdot \vsigma^j + \sqrt{\big(\vm\cdot \vsigma^j\big)^2 + 4\rho A^j_{\rho,Q}(\vm,\mU,q) + 4\rho (Q-q)} \right) .
\end{equation}
Since $q<Q$ and due to \eqref{eq:Agr0}, the radicand is positive and hence the square root is well defined on $\sV_{\rho,Q}$. Furthermore 
$$
	\sqrt{\big(\vm\cdot \vsigma^j\big)^2 + 4\rho A^j_{\rho,Q}(\vm,\mU,q) + 4\rho (Q-q)} > \big|\vm\cdot \vsigma^j\big|
$$
and thus 
\begin{equation}\label{eq:rgr0}
	r^j_{\rho,Q} (\vm,\mU,q)>0 \qquad \text{ for all } (\vm,\mU,q) \in \sV_{\rho,Q} \text{ and all } j = 1,2,3,4 .
\end{equation}

Finally we define $\vf^j_{\rho,Q}:\sV_{\rho,Q}\to \R^2$ by 
\begin{equation} \label{eq:defn-f}
	\vf^j_{\rho,Q} (\vm,\mU,q) := \frac{A^j_{\rho,Q}(\vm,\mU,q)}{r^j_{\rho,Q}(\vm,\mU,q)} \vsigma^j ,
\end{equation} 
and note that for all $(\vm,\mU,q) \in \sV_{\rho,Q}$ and all $j = 1,2,3,4$, $\vf^j_{\rho,Q} (\vm,\mU,q)$ is a positive multiple of $\vsigma^j$ according to \eqref{eq:Agr0} and \eqref{eq:rgr0}.

Now we are ready to define the set $\sW_{\rho,Q}$ as follows
\begin{align*} 
	\sW_{\rho,Q} := \Bigg\{ (\vm,\mU,q,\vF) \in \phase\,\Big|\, & \bullet \ (\vm,\mU,q)\in \sV_{\rho,Q}, \\
	& \bullet \ \exists\,\kappa_1,\kappa_2,\kappa_3,\kappa_4\in \R^+\text{ with } \ \sum_{j=1}^4\kappa_j =1\ \text{ such that }  \\
	&\quad\qquad \vF-\frac{q+P(\rho)}{\rho}\vm = \sum_{j=1}^4\kappa_j \vf^j_{\rho,Q}(\vm,\mU,q) \quad \Bigg\} ,
\end{align*}
which is - in view of \eqref{eq:V} - equivalent to 
\begin{align} 
	\sW_{\rho,Q} = \Bigg\{ (\vm,\mU,q,\vF) \in \phase\,\Big|\, & \bullet \ \lambda_{\max}\left(\frac{\vm\otimes\vm}{\rho} - \mU + (p(\rho)-q)\id\right)< 0, \label{eq:WrhoQ} \\
	& \bullet \ \exists\,\kappa_1,\kappa_2,\kappa_3,\kappa_4\in \R^+\text{ with } \ \sum_{j=1}^4\kappa_j =1\ \text{ such that } \notag \\
	&\quad\qquad \vF-\frac{q+P(\rho)}{\rho}\vm = \sum_{j=1}^4\kappa_j \vf^j_{\rho,Q}(\vm,\mU,q), \notag \\
	& \bullet \ q< Q \quad \Bigg\}\notag 
\end{align}

\begin{rem} 
	The condition on $\vF$ in \eqref{eq:WrhoQ} is quite cumbersome. In the proof of Prop.~\ref{prop:fan-subs=>inf-sol} below, we will observe that if we only deal with finitely many points $(\vm,\mU,q,\vF)$, we can simply get rid of the condition on $\vF$ by choosing $Q$ large enough.
\end{rem}

We claim that $\sW_{\rho,Q}$ is a desired subset, which is the content of the following proposition.

\begin{prop} \label{prop:WsubsetU} 
	It holds that $\sW_{\rho,Q} \subset \sU_{\rho,Q}$.
\end{prop}

We prove Prop.~\ref{prop:WsubsetU} in three steps. The final proof can be found in Sect.~\ref{subsubsec:eif-sub-step3}.

\subsubsection{Step 1: Rigid Energy Flux} \label{subsubsec:eif-sub-step1}

\begin{lemma} \label{lemma:subset-step1}
	Let $(\vm,\mU,q,\vF)\in \phase$ such that 
	\begin{align*}
		\lambda_{\max}\left(\frac{\vm\otimes\vm}{\rho} - \mU + (p(\rho)-q)\id\right)&\leq 0, \\ 
		q &\leq Q , \\
		\vF &= \frac{q+P(\rho)}{\rho}\vm .  
	\end{align*}
	Then $(\vm,\mU,q,\vF)\in (\sK_{\rho,Q})^\co$. 
\end{lemma}

\begin{proof} 
Lemma~\ref{lemma:app-euler} and Prop.~\ref{prop:app-caratheodory} yield $N\in \N$ and $\big(\tau_i,(\vm_i,\mU_i)\big)\in \R^+ \times \R^2\times \sz$ for each $i=1,...,N$ such that 
\begin{align*}
	\sum_{i=1}^N \tau_i &= 1, \\
	\frac{\vm_i\otimes \vm_i}{\rho} + p(\rho) \id &= \mU_i + q \id \qquad \text{ for all }i=1,...,N, \text{ and } \\
	(\vm,\mU) &= \sum_{i=1}^N \tau_i (\vm_i,\mU_i) .
\end{align*}

Next we define $\vF_i := \frac{q+P(\rho)}{\rho}\vm_i$ for all $i=1,...,N$. Hence we have $(\vm_i,\mU_i,q,\vF_i)\in \sK_{\rho,Q}$ for all $i=1,...,N$. Moreover it holds that 
$$
	(\vm,\mU,q,\vF) = \sum_{i=1}^N \tau_i (\vm_i,\mU_i,q,\vF_i) ,
$$
and therefore Prop.~\ref{prop:app-caratheodory} implies that $(\vm,\mU,q,\vF)\in (\sK_{\rho,Q})^\co$. 
\end{proof}

\subsubsection{Step 2: Four Particular Energy Fluxes} \label{subsubsec:eif-sub-step2} 

\begin{lemma} \label{lemma:subset-step2}
	Let $(\vm,\mU,q,\vF)\in \phase$ such that $\exists j\in \{1,2,3,4\}$ with
	\begin{align*}
		\lambda_{\max}\left(\frac{\vm\otimes\vm}{\rho} - \mU + (p(\rho)-q)\id\right)&< 0, \\ 
		q &< Q , \\
		\vF - \frac{q+P(\rho)}{\rho}\vm &= \vf^j_{\rho,Q}(\vm,\mU,q).  
	\end{align*}
	Then $(\vm,\mU,q,\vF)\in (\sK_{\rho,Q})^\co$. 
\end{lemma}

\begin{proof}
We start with some definitions. We set\footnote{In \eqref{eq:defn-muM}-\eqref{eq:defn-ri-F} we have omitted the argument $(\vm,\mU,q)$ of the functions $r^j_{\rho,Q}$ and $A^j_{\rho,Q}$ for shortness of presentation. We will do this several times in the sequel. }
\begin{align} 
	\mu_{-} &:= \frac{1}{2r^j_{\rho,Q}} \left(\left( \frac{\rho}{r^j_{\rho,Q}} - \vm\cdot\vsigma^j\right) - \sqrt{4\rho A^j_{\rho,Q} + \left( \frac{\rho}{r^j_{\rho,Q}} - \vm\cdot\vsigma^j\right)^2}\right) , \label{eq:defn-muM} \\
	\mu_{+} &:= \frac{1}{2r^j_{\rho,Q}} \left(\left( \frac{\rho}{r^j_{\rho,Q}} - \vm\cdot\vsigma^j\right) + \sqrt{4\rho A^j_{\rho,Q} + \left( \frac{\rho}{r^j_{\rho,Q}} - \vm\cdot\vsigma^j\right)^2}\right) , \label{eq:defn-muP} 
\end{align}
as well as
\begin{align}
	\widehat{\vm} &:= r^j_{\rho,Q}\vsigma^j , \label{eq:defn-ri-m} \\
	\widehat{\mU} &:= \left(\begin{array}{cc}([\vm]_1 [\vsigma^j]_1 - [\vm]_2 [\vsigma^j]_2) \frac{r^j_{\rho,Q}}{\rho} & [\vsigma^j]_1 [\vsigma^j]_2\\ {[\vsigma^j]_1} [\vsigma^j]_2 &- ([\vm]_1 [\vsigma^j]_1 - [\vm]_2 [\vsigma^j]_2) \frac{r^j_{\rho,Q}}{\rho}\end{array}\right), \label{eq:defn-ri-U} \\
	\widehat{q} &:= 1 , \label{eq:defn-ri-q} \\
	\widehat{\vF} &:= \frac{\vm}{\rho} + \left[\frac{q + P(\rho)}{\rho} r^j_{\rho,Q} + \frac{1}{\rho} \left( \frac{\rho}{r^j_{\rho,Q}} - \vm\cdot\vsigma^j\right)\right] \vsigma^j . \label{eq:defn-ri-F}
\end{align}

Notice that 
\begin{equation} \label{eq:mu-neg-pos}
	\mu_{-}<0 \quad \text{ and } \quad \mu_{+}>0 .
\end{equation}
Indeed, because $A^j_{\rho,Q}(\vm,\mU,q)>0$ according to \eqref{eq:Agr0}, we have 
\begin{equation*}
	\sqrt{4\rho A^j_{\rho,Q}(\vm,\mU,q) + \left( \frac{\rho}{r^j_{\rho,Q}(\vm,\mU,q)} - \vm\cdot\vsigma^j\right)^2} > \left| \frac{\rho}{r^j_{\rho,Q}(\vm,\mU,q)} - \vm\cdot\vsigma^j\right|.
\end{equation*}
Together with the fact that $r^j_{\rho,Q}(\vm,\mU,q)>0$ (see \eqref{eq:rgr0}), this implies \eqref{eq:mu-neg-pos}. 

Next we define 
\begin{align*}
	(\vm_1,\mU_1,q_1,\vF_1) := (\vm,\mU,q,\vF) + \mu_- (\widehat{\vm},\widehat{\mU},\widehat{q},\widehat{\vF}) , \\
	(\vm_2,\mU_2,q_2,\vF_2) := (\vm,\mU,q,\vF) + \mu_+ (\widehat{\vm},\widehat{\mU},\widehat{q},\widehat{\vF}) , 
\end{align*}
and 
\begin{align*}
	\tau_1 &:= \frac{\mu_+}{\mu_+ - \mu_-}, & \tau_2 &:= -\frac{\mu_-}{\mu_+ - \mu_-} .
\end{align*}
We observe from \eqref{eq:mu-neg-pos} that $\tau_1,\tau_2>0$. Since $\tau_1+\tau_2 = 1$, we even have $\tau_1,\tau_2\in (0,1)$. Moreover a straightforward computation shows 
$$
	(\vm,\mU,q,\vF) = \tau_1 (\vm_1,\mU_1,q_1,\vF_1) + \tau_2 (\vm_2,\mU_2,q_2,\vF_2) .
$$
Therefore the proof of Lemma~\ref{lemma:subset-step2} is finished as soon as we have proven that 
\begin{equation} \label{eq:pointsinKco}
	(\vm_1,\mU_1,q_1,\vF_1), (\vm_2,\mU_2,q_2,\vF_2)\ \in\ (\sK_{\rho,Q})^\co.
\end{equation}

We compute 
\begin{align} 
	&\frac{\vm_{1}\otimes\vm_{1}}{\rho} - \mU_{1} + (p(\rho)-q_{1})\id \notag\\
	&= \frac{\vm\otimes\vm}{\rho} - \mU + (p(\rho)-q)\id + \frac{r^j_{\rho,Q}}{\rho} \left[(\mu_-)^2 r^j_{\rho,Q} - \mu_- \left( \frac{\rho}{r^j_{\rho,Q}} -\vm\cdot\vsigma^j\right)\right] \vsigma^j \otimes \vsigma^j , \label{eq:M_1} \\
	&\frac{\vm_{2}\otimes\vm_{2}}{\rho} - \mU_{2} + (p(\rho)-q_{2})\id \notag\\
	&= \frac{\vm\otimes\vm}{\rho} - \mU + (p(\rho)-q)\id + \frac{r^j_{\rho,Q}}{\rho} \left[(\mu_+)^2 r^j_{\rho,Q} - \mu_+ \left( \frac{\rho}{r^j_{\rho,Q}} -\vm\cdot\vsigma^j\right)\right] \vsigma^j \otimes \vsigma^j . \label{eq:M_2}
\end{align}
Furthermore from \eqref{eq:defn-muM} and \eqref{eq:defn-muP} we obtain 
\begin{equation} \label{eq:solution-mu}
	(\mu_\pm)^2 r^j_{\rho,Q} - \mu_\pm \left( \frac{\rho}{r^j_{\rho,Q}} -\vm\cdot\vsigma^j\right) = \frac{\rho A^j_{\rho,Q}}{r^j_{\rho,Q}} .
\end{equation}
Plugging \eqref{eq:solution-mu} into \eqref{eq:M_1} and \eqref{eq:M_2}, we find after some elementary manipulations and using \eqref{eq:defn-Aj} that 
\begin{align}
	\tr \left(\frac{\vm_{1}\otimes\vm_{1}}{\rho} - \mU_{1} + (p(\rho)-q_{1})\id\right) &= \tr \left(\frac{\vm_{2}\otimes\vm_{2}}{\rho} - \mU_{2} + (p(\rho)-q_{2})\id\right) \notag \\
	&= \tr\left(\frac{\vm\otimes\vm}{\rho} - \mU + (p(\rho)-q)\id \right) + 2 A^j_{\rho,Q} , \label{eq:tr_12-prime} \\
	\det \left(\frac{\vm_{1}\otimes\vm_{1}}{\rho} - \mU_{1} + (p(\rho)-q_{1})\id\right) &= \det \left(\frac{\vm_{2}\otimes\vm_{2}}{\rho} - \mU_{2} + (p(\rho)-q_{2})\id\right) = 0 . \label{eq:det_12}
\end{align}
Combining \eqref{eq:tr_12-prime} with \eqref{eq:defn-Aj} and Lemma~\ref{lemma:app-computationoftrace}, we obtain
\begin{equation} \label{eq:tr_12}
	\tr \left(\frac{\vm_{1}\otimes\vm_{1}}{\rho} - \mU_{1} + (p(\rho)-q_{1})\id\right) = \tr \left(\frac{\vm_{2}\otimes\vm_{2}}{\rho} - \mU_{2} + (p(\rho)-q_{2})\id\right) \leq 0.
\end{equation}
Lemma~\ref{lemma:app-trdet} and \eqref{eq:tr_12}, \eqref{eq:det_12} yield 
\begin{equation} \label{eq:points-lambdamax}
	\lambda_{\max} \left(\frac{\vm_{1}\otimes\vm_{1}}{\rho} - \mU_{1} + (p(\rho)-q_{1})\id\right) = \lambda_{\max} \left(\frac{\vm_{2}\otimes\vm_{2}}{\rho} - \mU_{2} + (p(\rho)-q_{2})\id\right) \leq 0.
\end{equation} 

Next we show that 
\begin{equation} \label{eq:points-q}
	q_1 \leq Q, \qquad \text{ and }\qquad q_2 \leq Q .
\end{equation}
We begin noticing that \eqref{eq:defn-ri-q} and \eqref{eq:mu-neg-pos} obviously imply that 
$$
	q_1 = q + \mu_- \widehat{q} < q < Q ,
$$
i.e. the estimate on $q_1$ in \eqref{eq:points-q} holds. Furthermore we observe that the definition of $r^j_{\rho,Q}$ (see \eqref{eq:defn-r}) implies 
\begin{equation} \label{eq:aux09} 
	\left(r^j_{\rho,Q}\right)^2 (Q-q)^2 + r^j_{\rho,Q} (Q-q) \vm\cdot \vsigma^j - \rho A^j_{\rho,Q} - \rho (Q-q) = 0.
\end{equation}
Next we notice that \eqref{eq:aux09} is equivalent to 
\begin{equation} \label{eq:aux10} 
	4\rho A^j_{\rho,Q} + \left( \frac{\rho}{r^j_{\rho,Q}} -\vm\cdot\vsigma^j\right)^2 = \left(2r^j_{\rho,Q} (Q-q) -  \left( \frac{\rho}{r^j_{\rho,Q}} -\vm\cdot\vsigma^j\right)\right)^2 . 
\end{equation} 
Note that 
\begin{align*}
	&2r^j_{\rho,Q} (Q-q) -  \left( \frac{\rho}{r^j_{\rho,Q}} -\vm\cdot\vsigma^j\right) \\
	&= \frac{1}{r^j_{\rho,Q}(Q-q)} \left(2\left(r^j_{\rho,Q}\right)^2 (Q-q)^2 +  r^j_{\rho,Q}(Q-q) \vm\cdot\vsigma^j - \rho (Q-q)\right) \\
	&\geq \frac{1}{r^j_{\rho,Q}(Q-q)} \left(\left(r^j_{\rho,Q}\right)^2 (Q-q)^2 +  r^j_{\rho,Q}(Q-q) \vm\cdot\vsigma^j - \rho A^j_{\rho,Q} - \rho (Q-q)\right) \ = \ 0,
\end{align*}
according to \eqref{eq:aux09}. Hence by taking the square root in equation \eqref{eq:aux10} we obtain after some elementary manipulations and using the definition of $\mu_+$ (see \eqref{eq:defn-muP}), that $\mu_+ = Q-q$. Thus \eqref{eq:defn-ri-q} yields $q_2 = q+\mu_+ \widehat{q} = Q$, i.e. the proof of \eqref{eq:points-q} is finished.

Finally we prove that 
\begin{equation} \label{eq:points-F}
	\vF_1 = \frac{q_1 + P(\rho)}{\rho} \vm_1, \qquad \text{ and }\qquad \vF_2 = \frac{q_2 + P(\rho)}{\rho} \vm_2.
\end{equation}
Using \eqref{eq:defn-ri-m}, \eqref{eq:defn-ri-q}, \eqref{eq:defn-ri-F}, \eqref{eq:solution-mu} and \eqref{eq:defn-f} we compute 
\begin{align*}
	&(\vF + \mu_\pm \widehat{\vF} ) - \frac{(q + \mu_\pm \widehat{q}) + P(\rho)}{\rho} (\vm + \mu_\pm \widehat{\vm}) \\
	&= \vF - \frac{q + P(\rho)}{\rho} \vm - \frac{1}{\rho} \left[ (\mu_\pm)^2 r^j_{\rho,Q} - \mu_\pm \left( \frac{\rho}{r^j_{\rho,Q}} -\vm\cdot\vsigma^j\right) \right] \vsigma^j \\
	&= \vf^j_{\rho,Q}(\vm,\mU,q) - \vf^j_{\rho,Q}(\vm,\mU,q) \ = \ \vzero,
\end{align*} 
which immediately yields \eqref{eq:points-F}.

Because of \eqref{eq:points-lambdamax}, \eqref{eq:points-q} and \eqref{eq:points-F}, $(\vm_1,\mU_1,q_1,\vF_1)$ and $(\vm_2,\mU_2,q_2,\vF_2)$ fulfill the assumptions of Lemma~\ref{lemma:subset-step1}. Thus we obtain \eqref{eq:pointsinKco} which finishes the proof. 
\end{proof}

\subsubsection{Step 3: Final Proof of Prop.~\ref{prop:WsubsetU}} \label{subsubsec:eif-sub-step3} 

\begin{proof}[Proof of Prop.~\ref{prop:WsubsetU}] 
Let $(\vm,\mU,q,\vF)\in \sW_{\rho,Q}$. Define
$$
	\vF_j := \frac{q + P(\rho)}{\rho}\vm + \vf^j_{\rho,Q}(\vm,\mU,q)
$$
for $j=1,2,3,4$. Then obviously $(\vm,\mU,q,\vF_j)$ satisfies the assumptions of Lemma~\ref{lemma:subset-step2} for each $j=1,2,3,4$. Hence $(\vm,\mU,q,\vF_j)\in (\sK_{\rho,Q})^\co$ for all $j=1,2,3,4$. Since 
$$
	(\vm,\mU,q,\vF) = \sum_{j=1}^4 \kappa_j (\vm,\mU,q,\vF_j)
$$
with the $\kappa_j\in \R^+$ from \eqref{eq:WrhoQ}, we deduce $(\vm,\mU,q,\vF)\in (\sK_{\rho,Q})^\co$. Thus we have shown that $\sW_{\rho,Q}\subset (\sK_{\rho,Q})^\co$.

It is obvious that the mappings 
\begin{align*} 
	(\vm,\mU,q)\mapsto \lambda_{\max}\left(\frac{\vm\otimes\vm}{\rho} - \mU + (p(\rho)-q)\id\right) \qquad\text{ and }\qquad(\vm,\mU,q)\mapsto \vf^j_{\rho,Q}(\vm,\mU,q)
\end{align*}	
are continuous on $\sV_{\rho,Q}$ (for all $j=1,2,3,4$). This implies that $\sW_{\rho,Q}$ is open. Hence we even have $\sW_{\rho,Q}\subset \interior{\big((\sK_{\rho,Q})^\co\big)}= \sU_{\rho,Q}$.
\end{proof}

\section{Convex Integration in the Context of Riemann Data} \label{sec:riemann-ci} 

In this section we apply convex integration in the context of Riemann data, i.e. initial data of the form \eqref{eq:riemann-init}. The procedure that we use is similar to the one developed by \name{Chiodaroli}-\name{De~Lellis}-\name{Kreml}~\cite{ChiDelKre15}. However we will make use of convex integration for compressible Euler \emph{with energy inequality}, i.e. Thm.~\ref{thm:conv-int-euler+energy}. 

\subsection{Fan Subsolutions} \label{subsec:rci-fansubs}

We begin with the definition of a fan partition (see also \cite[Defn.~3.3]{ChiDelKre15}) and a fan subsolution.

\begin{defn} \label{defn:fan-part}
	Let $\mu_0<\mu_1<\mu_2<\mu_3$ be real numbers. A \emph{fan partition} of the space-time domain $(0,\infty)\times \R^2$ is a collection of 5 open sets $\Gamma_-,\Gamma_1,\Gamma_2,\Gamma_3,\Gamma_+$ which satisfy
	\begin{align*}
		\Gamma_- &= \left\{ (t,\vx)\in (0,\infty)\times \R^2\,\big|\, y< \mu_0 t \right\} , \\
		\Gamma_i\  &= \left\{ (t,\vx)\in (0,\infty)\times \R^2\,\big|\, \mu_{i-1}t <y< \mu_i t \right\} \qquad \text{ for }i=1,2,3, \\
		\Gamma_- &= \left\{ (t,\vx)\in (0,\infty)\times \R^2\,\big|\, y> \mu_3 t \right\} .
	\end{align*} 
\end{defn}

\begin{rem} 
	We would like to remark that the number of sets $\Gamma_i$ in Defn.~\ref{defn:fan-part} is tailored to the proof of Thm.~\ref{thm:local-max-diss}, see Sect.~\ref{subsec:rci-proof-of-main-thm} below. Prop.~\ref{prop:fan-subs=>inf-sol} below is analogously true for any other (finite) number of sets $\Gamma_-,\Gamma_1,...,\Gamma_{n^\ast},\Gamma_+$ with $n^\ast\in \N$. 
\end{rem}

\begin{defn} \label{defn:fan-subs}
	An \emph{admissible fan subsolution} to the initial value problem for compressible Euler \eqref{eq:euler-d}, \eqref{eq:euler-m}, \eqref{eq:init} with energy inequality \eqref{eq:euler-E} and Riemann initial data \eqref{eq:riemann-init} is a tuple
	$$
		(\rho,\ov{\vm},\ov{\mU},\ov{q},\ov{\vF})\in L^\infty ((0,\infty)\times \R^2; \R^+\times \phase) 
	$$ 
	of piecewise constant functions which satisfies the following properties.
	\begin{enumerate}
		\item There exists a fan partition $\Gamma_-,\Gamma_1,\Gamma_2,\Gamma_3,\Gamma_+$ of $(0,\infty)\times\R^2$ and constants 
		$$
			(\rho_i,\vm_i,\mU_i,q_i,\vF_i) \in \R^+\times \phase \qquad \text{ for any } i=1,2,3.
		$$
		such that 
		\begin{equation} \label{eq:fan-subs-pwc}
			(\rho,\ov{\vm},\ov{\mU},\ov{q},\ov{\vF}) = \left\{ \begin{array}{ll}
				(\rho_-,\vm_-,\mU_-,q_-,\vF_-) & \text{ if } (t,\vx)\in \Gamma_-, \\
				(\rho_i\ ,\vm_i\ ,\mU_i\ ,q_i\ ,\vF_i\ ) & \text{ if } (t,\vx)\in \Gamma_i\ , \ i=1,2,3, \\
				(\rho_+,\vm_+,\mU_+,q_+,\vF_+) & \text{ if } (t,\vx)\in \Gamma_+, 
			\end{array} \right.
		\end{equation}
		where 
		\begin{align}
			q_\pm &:= \frac{|\vm_\pm|^2}{2\rho_\pm} + p(\rho_\pm) , \label{eq:defn-qpm} \\
			\mU_\pm &:= \frac{\vm_\pm \otimes \vm_\pm}{\rho_\pm} + (p(\rho_\pm) - q_\pm ) \id , \label{eq:defn-Upm} \\
			\vF_\pm &:= \frac{q_\pm + P(\rho_\pm)}{\rho_\pm} \vm_\pm . \label{eq:defn-Fpm}
		\end{align}
		
		\item It holds that 
		\begin{equation} \label{eq:fan-subs-subs}
			\lambda_{\max} \left( \frac{\vm_i\otimes\vm_i}{\rho_i} - \mU_i + (p(\rho_i) - q_i) \id \right) < 0 \qquad \text{ for } i=1,2,3. 
		\end{equation}
		
		\item The equations
		\begin{align} 
			\int_0^\infty \int_{\R^2} \big[\rho \partial_t\phi + \ov{\vm} \cdot \Grad\phi \big] \dx\dt + \int_{\R^2} \rho_0 \phi(0,\cdot) \dx &= 0 , \label{eq:fan-subs-pde1}\\ 
			\int_0^\infty \int_{\R^2} \bigg[\ov{\vm}\cdot\partial_t\vphi + \ov{\mU} : \Grad\vphi + \ov{q} \Div\vphi\bigg] \dx\dt + \int_{\R^2} \vm_0\cdot \vphi(0,\cdot) \dx &= 0 , \label{eq:fan-subs-pde2} 
		\end{align}
		and inequality 
		\begin{align}
			\int_0^\infty \int_{\R^2} \bigg[\Big( \ov{q} + P(\rho) - p(\rho) \Big) \partial_t\psi + \ov{\vF} \cdot \Grad\psi \bigg] \dx\dt + \int_{\R^2} \left( \frac{|\vm_0|^2}{2\rho_0} + P(\rho_0) \right) \psi(0,\cdot) \dx &\geq 0 , \label{eq:fan-subs-pde3} 
		\end{align} 
		hold for all test functions $(\phi,\vphi,\psi)\in \Cc([0,\infty)\times \R^2;\R\times\R^2\times \R_0^+)$.
	\end{enumerate}
\end{defn}

\begin{rem} 
	The difference between our notion of an admissible fan subsolution as defined in Defn.~\ref{defn:fan-subs} and the corresponding notion used in the literature so far (see \cite[Defn.~3.5]{ChiDelKre15} and also \cite[Defn.~7.3.3]{Markfelder}) is that the latter requires additionally that $\ov{\vF} = \frac{\ov{q} + P(\rho)}{\rho}\ov{\vm}$. Since we do not have this requirement, our Defn.~\ref{defn:fan-subs} and the consequent statement (see Prop.~\ref{prop:fan-subs=>inf-sol} below) is a genuine generalization of what is done in the literature so far. This is the key point in the proof of Thm.~\ref{thm:local-max-diss}.
\end{rem}

Now we are ready to prove that existence of an admissible fan subsolution implies existence of infinitely many admissible weak solutions. 

\begin{prop} \label{prop:fan-subs=>inf-sol} 
	Let the initial states $(\rho_\pm,\vm_\pm)$ be such that there exists an admissible fan subsolution $(\rho,\ov{\vm},\ov{\mU},\ov{q},\ov{\vF})$ to the initial value problem \eqref{eq:euler-d}, \eqref{eq:euler-m}, \eqref{eq:init}, \eqref{eq:euler-E}, \eqref{eq:riemann-init}. Then there exist infinitely many $\vm\in L^\infty ((0,\infty)\times \R^2; \R^2)$ such that each pair $(\rho,\vm)$ is an admissible weak solution to the aforementioned initial value problem and the following statements hold.
	\begin{itemize}
		\item $\vm(t,\vx)=\vm_-$ for all $(t,\vx)\in \Gamma_-$ and $\vm(t,\vx)=\vm_+$ for all $(t,\vx)\in \Gamma_+$.
		
		\item The left-hand side of the energy inequality satisfies 
		\begin{align} 
			&\int_0^\infty \int_{\R^2} \left[\bigg(\frac{|\vm|^2}{2\rho} + P(\rho)\bigg) \partial_t \psi + \bigg(\frac{|\vm|^2}{2\rho} + P(\rho) + p(\rho)\bigg)\frac{\vm}{\rho}\cdot\Grad \psi \right]\dx\dt \notag \\
			&\geq \int_0^\infty \int_{\R^2} \bigg[\Big( \ov{q} + P(\rho) - p(\rho) \Big) \partial_t\psi + \ov{\vF} \cdot \Grad\psi \bigg] \dx\dt \label{eq:energyofsol-vs-energyofsubsol}
		\end{align}
		where $\psi\in \Cc((0,\infty)\times \R^2;\R^+_0)$ is an arbitrary test function.
	\end{itemize}
\end{prop}

\begin{proof} 
In each $\Gamma_i$, $i=1,2,3$, we apply Thm.~\ref{thm:conv-int-euler+energy}. Note that the restriction of $(\rho,\ov{\vm},\ov{\mU},\ov{q},\ov{\vF})\big|_{\Gamma_i}$ to each $\Gamma_i$ is constant. This means that\footnote{To be precise we need the continuous continuation of $(\rho,\ov{\vm},\ov{\mU},\ov{q},\ov{\vF})\big|_{\Gamma_i}$ onto $\closure{\Gamma_i}$.} $(\rho,\ov{\vm},\ov{\mU},\ov{q},\ov{\vF})\big|_{\Gamma_i}$ has the regularity that is required in the assumptions of Thm.~\ref{thm:conv-int-euler+energy}, and moreover that \eqref{eq:ci-subs-pde1}-\eqref{eq:ci-subs-pde3} hold pointwise on each $\Gamma_i$.

In order to verify assumption \eqref{eq:ci-subs} of Thm.~\ref{thm:conv-int-euler+energy} it suffices to show that there exist $Q_i>0$ with $Q_i> p(\rho_i)$ for $i=1,2,3$ such that 
\begin{equation} \label{eq:aux11}
	(\vm_i,\mU_i,q_i,\vF_i)\in \sW_{\rho_i,Q_i} \qquad \text{ for all }i=1,2,3,
\end{equation}
since $(\rho,\ov{\vm},\ov{\mU},\ov{q},\ov{\vF})$ are constant on each $\Gamma_i$, and because of Prop.~\ref{prop:WsubsetU}. If $Q>0$ is sufficiently large, then $(\vm_i,\mU_i,q_i)\in\sV_{\rho_i,Q}$ (see \eqref{eq:V}) for all $i=1,2,3$ due to \eqref{eq:fan-subs-subs}. Thus $A^j_{\rho_i,Q}(\vm_i,\mU_i,q_i)$ and $r^j_{\rho_i,Q}(\vm_i,\mU_i,q_i)$ are well-defined for all $i=1,2,3$ and $j=1,2,3,4$. Moreover we obtain 
$$
	\lim\limits_{Q\to \infty} \frac{A^j_{\rho_i,Q}(\vm_i,\mU_i,q_i)}{r^j_{\rho_i,Q}(\vm_i,\mU_i,q_i)} = \infty \qquad \text{ for all } i=1,2,3 \text{ and } j=1,2,3,4
$$
due to \eqref{eq:defn-Aj} and \eqref{eq:defn-r}. In view of \eqref{eq:defn-f} and \eqref{eq:defn-sigma} this means that if we choose $Q_i$ sufficiently large, we achieve that $\vF_i - \frac{q_i + P(\rho_i)}{\rho_i}\vm_i$ lies in the interior of the convex polytope spanned by the four points $\vf^j_{\rho_i,Q_i}(\vm_i,\mU_i,q_i)$, $j=1,2,3,4$. And the latter is true for all $i=1,2,3$. Thus \eqref{eq:aux11} holds, see \eqref{eq:WrhoQ}.

Now Thm.~\ref{thm:conv-int-euler+energy} yields infinitely many $\widetilde{\vm}_i\in L^\infty(\Gamma_i;\R^2)$ for each $i=1,2,3$ such that \eqref{eq:ci-sol-pde1}-\eqref{eq:ci-sol-pde3} hold on each $\Gamma_i$. By defining $\vm\in L^\infty((0,\infty)\times \R^2;\R^2)$ through\footnote{Notice that this way we obtain infinitely many functions $\vm$ since there are infinitely many $\widetilde{\vm}_1$, $\widetilde{\vm}_2$ and $\widetilde{\vm}_3$. Moreover we would like to warn the reader not the confuse $\vm_i$ with $\widetilde{\vm}_i$.}
\begin{equation} \label{eq:aux08}
	\vm(t,\vx):= \left\{ \begin{array}{ll} 
		\vm_- & \text{ if } (t,\vx)\in \Gamma_- , \\
		\widetilde{\vm}_i & \text{ if } (t,\vx)\in \Gamma_i \ ,i=1,2,3, \\
		\vm_+ & \text{ if } (t,\vx)\in \Gamma_+ , 
	\end{array} \right.
\end{equation}
the latter may be written as 
\begin{align} 
	\iint_{\Gamma_i} \big[\rho \partial_t\phi + \vm\cdot \Grad\phi \big] \dx\dt - \int_{\partial \Gamma_i} \big[\rho_i \,n_t + \vm_i\cdot \vn_\vx \big] \phi \dS_{t,\vx} &= 0 , \label{eq:aux-sol-pde1}\\ 
	\iint_{\Gamma_i} \bigg[\vm\cdot\partial_t\vphi + \frac{\vm\otimes\vm}{\rho} : \Grad\vphi + p(\rho) \Div\vphi\bigg] \dx\dt \ \;\quad \qquad & \notag \\
	- \int_{\partial\Gamma_i} \left[ \vm_i\cdot \vphi\,n_t + (\mU_i\cdot \vn_\vx)\cdot \vphi + q_i\, \vphi \cdot \vn_\vx \right] \dS_{t,\vx} &= 0 , \label{eq:aux-sol-pde2} \\ 
	\iint_{\Gamma_i} \bigg[\left( \frac{|\vm|^2}{2\rho} + P(\rho)\right) \partial_t\psi + \left( \frac{|\vm|^2}{2\rho} + P(\rho) + p(\rho) \right) \frac{\vm}{\rho} \cdot \Grad\psi \bigg] \dx\dt \ \;\quad \qquad & \notag \\
	- \int_{\partial\Gamma_i} \left[ (q_i + P(\rho_i) - p(\rho_i)) n_t + \vF_i\cdot \vn_\vx \right] \psi \dS_{t,\vx} &\geq 0 , \label{eq:aux-sol-pde3} 
\end{align} 
for all test functions $(\phi,\vphi,\psi)\in \Cc([0,\infty)\times\R^2;\R\times\R^2\times \R_0^+)$.

Next we check that each $(\rho,\vm)$ satisfies \eqref{eq:euler-weak-d}, \eqref{eq:euler-weak-m} and \eqref{eq:euler-weak-E}. We begin with \eqref{eq:euler-weak-d}. From \eqref{eq:fan-subs-pde1}, \eqref{eq:aux-sol-pde1} and the Divergence Theorem we find
\begin{align*} 
	&\int_0^\infty \int_{\R^2} \Big[\rho \partial_t \phi + \vm\cdot\Grad \phi\Big]\dx\dt + \int_{\R^2} \rho_0\phi(0,\cdot) \dx \\
	&= \int_0^\infty \int_{\R^2} \Big[\rho \partial_t \phi + \ov{\vm}\cdot\Grad \phi\Big]\dx\dt + \int_{\R^2} \rho_0\phi(0,\cdot) \dx \\
	&\qquad - \sum_{i=1}^3 \iint_{\Gamma_i} \big[\rho_i \partial_t\phi + \vm_i\cdot \Grad\phi \big] \dx\dt + \sum_{i=1}^3 \iint_{\Gamma_i} \big[\rho \partial_t\phi + \vm\cdot \Grad\phi \big] \dx\dt \\
	&= - \sum_{i=1}^3 \int_{\partial \Gamma_i} \big[\rho_i \,n_t + \vm_i\cdot \vn_\vx \big] \phi \dS_{t,\vx} + \sum_{i=1}^3 \int_{\partial \Gamma_i} \big[\rho_i \,n_t + \vm_i\cdot \vn_\vx \big] \phi \dS_{t,\vx} \ = \ 0
\end{align*} 
for all $\phi \in \Cc([0,\infty) \times \R^2)$.

Analogously \eqref{eq:fan-subs-pde2}, \eqref{eq:aux-sol-pde2} and the Divergence Theorem yield
\begin{align*} 
	&\int_0^\infty \int_{\R^2} \left[\vm \cdot\partial_t \vphi + \frac{\vm\otimes\vm}{\rho}:\Grad \vphi + p(\rho)\Div \vphi\right]\dx\dt + \int_{\R^2} \vm_0\cdot\vphi(0,\cdot) \dx \\
	&= \int_0^\infty \int_{\R^2} \left[\ov{\vm} \cdot\partial_t \vphi + \ov{\mU}:\Grad \vphi + \ov{q}\Div \vphi\right]\dx\dt + \int_{\R^2} \vm_0\cdot\vphi(0,\cdot) \dx \\
	&\qquad - \sum_{i=1}^3 \iint_{\Gamma_i} \left[\vm_i \cdot\partial_t \vphi + \mU_i:\Grad \vphi + q_i\Div \vphi\right]\dx\dt \\
	&\qquad + \sum_{i=1}^3 \iint_{\Gamma_i} \left[\vm \cdot\partial_t \vphi + \frac{\vm\otimes\vm}{\rho}:\Grad \vphi + p(\rho)\Div \vphi\right]\dx\dt \\
	&= - \sum_{i=1}^3 \int_{\partial\Gamma_i} \left[ \vm_i\cdot \vphi\,n_t + (\mU_i\cdot \vn_\vx)\cdot \vphi + q_i\, \vphi \cdot \vn_\vx \right] \dS_{t,\vx} \\ 
	&\qquad + \sum_{i=1}^3 \int_{\partial\Gamma_i} \left[ \vm_i\cdot \vphi\,n_t + (\mU_i\cdot \vn_\vx)\cdot \vphi + q_i\, \vphi \cdot \vn_\vx \right] \dS_{t,\vx} \ = \ 0
\end{align*} 
for any $\vphi \in \Cc([0,\infty) \times \R^2;\R^2)$.

Finally according to \eqref{eq:fan-subs-pde3}, \eqref{eq:aux-sol-pde3} and the Divergence Theorem we have 
\begin{align*} 
	&\int_0^\infty \int_{\R^2} \left[\bigg(\frac{|\vm|^2}{2\rho} + P(\rho)\bigg) \partial_t \psi + \bigg(\frac{|\vm|^2}{2\rho} + P(\rho) + p(\rho)\bigg)\frac{\vm}{\rho}\cdot\Grad \psi \right]\dx\dt \\ 
	&\quad + \int_{\R^2} \bigg(\frac{|\vm_0|^2}{2\rho_0} + P(\rho_0)\bigg)\psi(0,\cdot) \dx \\
	&= \int_0^\infty \int_{\R^2} \bigg[\Big( \ov{q} + P(\rho) - p(\rho) \Big) \partial_t\psi + \ov{\vF} \cdot \Grad\psi \bigg] \dx\dt + \int_{\R^2} \left( \frac{|\vm_0|^2}{2\rho_0} + P(\rho_0) \right) \psi(0,\cdot) \dx \\
	&\qquad - \sum_{i=1}^3 \iint_{\Gamma_i} \bigg[\Big( q_i + P(\rho_i) - p(\rho_i) \Big) \partial_t\psi + \vF_i \cdot \Grad\psi \bigg] \dx\dt \\
	&\qquad + \sum_{i=1}^3 \iint_{\Gamma_i} \left[\bigg(\frac{|\vm|^2}{2\rho} + P(\rho)\bigg) \partial_t \psi + \bigg(\frac{|\vm|^2}{2\rho} + P(\rho) + p(\rho)\bigg)\frac{\vm}{\rho}\cdot\Grad \psi \right] \dx\dt \\ 
	&\geq - \sum_{i=1}^3 \int_{\partial\Gamma_i} \left[ (q_i + P(\rho_i) - p(\rho_i)) n_t + \vF_i\cdot \vn_\vx \right] \psi \dS_{t,\vx} \\
	&\qquad + \sum_{i=1}^3 \int_{\partial\Gamma_i} \left[ (q_i + P(\rho_i) - p(\rho_i)) n_t + \vF_i\cdot \vn_\vx \right] \psi \dS_{t,\vx} \ = \ 0
\end{align*}
for all test functions $\psi \in \Cc([0,\infty) \times \R^2;\R^+_0)$.

Thus we have shown that each $(\rho,\vm)$ is indeed an admissible weak solution. The additional properties claimed in Prop.~\ref{prop:fan-subs=>inf-sol} immediately follow from \eqref{eq:aux08} and the computation above.
\end{proof}

We finish this subsection with a statement that translates the definition of an admissible fan subsolution (Defn.~\ref{defn:fan-subs}) into a system of algebraic equations and inequalities. The following proposition corresponds to \cite[Prop.~5.1]{ChiDelKre15}. 

\begin{prop} \label{prop:algebraic}
	Let $(\rho_\pm,\vm_\pm)\in \R^+ \times \R^2$ be given. Assume that there exist numbers $\mu_0,\mu_1,\mu_2,\mu_3\in \R$, $(\rho_i,\vm_i,\mU_i,q_i,\vF_i) \in \R^+\times \phase$ for $i=1,2,3$ which fulfill the following algebraic equations and inequalities.
	\begin{itemize}
		\item Order of the speeds:
		\begin{equation} \label{eq:alg-orderofspeeds}
			\mu_0 < \mu_1 < \mu_2 < \mu_3 .
		\end{equation}
		
		\item Rankine Hugoniot conditions at each interface: 
		\begin{align} 
			\mu_i \big(\rho_i - \rho_{i+1}\big) &= [\vm_i]_2 - [\vm_{i+1}]_2 , \label{eq:alg-rh1} \\
			\mu_i \big([\vm_i]_1 - [\vm_{i+1}]_1 \big) &= [\mU_i]_{12} - [\mU_{i+1}]_{12} , \label{eq:alg-rh2} \\
			\mu_i \big([\vm_i]_2 - [\vm_{i+1}]_2 \big) &= -[\mU_i]_{11} + q_i + [\mU_{i+1}]_{11} - q_{i+1} , \label{eq:alg-rh3} \\ 
			\mu_i \big(q_i + P(\rho_i) - p(\rho_i) - q_{i+1} - P(\rho_{i+1}) &+ p(\rho_{i+1})\big) \leq [\vF_i]_2 - [\vF_{i+1}]_2 , \label{eq:alg-rh4} 
		\end{align}
		for all $i=0,1,2,3$, where 
		\begin{align*}
			(\rho_0,\vm_0,\mU_0,q_0,\vF_0) &:= (\rho_-,\vm_-,\mU_-,q_-,\vF_-) , \\
			(\rho_4,\vm_4,\mU_4,q_4,\vF_4) &:= (\rho_+,\vm_+,\mU_+,q_+,\vF_+) ,
		\end{align*}
		with $\mU_\pm$, $q_\pm$ and $\vF_\pm$ defined in \eqref{eq:defn-qpm}-\eqref{eq:defn-Fpm}.
		
		\item Subsolution conditions: 
		\begin{align} 
			\frac{[\vm_i]_1^2 + [\vm_i]_2^2}{\rho_i} + 2(p(\rho_i) - q_i) &< 0, \label{eq:alg-subs1} \\ 
			\left( \frac{[\vm_i]_1^2}{\rho_i} - [\mU_i]_{11} + p(\rho_i) - q_i \right) \left( \frac{[\vm_i]_2^2}{\rho_i} + [\mU_i]_{11} + p(\rho_i) - q_i \right) \ \;\quad \qquad & \notag \\
			- \left( \frac{[\vm_i]_1 [\vm_i]_2}{\rho_i} - [\mU_i]_{12} \right)^2 &> 0 \label{eq:alg-subs2}
		\end{align}
		for all $i=1,2,3$.
	\end{itemize}
	Then $(\rho_i,\vm_i,\mU_i,q_i,\vF_i)_{i=1,2,3}$ define an admissible fan subsolution to the initial value problem \eqref{eq:euler-d}, \eqref{eq:euler-m}, \eqref{eq:init}, \eqref{eq:euler-E}, \eqref{eq:riemann-init} via \eqref{eq:fan-subs-pwc}, where the corresponding fan partition is determined by $\mu_i$, $i=0,1,2,3$. Moreover we have 
	\begin{align} 
		&\int_0^\infty \int_{\R^2} \bigg[\Big( \ov{q} + P(\rho) - p(\rho) \Big) \partial_t\psi + \ov{\vF} \cdot \Grad\psi \bigg] \dx\dt \notag \\
		&= \sum_{i=0}^3 \frac{1}{\sqrt{\mu_i^2 + 1}}\int_{S_i} \Big[ -\mu_i \big( q_i + P(\rho_i) - p(\rho_i) \label{eq:alg-additional} \\
		&\qquad\qquad\qquad\qquad\qquad\qquad - q_{i+1} - P(\rho_{i+1}) + p(\rho_{i+1})\big) + [\vF_i]_2 - [\vF_{i+1}]_2 \Big] \psi \dS_{t,\vx} \notag
	\end{align}
	for all test functions $\psi\in \Cc((0,\infty)\times \R^2;\R^+_0)$, where $S_i$ are the planes defined by 
	\begin{equation} \label{eq:alg-Si}
		S_i := \left\{ (t,\vx)\in (0,\infty)\times \R^2\,\big|\, y= \mu_i t \right\} \qquad \text{ for }i=0,1,2,3.
	\end{equation}
\end{prop} 

\begin{proof} 
Since the proof is very similar to \cite[Proof of Prop.~5.1]{ChiDelKre15}, see also \cite[Proof of Prop.~7.3.5]{Markfelder}, we only sketch it here. 

Because admissible fan subsolutions are piecewise constant, the partial differential equations \eqref{eq:fan-subs-pde1}, \eqref{eq:fan-subs-pde2} and inequality \eqref{eq:fan-subs-pde3} are equivalent to their corresponding Rankine-Hugoniot conditions \eqref{eq:alg-rh1}-\eqref{eq:alg-rh4}. In the same fashion one proves \eqref{eq:alg-additional}.

Finally, \eqref{eq:fan-subs-subs} is equivalent to \eqref{eq:alg-subs1} and \eqref{eq:alg-subs2} according to Lemma~\ref{lemma:app-eigenvalues}. This finishes the proof.
\end{proof}

\subsection{Proof of Thm.~\ref{thm:local-max-diss}} \label{subsec:rci-proof-of-main-thm}

Finally we are ready to prove Thm.~\ref{thm:local-max-diss}.

\begin{proof}[Proof of Thm.~\ref{thm:local-max-diss}]
Consider the pressure law $p(\rho)=\rho^2$, i.e. \eqref{eq:pressure} with $\gamma=2$. We observe that in this case the pressure potential defined in \eqref{eq:pressure-potential} reads $P(\rho)=\rho^2=p(\rho)$. 

We consider Riemann initial data with 
\begin{align*} 
	\rho_- &= 1, & \rho_+ &= 4, \\
	\vm_- &= \left(\begin{array}{c} 0 \\ \frac{3}{2}\sqrt{5} \end{array} \right), & \vm_+ &= \left(\begin{array}{c} 0 \\ 0 \end{array} \right).
\end{align*}
With classical techniques one can show that the one-dimensional self-similar solution $(\rho_\ts,\vm_\ts)$ to the initial data above consists of one shock. More precisely we have 
$$
	(\rho_\ts,\vm_\ts)(t,\vx) = \left\{ \begin{array}{ll}
		(\rho_-,\vm_-) & \text{ if } y < \sigma t, \\
		(\rho_+,\vm_+) & \text{ if } y > \sigma t, 
	\end{array} \right.
$$
where $\sigma= -\frac{\sqrt{5}}{2}$. Furthermore the left-hand side of the energy inequality reads 
\begin{align} 
	&\int_0^\infty \int_{\R^2} \left[\bigg(\frac{|\vm_\ts|^2}{2\rho_\ts} + P(\rho_\ts)\bigg) \partial_t \psi + \bigg(\frac{|\vm_\ts|^2}{2\rho_\ts} + P(\rho_\ts) + p(\rho_\ts)\bigg)\frac{\vm_\ts}{\rho_\ts}\cdot\Grad \psi \right]\dx\dt \notag\\
	&= \frac{1}{\sqrt{\sigma^2+1}} \int_{S_\ts} \bigg[-\sigma \bigg(\frac{|\vm_-|^2}{2\rho_-} + P(\rho_-) - \frac{|\vm_+|^2}{2\rho_+} - P(\rho_+)\bigg) \notag\\
	&\qquad\qquad + \bigg(\frac{|\vm_-|^2}{2\rho_-} + P(\rho_-) + p(\rho_-)\bigg)\frac{[\vm_-]_2}{\rho_-} - \bigg(\frac{|\vm_+|^2}{2\rho_+} + P(\rho_+) + p(\rho_+)\bigg)\frac{[\vm_+]_2}{\rho_+} \bigg] \psi \dS_{t,\vx} \notag\\
	&= \frac{1}{\sqrt{\sigma^2+1}} \int_{S_\ts} \frac{27}{4}\sqrt{5} \ \psi \dS_{t,\vx} \label{eq:energy-standard}
\end{align}
for all $\psi \in \Cc((0,\infty) \times \R^2;\R^+_0)$, where 
$$
	S_\ts:= \left\{ (t,\vx)\in (0,\infty)\times \R^2\,\big|\, y= \sigma t \right\}.
$$

Next we present an admissible fan subsolution to the data above. Let 
\begin{align*}
	\mu_0 &= \frac{-53750 \sqrt{5}+77 \sqrt{1141}-25102}{107500}, & \mu_1 &= -\frac{\sqrt{5}}{2}, \\
	\mu_2 &= \frac{77-125 \sqrt{5}}{250}, & \mu_3 &= \frac{-26875 \sqrt{5}+8316 \sqrt{1141}-12551}{53750}, 
\end{align*}
and
\begin{align*} 
	\rho_1 &= \frac{52}{25}, & \rho_2 &= \frac{319}{100}, & \rho_3 &= \frac{801}{200}, 
\end{align*}
\begin{align*} 
	\vm_1 &= \left(\begin{array}{c} 0 \\ \frac{3 \left(860000 \sqrt{5}+693 \sqrt{1141}-225918\right)}{2687500} \end{array} \right), \\
	\vm_2 &= \left(\begin{array}{c} 0 \\ \frac{27 \left(80625 \sqrt{5}+154 \sqrt{1141}-50204\right)}{5375000} \end{array} \right), \\
	\vm_3 &= \left(\begin{array}{c} 0 \\ \frac{-26875 \sqrt{5}+8316 \sqrt{1141}-12551}{10750000} \end{array} \right),
\end{align*}
\begin{align*} 
	\mU_1 &= \left(\begin{array}{cc} \frac{-72858555000 \sqrt{5}+8316 \sqrt{1141} (26875 \sqrt{5}+12551)-1272258135611}{288906250000} & 0 \\ 0 & \ast \end{array} \right), \\
	\mU_2 &= \left(\begin{array}{cc} \frac{-36429277500 \sqrt{5}+4158 \sqrt{1141} (26875 \sqrt{5}+12551)-295698403743}{144453125000} & 0 \\ 0 & \ast \end{array} \right), \\
	\mU_3 &= \left(\begin{array}{cc} \frac{-337308125 \sqrt{5}+8316 \sqrt{1141} (26875 \sqrt{5}+12551)-21181593711}{288906250000} & 0 \\ 0 & \ast \end{array} \right),
\end{align*}
where we wrote ``$\ast$'' for each lower right entry in order to keep the presentation short. Note that each lower right entry is equal to $-1$ times the upper left entry since the matrices are trace-less. Let moreover
\begin{align*}
	q_1 &= \frac{398}{43} , &
	q_2 &= 13 , &
	q_3 &= \frac{691}{43} , \\
	\vF_1 &= \left(\begin{array}{c} 0 \\ \frac{552}{25} \end{array} \right), & \vF_2 &= \left(\begin{array}{c} 0 \\ \frac{277}{100} \end{array} \right), & \vF_3 &= \left(\begin{array}{c} 0 \\ \frac{7}{25} \end{array} \right). 
\end{align*}
It is now straightforward to check the conditions \eqref{eq:alg-orderofspeeds}-\eqref{eq:alg-subs2} of Prop.~\ref{prop:algebraic}. As a consequence $(\rho,\ov{\vm},\ov{\mU},\ov{q},\ov{\vF})$ defined in \eqref{eq:fan-subs-pwc} is an admissible fan subsolution. In addition to that we compute in accordance with \eqref{eq:alg-additional} that 
\begin{align} 
	&\int_0^\infty \int_{\R^2} \bigg[\Big( \ov{q} + P(\rho) - p(\rho) \Big) \partial_t\psi + \ov{\vF} \cdot \Grad\psi \bigg] \dx\dt \notag \\
	&= \frac{1}{\sqrt{\mu_0^2 + 1}}\int_{S_0} \frac{74863000 \sqrt{5}+13937 \sqrt{1141}-167847142}{7396000} \ \psi \dS_{t,\vx} \label{eq:energy-wild} \\
	&\qquad + \frac{1}{\sqrt{\mu_1^2 + 1}}\int_{S_1} \frac{83033-8050 \sqrt{5}}{4300} \ \psi \dS_{t,\vx} \notag \\
	&\qquad + \frac{1}{\sqrt{\mu_2^2 + 1}}\int_{S_2} \frac{73863-33000 \sqrt{5}}{21500} \ \psi \dS_{t,\vx} \notag \\
	&\qquad + \frac{1}{\sqrt{\mu_3^2 + 1}}\int_{S_3} \frac{80625 \sqrt{5}-24948 \sqrt{1141}+684803}{2311250} \ \psi \dS_{t,\vx} \notag
\end{align}
for all test functions $\psi\in \Cc((0,\infty)\times \R^2;\R^+_0)$, where $S_i$ are defined in \eqref{eq:alg-Si}.

Now Prop.~\ref{prop:fan-subs=>inf-sol} yields infinitely many admissible weak solutions which satisfy \eqref{eq:energyofsol-vs-energyofsubsol}. Together with \eqref{eq:energy-standard} and \eqref{eq:energy-wild}, we infer for all test functions $\psi\in \Cc((0,\infty)\times \R^2;\R^+_0)$ that
\begin{align} 
	&\int_0^\infty \int_{\R^2} \left[\bigg(\frac{|\vm|^2}{2\rho} + P(\rho)\bigg) \partial_t \psi + \bigg(\frac{|\vm|^2}{2\rho} + P(\rho) + p(\rho)\bigg)\frac{\vm}{\rho}\cdot\Grad \psi \right]\dx\dt \notag\\
	&\geq \frac{1}{\sqrt{\mu_1^2 + 1}}\int_{S_1} \frac{83033-8050 \sqrt{5}}{4300} \ \psi \dS_{t,\vx} \notag \\
	&\geq \frac{1}{\sqrt{\sigma^2+1}} \int_{S_\ts} \frac{27}{4}\sqrt{5} \ \psi \dS_{t,\vx} \label{eq:final-ineq} \\ 
	&= \int_0^\infty \int_{\R^2} \left[\bigg(\frac{|\vm_\ts|^2}{2\rho_\ts} + P(\rho_\ts)\bigg) \partial_t \psi + \bigg(\frac{|\vm_\ts|^2}{2\rho_\ts} + P(\rho_\ts) + p(\rho_\ts)\bigg)\frac{\vm_\ts}{\rho_\ts}\cdot\Grad \psi \right]\dx\dt, \notag
\end{align}
because $\mu_1=\sigma$ and hence $S_1=S_\ts$, and because $\frac{83033-8050 \sqrt{5}}{4300} > \frac{151}{10} > \frac{27}{4}\sqrt{5}$. In addition, the inequality in \eqref{eq:final-ineq} is strict for all test functions $\psi\in \Cc((0,\infty)\times \R^2;\R^+_0)$ with the property $\supp(\psi)\cap S_\ts\neq \emptyset$.

This shows that the one-dimensional self-similar solution $(\rho_\ts,\vm_\ts)$ does not satisfy the local maximal dissipation criterion since we have constructed infinitely many additional admissible weak solution which emanate from the same initial data and satisfy \eqref{eq:D-ineq}. 
\end{proof}

\section*{Acknowledgements} 
The author acknowledges financial support from the Alexander von Humboldt foundation and also from the Deutsche Forschungsgemeinschaft (DFG, German Research Foundation) SPP 2410 \emph{Hyperbolic Balance Laws in Fluid Mechanics: Complexity, Scales, Randomness (CoScaRa)}, within the Project 525935467 \emph{Convex integration: towards a mathematical understanding of turbulence, Onsager conjectures and admissibility criteria}. Moreover the author would like to thank Eduard Feireisl who raised the question that is answered in Thm.~\ref{thm:local-max-diss}. The author is also very thankful to Daniel W. Boutros for many inspiring discussions. Finally the author wishes to thank the anonymous referees for their valuable suggestions and comments. 

\appendix

\section{Some Linear Algebra} 

We write $\sym{2}$ for the space of symmetric $2\times 2$-matrices. The eigenvalues of $\mM\in \sym{2}$ are denoted by $\lambda_{\min}(\mM)$ and $\lambda_{\max}(\mM)$ where $\lambda_{\min}(\mM)\leq \lambda_{\max}(\mM)$.

\begin{lemma} \label{lemma:app-computationoftrace} 
	For any $\mM\in \sym{2}$ it holds that 
	\begin{equation*}
		\tr(\mM) \big(\tr(\mM) \pm 2 [\mM]_{12}\big) - 2\det(\mM) = \big([\mM]_{11} \pm [\mM]_{12}\big)^2 + \big([\mM]_{22} \pm [\mM]_{12}\big)^2 \geq 0.
	\end{equation*}
\end{lemma}

\begin{proof}
We calculate
\begin{align*}
	&\tr(\mM) \big(\tr(\mM) \pm 2 [\mM]_{12}\big) - 2\det(\mM)\\
	&= \big([\mM]_{11} + [\mM]_{22} \big)^2 \pm 2 [\mM]_{11} [\mM]_{12} \pm 2 [\mM]_{22} [\mM]_{12} - 2[\mM]_{11}[\mM]_{22} +2 [\mM]_{12}^2 \\
	&= [\mM]_{11}^2 + [\mM]_{22}^2 \pm 2 [\mM]_{11} [\mM]_{12} \pm 2 [\mM]_{22} [\mM]_{12} +2 [\mM]_{12}^2 \\
	&=\big([\mM]_{11} \pm [\mM]_{12}\big)^2 + \big([\mM]_{22} \pm [\mM]_{12}\big)^2 .
\end{align*} 
\end{proof}

\begin{lemma} \label{lemma:app-trdet}
	For any $\mM\in \sym{2}$, we have 
	\begin{align*}
		\tr(\mM) &= \lambda_{\min}(\mM)+ \lambda_{\max}(\mM), & \det(\mM) &= \lambda_{\min}(\mM) \cdot \lambda_{\max}(\mM),
	\end{align*}
	and hence
	\begin{align*}
		\lambda_{\max}(\mM) \leq 0 \quad \Leftrightarrow \quad \tr(\mM)\leq 0\ \text{ and }\ \det(\mM)\geq 0, \\ 
		\lambda_{\max}(\mM) < 0 \quad \Leftrightarrow \quad \tr(\mM)< 0\ \text{ and }\ \det(\mM)> 0, \\
		\lambda_{\max}(\mM) = 0 \quad \Leftrightarrow \quad \tr(\mM)\leq 0\ \text{ and }\ \det(\mM)= 0.
	\end{align*} 
\end{lemma} 

Lemma~\ref{lemma:app-trdet} is a well-known fact in linear algebra, so we skip its proof.

In view of Lemma~\ref{lemma:app-euler} below, the following lemma provides some alternative representations of the convex hull $\sK^\co$. 

\renewcommand{\arraystretch}{1.3}
\begin{lemma} \label{lemma:app-eigenvalues} 
	Let $(\rho,\vm,\mU,q)\in \R^+\times \R^2 \times \sz\times \R$. Then the following assertions are equivalent: 
	\begin{align}
		&\bullet\ \lambda_{\max}\left(\frac{\vm\otimes\vm}{\rho} - \mU + (p(\rho) - q)\id\right)\leq 0, \label{eq:app-eigenvalues1} \\
		&\bullet\ \frac{[\vm]_1^2 + [\vm]_2^2}{2\rho} + p(\rho) - q + \sqrt{\left(\frac{[\vm]_1^2 - [\vm]_2^2}{2\rho} - [\mU]_{11} \right)^2 + \left( \frac{[\vm]_1 [\vm]_2}{\rho} - [\mU]_{12} \right)^2 } \leq 0 , \label{eq:app-eigenvalues2} \\ 
		&\bullet\ \left\{ \begin{array}{r} \frac{[\vm]_1^2 + [\vm]_2^2}{\rho} + 2(p(\rho) - q) \leq 0, \\ \Big( \frac{[\vm]_1^2}{\rho} - [\mU]_{11} + p(\rho) - q\Big) \Big( \frac{[\vm]_2^2}{\rho} + [\mU]_{11} + p(\rho) - q \Big) - \Big( \frac{[\vm]_1 [\vm]_2}{\rho} - [\mU]_{12} \Big)^2 \geq 0 . \end{array} \right. \label{eq:app-eigenvalues3}
	\end{align}
	The above claim is also true if one replaces all non-strict inequalities by strict inequalities in \eqref{eq:app-eigenvalues1}-\eqref{eq:app-eigenvalues3} (i.e. one replaces all ``$\leq$'' and ``$\geq$'' by ``$<$'' and ``$>$'' respectively).
\end{lemma} 
\renewcommand{\arraystretch}{1}

\begin{proof} 
A straightforward computation shows
\begin{align*}
	&\lambda_{\max}\left(\frac{\vm\otimes\vm}{\rho} - \mU + (p(\rho) - q)\id\right) \\
	&= \frac{[\vm]_1^2 + [\vm]_2^2}{2\rho} + p(\rho) - q + \sqrt{\left(\frac{[\vm]_1^2 - [\vm]_2^2}{2\rho} - [\mU]_{11} \right)^2 + \left( \frac{[\vm]_1 [\vm]_2}{\rho} - [\mU]_{12} \right)^2 } ,
\end{align*}
which immediately proves \eqref{eq:app-eigenvalues1} $\Leftrightarrow$ \eqref{eq:app-eigenvalues2}, both in the case of non-strict and strict inequalities.
	
Next according to Lemma~\ref{lemma:app-trdet}, $\lambda_{\max}\left(\frac{\vm\otimes\vm}{\rho} - \mU + (p(\rho) - q)\id\right) \leq 0$ is equivalent to 
\begin{equation*}
	\tr \left(\frac{\vm\otimes\vm}{\rho} - \mU + (p(\rho) - q)\id\right) \leq 0 , \qquad \text{ and } \qquad
	\det \left(\frac{\vm\otimes\vm}{\rho} - \mU + (p(\rho) - q)\id\right) \geq 0 . 
\end{equation*}
By computing trace and determinant of $\frac{\vm\otimes\vm}{\rho} - \mU + (p(\rho) - q)\id$, we obtain \eqref{eq:app-eigenvalues3}. The strict case follows analogously.
\end{proof}

\section{Convex Hulls in General} \label{app:convex} 

In this section we recall the definition of the convex hull. Most of the content can be also found in \cite[Sect.~A.5.1]{Markfelder}.

In this section we suppose $M \in \N$.

\begin{defn}[See {\cite[Defn.~A.5.1]{Markfelder}}] \label{defn:app-convex}
	\begin{itemize}
		\item The \emph{closed line segment} $[\vp,\vq]\subset \R^M$ between two points $\vp,\vq\in \R^M$ is defined by 
		$$
			[\vp,\vq]=\left\{\tau \vp + (1-\tau) \vq \,\Big|\, \tau\in [0,1]\right\}.
		$$
		
		\item A set $S\subset \R^M$ is called \emph{convex} if $[\vp,\vq]\subset S$ for all $\vp,\vq\in S$.
		
		\item Let $S\subset \R^M$ be closed and convex. A point $\vs\in S$ is called \emph{extreme point of $S$} if there are no two points $\vp,\vq\in S\setminus\{\vs\}$ with $\vs\in [\vp,\vq]$. The set of all extreme points of $S$ is denoted by $\ext(S)$.
		
		\item The \emph{convex hull $\sK^\co$} of a set $\sK\subset \R^M$ is the smallest convex set which contains $\sK$.
	\end{itemize}
\end{defn}

The following fact is well-known.

\begin{prop}[See e.g. {\cite[Prop.~A.5.2]{Markfelder}}] \label{prop:app-caratheodory}
	Let $\sK\subset \R^M$. Then 
	\begin{align*}
		\sK^\co = \bigg\{ \vp\in \R^M\, \Big|\, &\exists N\in \N, \exists (\tau_i,\vp_i)\in \R^+ \times \R^M \text{ for all }i=1,...,N \text{ such that } \\
		& \bullet \ \sum_{i=1}^N \tau_i = 1,  \\
		& \bullet \ \vp_i \in \sK \text{ for all }i=1,...,N, \text{ and } \\
		& \bullet \ \vp= \sum_{i=1}^N \tau_i \vp_i \bigg\} .
	\end{align*}
\end{prop}

The following lemma is a simple observation. 

\begin{lemma} \label{lemma:app-interior-convex} 
	Let $S\subset \R^M$ convex and $\vp\in S$, $\vq\in \interior{S}$, $\tau\in (0,1)$. Then 
	$$
		\tau \vp + (1-\tau) \vq \ \in\ \interior{S}.
	$$
\end{lemma}

\begin{proof}
	As $\vq\in\interior{S}$, there exists $r>0$ such that the ball $B_r(\vq)$ centered at $\vq$ with radius $r$ is contained in $S$, i.e. $B_r(\vq) \subset S$. For convenience we set $\vs:= \tau \vp + (1-\tau) \vq$. Let $\vu\in B_{(1-\tau)r}(\vs)$ and note that $(1-\tau)r>0$ by assumption. Then $\vu\in S$ because $\vu$ can be written as the convex combination of two points in $S$: 
	$$
		\vu = \tau \vp + (1-\tau) \frac{\vu-\tau \vp}{1-\tau},
	$$
	where $\left| \frac{\vu-\tau \vp}{1-\tau} - \vq \right| = \frac{1}{1-\tau} | \vu-\vs | < r$ and thus $\frac{\vu-\tau \vp}{1-\tau} \in B_r(\vq) \subset S$.
	
	So we have shown that $B_{(1-\tau)r}(\vs) \subset S$ and hence $\vs\in \interior{S}$ as desired.
\end{proof}

As a corollary of Lemma~\ref{lemma:app-interior-convex} we obtain the following statement.

\begin{cor} \label{cor:app-boundary-vs-boundaryofinterior}
	Let $S\subset \R^M$ convex and compact, and such that $\interior{S}\neq \emptyset$. Then $\closure{\interior{S}}= S$ and thus $\partial \interior{S} = \partial S$. 
\end{cor}

\begin{proof} 
Since $\interior{S}\subset S$, we have $\closure{\interior{S}}\subset \closure{S} = S$. For the converse let $\vp\in S$ and choose $\vq\in \interior{S}$ which exists by assumption. Define 
$$
	\vp_k := \left( 1 - \frac{1}{k} \right) \vp + \frac{1}{k} \vq \qquad \text{ for all } k\in \N. 
$$
Then obviously $\vp_k\to \vp$ as $k\to \infty$, and according to Lemma~\ref{lemma:app-interior-convex} we have $\vp_k\in \interior{S}$ for all $k\in \N$. Hence $\vp\in \closure{\interior{S}}$. 

Finally we have $\partial \interior{S} = \closure{\interior{S}} \setminus \interior{S} = S \setminus \interior{S} = \closure{S} \setminus \interior{S} = \partial S$.
\end{proof}

We finish this section with the following lemma. 

\begin{lemma} \label{lemma:app-distance-concave}
	Let $S\subset \R^M$ convex and compact. Then the map $\vp\mapsto \dist(\vp,\partial S)$ is concave on $S$. 
\end{lemma}

\begin{proof} 
Let $\tau\in [0,1]$ and $\vp,\vq\in S$. We have to show that 
\begin{equation} \label{eq:aux03} 
	\dist( \tau\vp + (1-\tau) \vq ,\partial S ) \geq \tau \dist (\vp,\partial S) + (1-\tau) \dist (\vq,\partial S) .
\end{equation}
Note that \eqref{eq:aux03} is obviously true if the right-hand side is $0$, so we may assume without loss of generality that $\tau \dist (\vp,\partial S) + (1-\tau) \dist (\vq,\partial S)>0$. By definition of the distance we know that both closed balls 
$$ 
	\closure{B}_{\dist(\vp,\partial S)} (\vp) \qquad \text{ and } \qquad \closure{B}_{\dist(\vq,\partial S)} (\vq)
$$
are contained in $S$. 

Let 
$$
	\vs\in \closure{B}_{\tau \dist (\vp,\partial S) + (1-\tau) \dist (\vq,\partial S)} (\tau\vp + (1-\tau) \vq).
$$
It is straightforward to check that the points
\begin{align*}
	\widehat{\vp} &:= \frac{\big(\vs - \tau \vp - (1-\tau) \vq\big) \dist (\vp,\partial S)}{\tau \dist (\vp,\partial S) + (1-\tau) \dist (\vq,\partial S)} + \vp , \\
	\widehat{\vq} &:= \frac{\big(\vs - \tau \vp - (1-\tau) \vq\big) \dist (\vq,\partial S)}{\tau \dist (\vp,\partial S) + (1-\tau) \dist (\vq,\partial S)} + \vq 
\end{align*}
satisfy $\widehat{\vp}\in \closure{B}_{\dist(\vp,\partial S)} (\vp) \subset S$ and $\widehat{\vq}\in \closure{B}_{\dist(\vq,\partial S)} (\vq)\subset S$, and that $\vs = \tau \widehat{\vp} + (1-\tau) \widehat{\vq}$. Thus by convexity of $S$ we have $\vs\in S$, so we have proven that
$$
	\closure{B}_{\tau \dist (\vp,\partial S) + (1-\tau) \dist (\vq,\partial S)} (\tau\vp + (1-\tau) \vq) \subset S.
$$
The latter leads to the desired relation \eqref{eq:aux03}.
\end{proof}

\section{The Convex Hull in the Context of the Compressible Euler Equations} 

For an incompressible convex integration approach applied to the compressible Euler system \eqref{eq:euler-d}, \eqref{eq:euler-m} which does not determine energy and energy flux, the constitutive set $\sK$ and its convex hull $\sK^\co$ read as follows.

\begin{lemma} \label{lemma:app-euler}
	Let $\rho\in \R^+$, $q\in \R$ fixed, and set 
	$$
		\sK := \left\{ (\vm,\mU) \in \R^2 \times \sz \, \Big| \, \frac{\vm\otimes \vm}{\rho} + p(\rho) \id = \mU + q \id \right\} .
	$$
	Then     
	$$
		\sK^\co := \left\{ (\vm,\mU) \in \R^2 \times \sz \, \Big| \, \lambda_{\max}\left(\frac{\vm\otimes \vm}{\rho} - \mU + (p(\rho)-q) \id \right) \leq 0 \right\} .
	$$
\end{lemma}

We refer to \cite[Lemma~4.3.6]{Markfelder} for a detailed proof of Lemma~\ref{lemma:app-euler} which is based on \cite[Lemma~3]{DelSze10}.

\section{$\Lambda$-Convex Hulls and the $H_N$-Condition} \label{app:Lconvex}

In this section we summarize the definitions and some properties of the $\Lambda$-convex hull and the $H_N$-condition. Most of the content can be found in \cite[Sect.~4.2]{Markfelder} and references therein.

In this section we suppose $M>1$, $\sK\subset \R^M$ a set and $\Lambda\subset \R^M$ a cone.

\subsection{Definition and Properties}

Let us first recall the definition of the $\Lambda$-convex hull.

\begin{defn}[See {\cite[Defn.~4.2.2]{Markfelder}}] \label{defn:app-Lconvex}
	\begin{itemize}
		\item A set $S\subset \R^M$ is called \emph{$\Lambda$-convex} if $[\vp,\vq]\subset S$ for all $\vp,\vq\in S$ with $\vp-\vq\in \Lambda$.
		
		\item The \emph{$\Lambda$-convex hull $\sK^\Lambda$} of $\sK$ is the smallest $\Lambda$-convex set which contains $\sK$.
	\end{itemize}
\end{defn}

The following is a simple observation.

\begin{prop}[See {\cite[Prop.~4.2.3]{Markfelder}}] \label{prop:app-Lconvex}
	\begin{itemize}
		\item Every convex set is $\Lambda$-convex.
		\item $\sK^\Lambda\subset \sK^\co$.
	\end{itemize}
\end{prop}

Next we recall the definition of the $H_N$-condition and the barycenter, see \cite[Sect.~4.2]{Markfelder} and references therein. 

\begin{defn}[See {\cite[Defn.~4.2.4]{Markfelder}}] \label{defn:app-hn} 
	Let $N\in \N$ and $(\tau_i,\vp_i)\in \R^+ \times \R^M$ for $i=1,...,N$. We say that the family of pairs $\big\{(\tau_i,\vp_i)\big\}_{i=1,...,N}$ satisfies the \emph{$H_N$-condition} if the following holds.
	\begin{itemize}
		\item If $N=1$, then $\tau_1=1$.
		\item If $N\geq 2$, then (after relabeling if necessary) $\vp_2-\vp_1 \in\Lambda$ and the family 
		\begin{equation*} 
			\left\{ \left( \tau_1 + \tau_2 , \frac{\tau_1}{\tau_1 + \tau_2} \vp_1 + \frac{\tau_2}{\tau_1 + \tau_2} \vp_2\right) \right\} \cup \big\{(\tau_i,\vp_i)\big\}_{i=3,...,N}
		\end{equation*}
		satisfies the $H_{N-1}$-condition.
	\end{itemize}
\end{defn} 

\begin{defn}[See {\cite[Defn.~4.2.5]{Markfelder}}] \label{defn:app-barycenter}
	Let $N\in \N$ and $(\tau_i,\vp_i)\in \R^+ \times \R^M$ for $i=1,...,N$ with $\sum_{i=1}^N \tau_i =1 $. Then 
	$$
	\vp := \sum_{i=1}^N \tau_i \vp_i
	$$
	is the \emph{barycenter} of the family of pairs $\big\{(\tau_i,\vp_i)\big\}_{i=1,...,N}$.
\end{defn}

The following proposition covers one of the most important properties of the $\Lambda$-convex hull.

\begin{prop}[See {\cite[Prop.~4.2.9]{Markfelder}}] \label{prop:app-laminates}
	It holds that 
	\begin{align*}
		\sK^\Lambda = \bigg\{ \vp\in \R^M\, \Big|\, &\exists N\in \N, \exists (\tau_i,\vp_i)\in \R^+ \times \R^M \text{ for all }i=1,...,N \text{ such that } \\
		& \bullet \ \text{the family } \big\{ (\tau_i,\vp_i)\big\}_{i=1,...,N} \text{ satisfies the $H_N$-condition, } \\
		& \bullet \ \vp_i \in \sK \text{ for all }i=1,...,N \text{ and } \\
		& \bullet \ \vp \text{ is the barycenter of the family }\big\{ (\tau_i,\vp_i)\big\}_{i=1,...,N}, \text{ i.e. } \vp= \sum_{i=1}^N \tau_i \vp_i \bigg\} .
	\end{align*}
\end{prop}

\subsection{Complete Wave Cones}

\begin{defn}[See {\cite[Defn.~4.2.12]{Markfelder}}] \label{defn:app-complete-wc}
	The wave cone $\Lambda$ is called \emph{complete with respect to $\sK$} if $\vp-\vq \in \Lambda$ for all $\vp,\vq\in \sK$.
\end{defn}

In the case where the wave cone is complete, the $\Lambda$-convex hull has the following important property.

\begin{prop}[See {\cite[Cor.~4.2.13]{Markfelder}}] \label{prop:app-complete-wc}
	If $\Lambda$ is complete with respect to $\sK$, then $\sK^\Lambda=\sK^\co$.
\end{prop}

\printbibliography[heading=bibintoc]

\end{document}